\documentclass[a4paper, 11pt]{article}
\usepackage{mymacros}
\usepackage{tikz}
\usepackage{tikz-cd}
\usepackage[normalem]{ulem}

\newcommand{\st}{{\rm st}}

\newcommand{\bmrho}{{\bm\rho}}

\renewcommand{\AA}{\mathbb{A}}

\newcommand{\pexp}{\wp}
\renewcommand{\kl}{\mathfrak{l}}
\newcommand{\ssigma}{{\sigma\sigma}}

\definecolor{grey}{rgb}{0.55, 0.55, 0.55}
\definecolor{gray}{rgb}{0.55, 0.55, 0.55}
\definecolor{darkmagenta}{rgb}{0.55, 0.0, 0.55}
\definecolor{magenta}{rgb}{0.85, 0.0, 0.55}

\newcommand{\Loc}{\operatorname{Loc}}  
\newcommand{\Cov}{\operatorname{Cov}}  

\newcommand{\rg}{{\operatorname{RG}}} 
\newcommand{\pt}{{\operatorname{pt}}}
\newcommand{\bulk}{\varnothing}
\newcommand{\stable}{{(\bs)}}

\newcommand{\WN}{\operatorname{\bf WN}} 
\newcommand{\NGd}{\operatorname{\bf NG}} 
\newcommand{\GF}{\operatorname{\bf GF}} 

\newcommand{\f}{{\sf f}}
\newcommand{\g}{{\sf g}}
\newcommand{\sh}{{\sf h}}
\newcommand{\sm}{{\sf m}}
\renewcommand{\o}{{\sf o}}
\newcommand{\x}{{\sf x}}

\newcommand{\sa}{{\sf a}}
\renewcommand{\sb}{{\sf b}}
\renewcommand{\sc}{{\sf c}}

\newcommand{\ox}{{\sf ox}}
\newcommand{\ba}{\textbf{a}}

\newcommand{\bs}{\textbf{s}}
\newcommand{\rd}{{\rm d}}

\newcommand{\bfb}{\textbf{b}}

\newcommand{\kae}{{\ka}}
\newcommand{\kbe}{{\kb}}
\newcommand{\kpe}{{\kp}}

\newcommand{\Eplus}{\E_+}

\newcommand{\scale}{r}

\newcommand{\thetaz}{\theta_{\zeta}}
\newcommand{\II}{\mathbb{I}}
\newcommand{\nnabla}{{\nabla\nabla}}
\newcommand{\Phipt}{\Phi^{\pt}}
\newcommand{\HB}{\operatorname{HB}}
\newcommand{\xyz}{\mathbb{A}}
\newcommand{\proj}{\operatorname{proj}}

\renewcommand{\hat}[1]{\widehat{#1}}
\renewcommand{\bar}[1]{\overline{#1}}

\usepackage[titles]{tocloft}
\title{Torus scaling limits and the plateau of the critical weakly coupled $|\varphi|^4$ model in $d \ge 4$}

\author{
  Jiwoon Park
  \footnote{E-mail: {\tt jp711@cantab.ac.uk},  orcid: {\tt 0000-0002-1159-2676}}  
}

\makeatletter
\newcommand{\subjclass}[2][1991]{
  \let\@oldtitle\@title
  \gdef\@title{\@oldtitle\footnotetext{#1 \emph{Mathematics subject classification.} #2}}
}

\makeatother

\date{\vspace*{-2em}}

\usepackage[greek,english]{babel}
\def\coppa{{\fontencoding{LGR}\fontfamily{cmr}\selectfont\textqoppa}}
\def\q{\hbox{\foreignlanguage{greek}{\coppa}}}
\def\qq{{\hbox{\foreignlanguage{greek}{\footnotesize\coppa}}}}

\begin{document}
\maketitle

\begin{abstract}
The $n$-component weakly coupled $|\varphi|^4$ model on the $\Z^d$ lattice ($d\ge 4$) exhibits a critical two-point correlation function with an exact polynomial decay in infinite volume, regardless of whether the interaction is short- or long-range. This paper presents a rigorous analysis of the system in both $\Z^d$ and a finite-volume torus. In a torus, we prove the existence of a plateau effect, where the correlation function undergoes a crossover from the polynomial decay to a uniform constant state.

We then establish the precise scaling limit picture that provides a complementary description of this crossover. As immediate consequences, we verify the finite-size scaling limit predicted by Zinn-Justin, the finite-size scaling exponents (qoppas) suggested by Kenna and Berche and the role of the Fourier modes in finite-size scaling suggested by Flores-Sola, Berche, Kenna and Weigel. The proofs use the renormalisation group map constructed in the author's previous work. 
\end{abstract}

\setcounter{tocdepth}{1}
\tableofcontents

\section{Introduction and main results}

This paper is dedicated to the study of finite-size scaling (FSS) of the short- and long-range $|\varphi|^4$ model in Definition~\ref{def:phi4}.
FSS is of critical importance both practically,  in the analysis of laboratory experiments and numerical simulations,  and theoretically,  as it elucidates the role of the underlying geometry in statistical physics systems. 
We investigate two fundamental aspects of FSS in this work: the torus scaling limit and the plateau effect,  explained in detail below.

In a general statistical physics problem,  the complexity of systems means one cannot discuss every physical detail with arbitrary accuracy.  However, the phenomenon of universality,  the emergence of a common structure in a wide range of physical systems under scaling,  makes this a mathematically tractable field of study.  
One of the most well-understood mechanisms of universality is the presence of the Gaussian fixed point under the renormalisation group (RG),  a multi-scale argument addressing fluctuations in different lengths scales independently.  It is often predicted that, when the dimension is sufficiently high and the temperature is at or above the critical point,  high-order cumulants vanish under the flow generated by RG.  
This concept is confirmed in fundamental models, ranging from the Central Limit Theorem \cite{ott2023note} and the Curie-Weiss model \cite{P89gaussian} to the $|\varphi|^4$ model \cite{MR790736,MR882810,BBS5,MR892925}, Bernoulli percolation \cite{H25critical}, and the Self-avoiding walk \cite{MR2827969}.  The smallest dimension where the Gaussian fixed point is present is called the upper critical dimension ($d_{c,u}$).

Among these models, the $|\varphi|^4$ model at the critical point in dimensions at and above $d_{c,u}$ is particularly well-understood. In the physics literature, it serves as a classical pedagogical example for RG computations \cite{PhysRevB.4.3174, ZZ21Q}. Rigorously, many anticipated properties have been verified through seminal mathematical works \cite{MR857063, MR3269689, MR3459163, MR880526, MR643591, MR1219313, MR678000, AD21, HHS08M, CS19critical,  MR892925}.

However, these well-known results often overlook a tension regarding the role of the boundary condition in the scaling limit.  In the scaling limits established by Bauerschdmit, Brydges, and Slade \cite{MR3269689}, the Gaussian scaling limit appears slightly above the critical point under periodic boundary conditions (PBC).  In contrast,  in the work of Aizenman and Duminil-Copin \cite{AD21}, the Gaussian scaling limit appears at the critical point under free boundary condition (FBC).  
The scaling of the field used in \cite{AD21} would not yield a well-defined limit in a finite-volume torus under PBC,  as the moment generating function would blow up. 

A related discrepancy is seen in the susceptibility at the critical point under FBC versus PBC.  While Camia, Jiang, and Newman \cite{CJN21} proved that the susceptibility of the Ising model in dimensions $d \ge 5$ has an asymptotic of $|\Lambda|^{2/d}$ (where $|\Lambda|$ is the volume of the system) as $|\Lambda| \rightarrow \infty$ under FBC,  recent results by Liu, Panis, and Slade \cite{LPS25T} show a much stronger divergence rate of $|\Lambda|^{1/2}$ under PBC.

Conceptually,  this tension can be understood as a competition between complete-graph behaviour (like the model on a complete graph) and Gaussian behaviour near the critical point.  Although the critical exponents at the Gaussian fixed point are described by the complete-graph theory,  the scaling limit itself is described in terms of Gaussian fields.  However, a classical result \cite{MR428998} indicates that non-Gaussian scaling limits appear at the critical point of the Curie-Weiss model (the Ising model on the complete graph). 
Thus, depending on the choice of scaling, either the Gaussian scaling limit or the complete-graph description must fail to fully capture the system at criticality.

Resolutions to many of the issues raised in this introduction are discussed for the case of the hierarchical model \cite{MPS23,park2025boundary} using the RG method.
In this paper,  by adapting the RG method to the Euclidean (usual) $|\varphi|^4$,
we give a partial resolution to the first problem by describing the scaling limit under PBC,  with a suggested plausible full solution referencing \cite{MPS23}.  We provide a complete description of the second problem using a representation of the correlation function that captures both regimes simultaneously (see \eqref{eq:plateauIntro}).  These two results constitute our study of the torus scaling limit and the plateau.

\subsection{Definition of the model}
\label{sec:defimodel}

We consider the $n$-component $|\varphi|^4$ model on a high-dimensional discrete torus. 
For integers $L,N\ge 2$, let $\Lambda_N$ be the $d$-dimensional discrete torus defined by $\Lambda_N =[ - \lfloor \frac{L^N - 1}{2} \rfloor , \lfloor \frac{L^N }{2} \rfloor ]^d$,  equipped with a periodic graph structure. 
We denote $\Delta$ as the Laplacian with periodic boundary condition (PBC). 
Its fractional power,  $(-\Delta)^{1-\eta/2}$ ($\eta >0$),  is defined in detail in Section~\ref{sec:fracLap}.
We will always restrict  $d \ge 4$ and $\eta \in [0,1/2)$ so that the model stays at or above the upper critical dimension.

\begin{definition} \label{def:phi4}
Given $\nu \in \R$ and $g>0$,
the $|\varphi|^4$ model on $\Lambda_N$ (with periodic boundary condition) is the probability measure
\begin{align}
	\P_{g, \nu, N} (d \varphi) = \frac{1}{Z_{g,\nu, N}} e^{-H_N (\varphi)} d \varphi
	,  \qquad \varphi \in (\R^n)^{\Lambda_N}
\end{align}
where $Z_{g,\nu, N}$ is a normalisation constant and $H_{g,\nu,N}$ is the Hamiltonian given by
\begin{align}
	H_{g,\nu,N} (\varphi) &= \frac{1}{2} \big( \varphi,   (-\Delta)^{1-\eta/2}  \varphi \big) + V_{g,\nu,N} (\varphi)  ,   \label{eq:HgnuN} \\
	V_{g,\nu,N} (\varphi) &= \sum_{x \in \Lambda_N}  \frac{1}{2} \nu |\varphi(x)|^2  + \frac{1}{4} g |\varphi (x)|^4 .
	 \label{eq:VgnuN}
\end{align}
Expectation is denoted either $\E_{g,\nu,N}$ or $\langle \cdot \rangle_{g,\nu,N}$.
\end{definition}

The $\eta=0$ case corresponds to the short-range interaction model,  while for $\eta >0$,  we have $(-\Delta)^{1-\eta/2}_{x,y} \asymp |x-y|^{-(d+2-\eta)}$,  corresponding to the long-range interaction model.  
It is also common to consider the general interaction $\sum_{x,y} J_{x,y} \varphi_x \varphi_y$ with $J_{x,y} \asymp |x-y|^{-(d+\tilde{\alpha})}$,  which can be translated as $\tilde{\alpha} = 2-\eta$ under the setting above.

The upper critical dimension for the $|\varphi|^4$ model is $d_{c,u} = 4 - 2(\eta \vee 0) = 2\tilde{\alpha} \wedge 4$ \cite{FMN72C}.
Thus by assuming $d\ge 4$ and $\eta \in [0,1/2)$,  we restrict our analysis to the upper critical dimension and above,  and we anticipate mean-field critical exponents at the critical point in the thermodynamic limit,  as predicted by Fisher,  Ma and Nickel \cite{FMN72C}.
Some more recent accounts are given in \cite{BRRZ17scaling,BPR14crossover}.
In particular,  the correlation function is expected to attain algebraic decay $r^{-(d-2+\eta)} = r^{-(d-\tilde{\alpha})}$ and have a Gaussian process as a scaling limit.  These predictions led to a development of rigorous theories  in the $|\varphi|^4$-model and related probabilistic models,  such as the self-avoiding walk and percolation models \cite{HHS08M,  CS15critical, CS19critical,  panis2023triviality,  H25critical}.
FSS introduces corrections to these pictures in a finite volume.

First,  the torus scaling limit focuses on the regime  $N\rightarrow \infty$ 
while we consider observables in the scale of the $d$-dimensional torus $\T^d = [0,1)^d$.  
We define the scaling by considering a natural injection $\hat{i} : \Lambda_N \rightarrow \T^d$ with $d_{\T^d} ( \hat{i} (x) , \hat{i} (y) ) = L^{-N} d_{\Lambda_N}  (x,y)$ for any $x,y \in \Lambda_N$ and given a function $f :\T^d \rightarrow \R^n$,  let
\begin{align} \label{eq:fN}
	f_N (x) = \frac{1}{|\Lambda_N|} f (  \hat{i} (x) ) .  
\end{align}
Our goal is to study the exact limiting behaviour of
$\sum_{x \in \Lambda_N}  \varphi(x) \cdot f_N (x)$.
As shown in Theorem \ref{thm:NGlimit},  this limiting random variable is proven to follow a non-trivial distribution,  providing a precise description of the fluctuations on the torus.

A second, closely related phenomenon we investigate is the plateau effect. 
This effect describes a crossover in the two-point correlation function,  $G^{\Lambda_N} (x,y)$ on the finite torus $\Lambda_N$.
It arises from the competition between two terms: a polynomially decaying function characteristic of the infinite-volume critical point,  and a constant function reflecting the finite volume.
Specifically, general theory \cite{L24universal} predicts that the correlation function is asymptotically described by
\begin{align}
	G^{\Lambda_N} (x,y) \asymp \frac{1}{ (\dist_2 (x,y) )^{d-2+\eta}} + \frac{1}{|\Lambda_N|^{1-2/d_c}}
	\label{eq:plateauIntro}
\end{align}
where $d_c$ is the upper critical dimension of the short-range model, i.e.
when $\eta=0$.  
A surprising fact is that the exponent determined by $d_c$ appears exactly the same for long-range models ($\eta \neq 0$).
We rigorously prove this result in Theorem \ref{thm:plateauGen} with precise coefficients. 
Also see the paragraphs below the theorem for related results in distinct probabilistic models.

For the short-range model ($\eta=0$), the focus of the probability and mathematical physics communities has often been on the macroscopic scaling.  In this regime, observables are studied on the scale of $\R^d$, and the resulting limiting behaviour is insensitive to the system volume and the choice of boundary conditions.  This research area includes extensive work on: the computation of critical exponents \cite{MR857063,  MR3269689,  MR892925} including the analysis of correlation functions \cite{MR3459163,  MR880526,  DP25N,  BHH19,  Saka14}; and the study of scaling limits \cite{MR643591,  MR1219313,  MR678000,  AD21},  verifying the mean-field predictions.  
In contrast, the present work focuses on the FSS regime,  where the system's finite volume and geometry are crucial,  particularly as they lead to phenomena like the plateau effect and the torus scaling limit.

The literature concerning the long-range model has followed a closely related, yet distinct, line of development,  originating with the conjectures of \cite{FMN72C} built on a non-rigorous RG argument,  which also observed the lowering of the upper critical dimension $d_{c,u} = 4-2 (\eta \vee 0) = 2\tilde{\alpha} \wedge 4$,  thus the appearance of mean-field exponents at the critical point in $d\ge d_{c,u}$.
Many of the predicted critical phenomena have been rigorously resolved using a variety of powerful techniques, including the lace expansion \cite{HHS08M,  CS15critical, CS19critical},  and the random current representation \cite{MR857063,  MR958462,  panis2023triviality}.
The study of long-range interactions has also been extended to various other models in probability theory,  such as the self-avoiding walk \cite{CS19critical} and percolation \cite{H24P, H25critical, H24critical}, establishing it as a field of rich,  ongoing development.
There are also developments of the RG method for long-range model in \cite{MR3723429,MR3772040},  where the critical exponents are computed below the critical dimension,  
but these developments lack a systematic application of the RG method at or above the upper critical dimension.  

Our work specifically addresses this gap by employing a rigorous RG map to analyse the long-range $|\varphi|^4$ model.

We also make remarks about the restrictions $g \ll 1$ and $\eta \in [0,1/2)$.  
While the $|\varphi|^4$ model is structurally similar to the $O(n)$-spin model with a smooth potential,  certain robust non-perturbative methods, such as correlation inequalities and the random current representation,  have limitations. Specifically,  they often lack the precision required for quantitative analysis and are less effective when applied to multi-component models (where $n > 2$).
The weak coupling condition is essential when employing powerful quantitative analytic techniques like the RG or the lace expansion.  In the context of this paper, the assumption $g \ll 1$ is not merely a convenience,  but it is critical for our rigorous approach.  It is required both for the construction of the RG map itself \cite{FSmap} and for the subsequent construction of the stable manifold presented in Section \ref{sec:cotcpfdf}, which is necessary to define the critical point.
Also, we predict all our results will hold for $\eta \in [0,2)$, but we do not pursue this generality because it would require significant modifications to the RG map construction \cite{FSmap}.

\subsection{Limiting scales and measures}
\label{sec:limscalesandmeas}

There are three distinct scales for the field fluctuation in different scaling regimes.  We define them as
\begin{align} \label{eq:bNcN}
	\sa_N = L^{-dN/2} , \quad\; \sb_N = \begin{cases}
		N^{1/4} L^{-N} & (d=d_{c,u}) \\
		g^{-1/4} L^{-dN/4} & (d > d_{c,u}) ,
	\end{cases}
	\quad\;
	\sc_N = L^{- \frac{d-2 + \eta}{2} N} .
\end{align}

Also,  we define measures that appear as the scaling limit of the spin field on the torus. 
Averaged fields of $\varphi \in (\R^n)^{\Lambda_N}$ and $\psi \in L^1 (\T^d ; \R^n)$ are
\begin{align}
	\Phi_N (\varphi) = \frac{1}{|\Lambda_N|} \sum_{x \in \Lambda_N} \varphi_x
	,
	\qquad \Phi (\psi) = \int_{\T^d} \psi (x) \rd x .
\end{align}
We consider the White noise(WN) measure with mass $s >0$,
the massless (fractional) Gaussian field and a non-Gaussian measure on the field $\psi$ given by moment generating functions
\begin{align}
	\WN_s ( e^{(\psi, f)} ) &= \exp \big( (f, f) /2s \big) \label{eq:WNs}  \\
	\GF_{\eta} ( e^{(\psi, f)} ) &= \exp \big( (f,  (-\Delta)^{-1+\eta/2} f) / 2 \big) \label{eq:GFFba} \\	
	\NGd (e^{(\psi, f)}) &= \frac{\int_{\R^n}  e^{- \frac{1}{4} |x|^4 } e^{\Phi(f) \cdot x } \rd x }{\int_{\R^n}  e^{- \frac{1}{4} |x|^4 } \rd x } \label{eq:NGdg}
\end{align}
for $f \in \cS ( \T^d  ; \R^n )$.
One can check that \eqref{eq:WNs} and \eqref{eq:NGdg} give probability measures on Schwartz distributions and \eqref{eq:GFFba} gives a probability measure on Schwartz distributions quotiented by constant functions.
These measure appear in the scaling limits of Section~\ref{sec:eslatcp}.

Next lemma clarifies the nature of $\NGd$.  It shows that $\psi \sim \NGd$ is essentially a $Y$-multiple of the unit function,  where $Y$ is an $\R^n$-valued random variable given by 
\begin{align}
	\E [ e^{t \cdot Y} ] = \int_{\R^n}  e^{- \frac{1}{4} |x|^4 } e^{t \cdot x} \rd x \Big/ \int_{\R^d}  e^{- \frac{1}{4} |x|^4 } \rd x  \label{eq:Ydistn} .
\end{align}
Moments of $Y$ appear universally,  as it characterises the magnitude of the plateau,  see Theorem~\ref{thm:plateauGen}.

\begin{lemma} \label{lemma:NGcnst}
Let $Y$ be a $\R^n$-valued random variable given by \eqref{eq:Ydistn}.
Then $Y \one =^d \NGd$.
\end{lemma}
\begin{proof}
For $f \in \cS (\T^d ; \R^n)$,
\begin{align}
	\E [ e^{(Y\one,  f)} ] = \E [ e^{Y \cdot \Phi(f)} ] = \NGd (e^{(\psi, f)}) .
\end{align}
\end{proof}

\subsection{Two-point function and plateau}
\label{sec:tpfap}

We first start with an exact asymptotic on the decay of the two-point function at the critical point.
Lattice points $\o,\x$ are considered,  where $\o$ is the origin and $|\x|= \dist_2 (\o,\x)$,  where $\dist_p$ is the $\ell^p$-distance. 
We introduce a few notations.  Let $\gamma(d,\eta) = 2^{-2+\eta} \pi^{-d/2}  \frac{\Gamma(\frac{d + \eta}{2} - 1)}{\Gamma (1-\eta/2)}$ and $\varphi_x^{(i)}$ be the $i^{\rm th}$ component of $\varphi_x \in \R^n$. 
Throughout the article,  we use $f(x) \sim g(x)$ to denote $\lim_{x} f(x) / g(x) =1$.

\begin{theorem}  \label{thm:infvol2pt}
Let $d \ge 4$,  $\eta \in [0,1/2)$,  $L$ be sufficiently large and $g>0$ be sufficiently small.
Then there exists $\nu_c  \equiv \nu_c (g) = O(g)$ such that
\begin{align}
	\lim_{N\rightarrow \infty} \langle \varphi_\o^{(1)} \varphi_\x^{(1)} \rangle_{g,\nu_c,  N} =: \langle \varphi_\o^{(1)} \varphi_\x^{(1)} \rangle_{g,\nu_c,  \Z^d} =: \C_\x
\end{align}
exists and satisfies 
\begin{align}
	\C_\x \sim c_1 |\x|^{-(d-2+\eta)} \quad \text{as} \quad |\x| \rightarrow \infty \label{eq:infvol2pt}
\end{align}
for some constant $c_1$ such that $c_1 = \gamma(d,\eta) + O(g)$ when $\eta = 0$ and $c_1 = \gamma(d,\eta)$ when $\eta \neq 0$.
\end{theorem}

\begin{remark}
\textbf{(i)} Although the theorem is known for the long-range Ising model  \cite{CS15critical},  our results extend this result to the general $n$-component $|\varphi|^4$ models.
Thus we extend the parameter regime where \cite{FMN72C} is verified.
For the short-range model,  the decay was proved in \cite{MR3459163,  BHH19} when $(d,n) \in \{4\} \times \mathbb{Z}_{\ge 1}$ and $(d,n) \in \Z_{\ge 5} \times \{1,2\}$,  we extend the result to general $d\ge 4$ and $n \in \Z_{\ge 1}$.

By Lemma~\ref{lemma:fracLapGreen},  $\gamma$ is a constant such that $(-\Delta)^{-1+\eta/2}(\o,\x) \sim \gamma |\x|^{-(d-2+\eta)}$.  Consequently, at the critical point when $\eta\neq 0$,  the $|\varphi|^4$ potential is asymptotically invisible in the infinite-volume two-point correlation function.  This finding is consistent with the established result for the Ising model \cite[Theorem 1.2]{CS15critical}.
\smallskip

\noindent\textbf{(ii)} By the theorem,  we have the infinite-volume susceptibility $\sum_{\x \in \Z^d} \langle \varphi_\o^{(1)} \varphi_\x^{(1)} \rangle_{g,\nu_c, \Z^d} = \infty$. 
For $\eta=0$,  reflecting on the $n=1,2$ cases,  it is natural to expect that the polynomial decay rate of the correlation function uniquely characterises the critical point (due to Simon-Lieb inequality \cite{MR589427},  also see \cite{DP25N}).
However,  since the Griffiths inequality is not available when $n\ge 3$,  the theorem does not imply the uniqueness of the point satisfying \eqref{eq:infvol2pt}.
It will become more apparent in Corollary~\ref{cor:WNlimit} that $\nu_c$ is the critical point,  also see Remark~\ref{remark:nuc}.
\end{remark}

Next,  we prove the the plateau effect in FSS,  where $Y$ is as in Lemma~\ref{lemma:NGcnst}.
In the theorem,  $\o, \x$ are now points in the discrete torus $\Lambda_N$,  but we still consider the same $\mathbb{C}_\x$ given in Theorem~\ref{thm:infvol2pt},  that we define via the natural imbedding $\Lambda_N \rightarrow \Z^d$. 

\begin{theorem} \label{thm:plateauGen}
Under the assumptions of Theorem~\ref{thm:infvol2pt},
there exists $c_2 > 0$ such that
\begin{align}
	\langle \varphi_\o^{(1)} \varphi_\x^{(1)} \rangle_{g,\nu_c, N} = \big( \mathbb{C}_\x + c_2 \frac{\E[|Y|^2]}{n} \sb_N^2 \big) (1 +o(1)) .
\end{align}
where $o(1)$ is a function that tends to 0 as $|\x| \rightarrow \infty$,  uniformly in $N$.
Moreover,  $c_2 = \sqrt{n+8} /4\pi$ when $d= d_{c,u}$ and $c_2 = 1+ O(g)$ when $d \neq d_{c,u}$.
\end{theorem}

The plateau effect is a robust phenomenon observed across numerous distinct statistical physics models. For nearest-neighbor interactions, it has been rigorously established in: the Ising and one-component $\varphi^4$ models for dimensions $d>4$ \cite{LPS25T}; the self-avoiding walk for $d>4$ \cite{S23N, L25G}; branched polymers for $d>8$ \cite{L25N}; and percolation for $d>6$ \cite{HMS23}.  In contrast, the analysis of the long-range models with respect to the plateau effect is relatively limited,  with some examples including \cite{Liu25, H24P}.

Crucially,  the plateau effect remains largely unproven exactly at the upper critical dimension $d_{c,u}$,  with the hierarchical $|\varphi|^4$ model \cite{park2025boundary} being an exception.  Furthermore,  many existing results only prove the effect in a massive,  asymptotic regime slightly above the critical point ($0<\nu - \nu_c \ll 1$).
Also,  while it is believed that the correlation function should exhibit a plateau effect within a full critical window around the critical point $\nu = \nu_c$,  this rigorous proof has only been achieved for a few models \cite{MPS23, park2025boundary}.  For related open problems, see Section \ref{sec:openprobls}.

\subsection{Near-critical point limits}
\label{sec:nearcplim}

We first state a scaling limit on a sequence of $\nu$ that converges to $\nu_c$.  This result is a generalisation of \cite[Theorem~1.3.(i)]{MR3269689} to $(d\ge 4,  \; \eta >0)$ and $(d\ge 5, \; \eta \ge 0)$.

\begin{theorem}
\label{thm:WNlimit}
Let $\eta \in [0,1/2)$,  $d \ge 4$,  $L$ sufficiently large and $g>0$ be sufficiently small.
There exist some sequences $(\epsilon_k ,  \epsilon'_k )_{k \ge 0}$ such that $\epsilon_k \rightarrow 0$,  $\epsilon'_k \downarrow 0$ as $k\rightarrow \infty$ and 
for any $f \in \cS (\T^d ; \R^n)$,
\begin{align} 
	\label{eq:WNlimit}
	\lim_{N\rightarrow \infty} \big\langle  e^{(\varphi, \,  f_N) /\sa_N} \big\rangle_{g,\nu_c + \epsilon_k,N}  
		=
		\WN_{\epsilon'_k} ( e^{(\psi,f)}  )  .
\end{align}

\end{theorem}

The theorem is a strong indication that $\nu_c$ is the critical point of the $|\varphi|^4$ model.  To see this more clearly,  we can compute the susceptibility.

\begin{corollary}
\label{cor:WNlimit}

Under the conditions of Theorem~\ref{thm:WNlimit},
for any $p \in \Z_{\ge 1}$ and $\lambda \in \R^n$
\begin{align}
	\lim_{N\rightarrow \infty} \frac{1}{\sa_N^{2p}} \left\langle | \lambda \cdot \Phi_N (\varphi)|^{2p} \right\rangle_{g,\nu_c + \epsilon_k , N} 
		= \frac{(2p) !}{p !} \frac{1}{(2 \epsilon'_k)^p}  .
\end{align}
\end{corollary}
\begin{proof}
For $\lambda \in \R^n$,  we can take $f = \lambda \one$ in \eqref{eq:WNlimit}, 
use \eqref{eq:WNs} and expand both sides to get
\begin{align}
	\sum_{n=0}^{\infty} \frac{1}{n !} \langle (\lambda \cdot \Phi_N (\varphi))^n \sa_N^{-n}  \rangle_{g,\nu_c + \epsilon_k,N} 
		\sim \exp\left( \frac{|\lambda|^2}{2 \epsilon'_{k}} (\one, \one) \right) = \sum_{p=0}^{\infty} \frac{|\lambda|^{2p}}{p! (2 \epsilon'_k)^{p}} 
\end{align}
as $N\rightarrow \infty$.
We get the desired conclusion by comparing the coefficients of $\lambda$.
\end{proof}

If we define the susceptibility as
\begin{align}
	\chi_{g, \nu,N} = L^{dN} \left\langle |\Phi_N (\varphi)|^2 \right\rangle_{g, \nu,N}
	,
	\qquad
	\chi_{g, \nu,\infty} = \limsup_{N\rightarrow \infty} \chi_{g, \nu,N} 
	, \label{eq:chidefn}
\end{align}
it follows immediately that
\begin{align}
	\chi_{g,\nu_c + \epsilon_k ,\infty } = n ( \epsilon'_k )^{-1} \rightarrow \infty \quad \text{as} \quad \epsilon_k \rightarrow 0 
\end{align}
so $\nu_c$ is a cluster point of a sequence along which the susceptibility diverges.  This matches with the usual sense of the critical point.
That $\chi_{g,\nu_c,\infty} = \infty$ can be deduced from the torus scaling limit,  see Corollary~\ref{cor:NGlimit}.

\begin{remark} \label{remark:nuc}
\textbf{(i)} These statements do not guarantee the uniqueness of $\nu_c$ with the same property,  or even the sign of $\epsilon_k$.
But for $n =1,2$,  the Ginibre inequality \cite{MR0269252} implies that $\chi_{g,\nu,N}$ is increasing in $\nu$,  so the theorem implies
\begin{align}
	\nu_c = \inf\{ \nu \in \R : \chi_{g,\nu,\infty} < \infty \}
\end{align}
and we should have $\epsilon_k > 0$.  In that case,  we can also replace the sequence $(\epsilon_k )_{k\ge 0}$ by any sequence approaching 0 from above. 

For the general case,  we expect that the same would hold and can be proved using the method of \cite{MPS23}.
However,  we leave it open in this paper because it requires careful analysis of the dynamical system generated by the RG flow. 

\smallskip\noindent\textbf{(ii)}
When $(d,\eta)=(4,0)$,  a statement similar to Theorem~\ref{thm:WNlimit} was proved in \cite[Theorem~1.3(i)]{BBSphi4}.  The main difference is that $\epsilon_k \downarrow 0$ and the asymptotic of $\epsilon'_k$ was computed as a function of $\epsilon_k$ there.  This is not in the scope of this article,  but we expect that a similar analysis can be applied to our setting.

\smallskip\noindent\textbf{(iii)} 
The construction of the critical point via the RG presents distinct difficulties depending on the dimension.  For $(d,\eta)=(4,0)$  \cite{BBSphi4}, 
the difficulty stems from the hyperbolic nature of the quartic term coefficient in the RG flow.
In contrast,  for dimensions above the upper critical dimension ($d>d_{c,u}$), the quartic term coefficient is elliptic,  which is conceptually simpler for construction.  However, this ease is counterbalanced by the need to tune many more initial variables.

A reader familiar with RG theory might wonder why tuning these additional variables is difficult,  given that most of them are irrelevant when $d>d_{c,u}$.  The difficulty arises because the quartic coupling constant $g$ is a `dangerous irrelevant' in the context of FSS.  This means that despite its irrelevance in the infinite-volume limit,  its flow cannot be neglected when analysing finite systems.  Consequently,  our critical point construction demands a stronger stability property.  
This requirement makes the RG analysis of the torus scaling limit non-trivial in $d>d_{c,u}$,  differentiating our approach from previous rigorous RG formulations of the $|\varphi|^4$ model, such as those established at $d= d_{c,u}$ \cite{MR790736,  MR3269689,  MR882810}.
\end{remark}

\subsection{Torus scaling limit at the critical point}
\label{sec:eslatcp}

Torus scaling limit is obtained when the macroscopic system is fixed as a torus with finite size.  
Recall $\sb_N$ and $\sc_N$ from \eqref{eq:bNcN}.

\begin{theorem} 
\label{thm:NGlimit}
Under the setting of Theorem~\ref{thm:infvol2pt},  let $f \in \cS (\T^d ; \R^n)$ and $c_1$ and $c_2$ be as in Theorem~\ref{thm:infvol2pt} and \ref{thm:plateauGen}.
\begin{enumerate}
\item For $c_3 (d,\eta) = c_2 (d,\eta)^{-2}$,
\begin{align}
	\lim_{N\rightarrow \infty} \langle  e^{(\varphi,f_N) / \sb_N} \rangle_{g,\nu_c ,N}  
		= \NGd \left( e^{(\psi, f) / c_3^{1/4} }  \right) .
\end{align}

\item For $c_4 (d,\eta) = \gamma(d,\eta) /  c_1 (d,\eta)$,
\begin{align}
	\lim_{N\rightarrow \infty} \langle  e^{(\varphi,f_N - \Phi_N (f_N)) / \sc_N} \rangle_{g,\nu_c ,N}  
		=	
		\GF_{\eta} \big( e^{(\psi,  f - \Phi(f) ) / c_4^{1/2} }  \big) .
\end{align}
\end{enumerate}
\end{theorem}

The presence of the non-Gaussian limit in (i) can be made contrast with the macroscopic scaling limit,  where the critical $|\varphi|^4$ model attains a Gaussian limit,  as proved in \cite{AD21, panis2023triviality} for $n=1$.

Under the scaling of (i),  based on Lemma 1.2,  reflects a complete coherence of the field $\varphi$ with a constant amplitude.  This same constant-amplitude limit is also attained by the complete-graph model, and it is directly related to the classical problem of determining the distribution of the Curie-Weiss model at the critical point \cite{MR428998, CS11nonnormal, DM23FMFIM} (see \cite{CGK14anomalous} for a physics perspective).  The Gaussian scaling of (ii) is also found in the complete-graph model \cite{P89gaussian},  but because the Gaussian field is spatially uncorrelated, it fails to capture the crucial geometry of the torus observed in the (fractional) Gaussian field inside (ii).

Since $\sb_N \gg \sc_N$,  we can observe a scale hierarchy by comparing (i) and (ii).  By Lemma~\ref{lemma:NGcnst},  part (i) indicates that $\varphi / \sb_N$ converges to a constant-valued function,  while part (ii) indicates that $(\varphi - \Phi_N (\varphi) ) / \sc_N$ tends to a (fractional) Gaussian field.  Informally,  we have an expansion
\begin{align}
	\varphi \,  \text{``}\sim\text{''} \,   \sb_N Y \one / c_3^{1/4} +  \sc_N \varphi_{\GF_{\eta}} / c_4^{1/2}  \quad \text{as} \quad N\rightarrow \infty 
	\label{eq:FSSexpansion}
\end{align}
where $Y$ follows distribution \eqref{eq:Ydistn} (with $\gamma$ as appropriate).
Although $\sb_N \gg \sc_N$,  the massless $(2-\eta)$-stable process is highly irregular,  so by zooming into a point on the torus sufficiently faster than the rate of $N\rightarrow \infty$,  we can expect that the dominant fluctuation is reversed.  This is precisely what is expected to be observed for the macroscopic scaling limit,  giving results equivalent to \cite{AD21}.
The phenomenon is also reflected in the plateau effect described in Theorem~\ref{thm:plateauGen}.
In these contexts,  the FSS should be understood as an interpolation of the complete-graph model and the Gaussian process.

\subsubsection{Role of Fourier modes}

In physics literature,  the distinction between the two scaling regimes is expressed in terms of a dichotomy of the $\vec{k}$-susceptibilities.  Namely,  if we let
\begin{align}
	\mathbb{K} = \big\{ \vec{k} = (k_1, \cdots, k_d) \in \big( 2\pi \Z_{\ge 0} \big)^d \big\} , \qquad 
	\textbf{s} (x_i ; k_i) = \begin{cases}
		\sqrt{2} \sin ( k_i x_i ) & (k_i \neq 0) \\
		1 & (k_i = 0)
	\end{cases}
\end{align}
for $x_i \in \T^1$ and let
\begin{align}	
	E^{( \vec{k})} (x) = \prod_{i=1}^d \textbf{s} (x_i ; k_i), \qquad 
		x = (x_1, \cdots, x_d) \in \T^d
\end{align}
so that they are orthonormal eigenvectors of the Laplacian,  we can define the $\vec{k}$-susceptibility as
\begin{align}
	\chi_{g,\nu, N}^{(\vec{k})} = |\Lambda_N | \big\langle (\varphi ,  E^{(\vec{k})}_N )^2 \big\rangle_{g, \nu, N}
\end{align}
where $E^{(\vec{k})}_N$ is the discretisation of $E^{(\vec{k})}$ given by \eqref{eq:fN}.  
In the terminology of \cite{WY14,F-SBKW16},
the $0$-susceptibility follows a ``non-standard FSS'' while for $\vec{k} \neq 0$,  $\vec{k}$-susceptibility follows a ``standard FSS'',  as the following result says.
This verifies the susceptibility rows of PBC part of \cite[Table I]{F-SBKW16} (where we translate $\sigma  =2-\eta$ and we also obtain the logarithmic correction).

\begin{corollary} \label{cor:roleofFouriermodes}
Under the conditions of Theorem~\ref{thm:NGlimit},  for $\vec{k} \in \mathbb{K}$,  as $N\rightarrow \infty$,
\begin{align}
	(\vec{k}=0) \qquad 
		& \chi_{g,\nu_c, N}^{(0)} \sim c_3^{-1/2} \E \big[ |Y|^{2} \big] \times \begin{cases}
		N^{1/2} L^{2 N} & (d = d_{c,u}) \\
		L^{\frac{d}{2} N} & (d > d_{c,u}) 
	\end{cases} \\
	(\vec{k} \neq 0) \qquad 
		& \chi_{g,\nu_c, N}^{(\vec{k})} \sim n c_4^{-1} |\vec{k}|^{-2+\eta} L^{(2-\eta) N} .
\end{align}
\end{corollary}
\begin{proof}
The proof of the first statement is exactly as in Corollary~\ref{cor:WNlimit}: we take $f = \lambda \one$ in Theorem~\ref{thm:NGlimit},  expand both sides in $\lambda$ and use
\begin{align}
	\E \big[ | \lambda \cdot  Y|^{2p} \big] = \int_{\R^n} | \lambda \cdot x |^{2p}  e^{- \frac{1}{4} |x|^4 } \rd x \Big/ \int_{\R^d}  e^{- \frac{1}{4} |x|^4 } \rd x .
	\label{eq:roleofFouriermodes}
\end{align}
The second follows directly from Theorem~\ref{thm:NGlimit}(ii) by plugging in $f = E^{(\vec{k})}$.
\end{proof}

Since the 0-mode susceptibility grows much faster than the other $\vec{k}$-modes in \eqref{eq:FSSexpansion},  FSS is said to violate the hyperscaling relation in certain contexts \cite{K12universal, berche2012hyperscaling,  LB96finite}.  A parallel line of (non-rigorous) theories have been developed from the perspective of scaling relations,  introducing $\q$ and $\hat{\qq}$ exponents.  We do not look into the details of this theory,  but a comparison with their results in Corollary~\ref{cor:NGlimit} would be valuable.

\subsubsection{$\q$ and $\hat\qq$ exponents}

The finite-size susceptibility can also be obtained.  
The next result verifies the FSS (qoppa-)exponents $\q = \frac{d}{d_{c,u}} = \frac{d}{4-2\eta}$ (for $\eta \in [0,1/2)$) and $\hat\qq=\frac{1}{4}$ (for $\eta=0$) in the context of \cite{K12universal, K14fisher}.
(In \cite{K14fisher},  it is suggested for the $|\varphi|^4$ model,  that $\chi_{g,\nu_c, N} \asymp |\Lambda|^{\gamma\qq / d\nu}$ when $d > d_{c,u}$ with $\gamma = 1$ and $\nu= \frac{1}{2-\eta} $ and $\chi_{g,\nu_c, N} \asymp (\log |\Lambda|)^{\hat{\gamma}\hat\qq / \hat{\nu} } |  \Lambda|^{\gamma\qq / d\nu}$ when $d=d_{c,u}$ with $\hat{\gamma}=\frac{n+2}{n+8}$ and $\hat{\nu}=\frac{n+2}{(2-\eta)(n+8)}$.)

\begin{corollary} \label{cor:NGlimit}
Under the conditions of Theorem~\ref{thm:NGlimit},  for any $p \in \Z_{\ge 1}$ and $\lambda \in \R^n$ 
\begin{align}
	\langle | \lambda \cdot \Phi_N (\varphi)|^{2p} \rangle_{g,\nu_c, N}
		\sim c_3^{-\frac{p}{2}} \,  \sb_N^{2p}  \,  \E \big[ | \lambda \cdot  Y|^{2p} \big]
		\qquad \text{as } N\rightarrow \infty
		\label{eq:corNGlimit1}
\end{align}
for $Y$ defined as in \eqref{eq:Ydistn}.  In particular,  as $N\rightarrow \infty$,
\begin{align}
	\chi_{g,\nu_c,N} \sim c_3^{-1/2} \E \big[ |Y|^{2} \big] \times \begin{cases}
		N^{1/2} L^{2 N}  \propto ( \log |\Lambda| )^{1/2} |\Lambda|^{1/2} & (d = d_{c,u}) \\
		L^{\frac{d}{2} N}  \propto |\Lambda |^{1/2} & (d > d_{c,u})  .
	\end{cases}
		\label{eq:corNGlimit2}
\end{align}
\end{corollary}
\begin{proof}
The proof follows directly from the expansion \eqref{eq:roleofFouriermodes}.
\end{proof}

\subsection{Open problems and future directions}
\label{sec:openprobls}

While the present article establishes a precise and rigorous picture of the FSS limits,  our results also confirm and contextualise a broader set of predictions from the physics literature \cite{F-SBKW16,  WY14,  LB97}. 
These broader theoretical predictions, which offer a more complete understanding of FSS above the critical dimension,  have been partially verified in specific rigorous settings,  notably for the hierarchical $|\varphi|^4$ model (with $\eta=0$) \cite{MPS23, park2025boundary}. 
For a comprehensive overview of related rigorous results and predictions, we refer the reader to the summary in \cite{L24universal}.  We collect several conjectures that we anticipate will be meaningful for a complete understanding of the FSS behaviour above the upper critical dimension.

\medskip\noindent\textbf{(1) Critical window:}
The critical window is the $\Lambda_N$-dependent range of $\nu$ where the finite-size susceptibility stays in the scale of $\chi_{g, \nu_c, N}$.
More precisely,  in \cite{MPS23} (in the hierarchical setting with $\eta =0$),  it was defined
$w_N \propto N^{-\hat{\theta}} L^{-2N}$ ($\hat{\theta} = \frac{4-n}{2(n+8)}$) for $d=4$ and $w_N = L^{-Nd/2}$ for $d >4$.  Then we have the susceptiblity \emph{scaling profile}
\begin{align}
	\chi_{g,\nu_c + s w_N,  N} \sim \chi_{g,\nu_c, N} \frac{\int_{\R^n} |y|^2 e^{-\frac{1}{4} |y|^4 -  c s |y|^2 }  \rd y}{\int_{\R^n} |y|^2 e^{-\frac{1}{4} |y|^4}  \rd y}  ,  \qquad  s \in \R
	\label{eq:critwindowprof}
\end{align}
for some constant $c >0$.  The same profile was predicted by \cite[Section~32]{ZZ21Q}.
The same would happen in our setting,  but we expect a stronger statement. 

\begin{conjecture} \label{conj:1}
Let $\NGd_{s}$ be a probability measure given by moment generating function
\begin{align}
	\NGd_s [e^{(\psi, f)} ] \propto \int_{\R} e^{(\psi,  \Phi(f) )} e^{-\frac{1}{4} |y|^4 -  s |y|^2} \rd y .
\end{align}
Under the assumptions of Theorem~\ref{thm:NGlimit},  there exists $c>0$ such that
\begin{align}
	\lim_{N\rightarrow \infty}  \langle e^{(\varphi,f_N) / \sb_N} \rangle_{g,\nu_c + sw_N, N} = \NGd_{cs} ( e^{(\psi, f) / c_3^{1/4}}  ) .
\end{align}
\end{conjecture}

In the RG framework,  the solution would amount to differentiating the RG dynamics constructed in Section~\ref{sec:cotcpfdf} and a simpler version of this computation is shown in \cite{MPS23}.

\medskip\noindent\textbf{(2) Boundary condition:}
The FSS picture under free boundary conditions (FBC) differs notably from the periodic setting we analyse. 
For instance, \cite{CJN21} demonstrates that a completely distinct FSS applies at the critical point under FBC.
However,  \cite{MPS23} shows that (again under hierarchical setting with $\eta=0$) the standard scaling profile can be recovered at a shifted, volume-dependent pseudocritical point.
Specifically,  if we let $v_N = L^{-2N}$ for $d > 4$ and $v_N = N^{\frac{n+2}{n+8}} L^{-2N}$ for $d = d_{c,u}$,  there exists $c > 0$ such that
\begin{align}
	\chi_{g, \nu_c - cv_N + sw_N , N}^{\rm F} \sim \chi_{g, \nu_c + sw_N, N}
	\label{eq:pseudocritical}
\end{align}
where superscript ${\rm F}$ indicates FBC. 
As predicted in \cite{F-SBKW16},  the same should happen under the Euclidean setting,  but now with $v_N = L^{-(2-\eta) N}$ for $d > d_{c,u}$ and different constants.

\begin{conjecture}
For the $|\varphi|^4$ model with FBC,  there exist $C, c,  c' >0$ such that
\begin{align}
	\chi_{g, \nu_c - cv_N + sw_N , N}^{\rm F} \sim  C \chi_{g, \nu_c + c' sw_N, N} .
\end{align}
\end{conjecture}

We expect the same shift of critical point should be observed for the profile of Conjecture~\ref{conj:1}.
Unlike the scaling window,  running a rigorous RG under FBC would require significant modification to the method.

\subsection{Outline of the method}

The proof relies on two steps: the construction of the critical point for the $|\varphi|^4$ model's RG flow (based on the map from \cite{FSmap}) and the derivation of asymptotic estimates for the 0-mode fluctuation integrals. Stability is a prerequisite for the latter,  where we also employ functional inequalities and Fourier analysis.  By combining the RG flow with these estimates,  we rigorously determine the moment generating function and the two-point function.
In below,  we explain this paragraph in more detail.

The initial formulation of the $|\varphi|^4$  measure naturally splits the Hamiltonian $H_{g,\nu,N}$ \eqref{eq:HgnuN} into a quadratic interaction part and the potential $V_{g,\nu,N}$
This allows the measure to be viewed as a Gaussian integral with covariance $(-\Delta)^{-1+\eta/2}$,  modified by the non-Gaussian potential. 
However, simply using this natural splitting does not automatically generate a stable RG flow,  as the critical parameters are not initially tuned. 
To achieve stability and rigorously define the critical point, we must add and subtract quadratic counterterms to the potential function $V_{g,\nu,N}$ and the quartic interaction.
These counterterms ensure that the flow remains controlled.  This essential reformulation, including the most general form of the required counterterms,  is detailed in Section \ref{sec:gssintsundeffpots}.

We now decompose the modified Gaussian integral into two distinct fluctuation integrals. The first integral governs the local fluctuations, which are asymptotically described by a $(2-\eta)$-stable Gaussian process. The second integral captures the global fluctuations across the scale of the finite torus.  We apply the RG method to rigorously control the first (local) fluctuation.  Effectively,  the second (global) fluctuation is approximated by the 0-mode component (in Fourier space) of the Gaussian field on the torus. This fluctuation decomposition is formally introduced in Section~\ref{sec:FRDdefi} and the rigorous proof of its properties is isolated in Section \ref{sec:covdcmp}.  The subsequent analysis follows this structure. The control of the local fluctuation is detailed in Sections~\ref{sec:polacts}--\ref{sec:cotcpfdf},  while 0-mode fluctuation is treated in Section \ref{sec:mainresults}.  This approach is a standard technique in the physics literature,  as outlined in \cite[Section 32]{ZZ21Q}.

The RG analysis for the locally $(2-\eta)$-stable Gaussian fluctuation field relies on the RG map constructed in \cite{FSmap}. To adapt this framework,  we first introduce the RG notations in Sections \ref{sec:polacts} and \ref{sec:RGmap}.  Since \cite{FSmap} only demonstrated the existence of the RG map for a single step,  a major component of our work is to prove stability of the dynamical system generated by the RG maps in Section~\ref{sec:cotcpfdf} under a specific initial condition.
We simultaneously construct the critical point for the $|\varphi|^4$ model.

As mentioned,  Section \ref{sec:mainresults} is dedicated to the analysis of the 0-mode fluctuation integral,  which completes the proof of Theorem~\ref{thm:WNlimit} and \ref{thm:NGlimit}. This integral simplifies to an $n$-dimensional integral,  and a simple asymptotic analysis along with some functional bounds forms the bulk of this section. 
This relies on Fourier analysis built in Appendix~\ref{sec:fa} and \ref{sec:covcomp}.

Sections \ref{sec:stabobsflow} and \ref{sec:plateauproof} revisit the concepts from the preceding analysis,   but with a focus on extending the RG flow to include the relevant observables. 
These sections utilise the established stability and asymptotic analysis to  prove Theorem \ref{thm:infvol2pt} and \ref{thm:plateauGen}.
Since the two-point function is a more singular observable compared to the  scaling limits,  it necessitates a more detailed analysis.

Aside from extending the parameter regimes of $d$ and $\eta$, the main technical difference between the method of this paper and \cite{MR3269689} is the utilization of the decay estimate Lemma~\ref{lemma:KNbnd}. This idea was already displayed in \cite{MPS23} to obtain the FSS profile in the hierarchical model. 
Compared to the hierarchical model,  both the perturabative and the full RG map \cite{FSmap} are considerably more intricate,  thus we leave some of the questions resolved for the hierarchical case open for later research.

\subsection{Notation}
\label{sec:notation}

Let $\hat{e}_+ = \{e_1, \cdots, e_d\}$ for the standard basis $e_1, \cdots, e_d$ of $\Z^d$ and $\hat{e} = \{\pm e_1, \cdots,  \pm e_d\}$.
For $f,g : \Lambda \rightarrow \R$,
\begin{align}
\begin{cases}
	(p \in [1,\infty)) & \norm{f}_{\ell^p} = \Big( \sum_{x \in \Lambda} |f(x)|^p \Big)^{1/p} , \\
	(p = \infty) & \norm{f}_{\ell^{\infty}} = \sup_{x \in \Lambda} |f(x)| 
\end{cases}	
\end{align}
and $(f,g)$ be the $\ell^2$-product.

\begin{itemize}
\item Covariance matrices $(\Gamma_j,  \Gamma_N^{\Lambda} )_{j\ge 0}$ on $t_N > 0$ are given by Proposition~\ref{prop:theFRD}.  
Due to translation and reflection invariance,  we also write $\Gamma_j (x,y) = \Gamma_j (x-y)$ and $\Gamma_N^{\Lambda_N} (x,y)= \Gamma_N^{\Lambda_N} (x-y)$.
We also define
\begin{align}
	w_j = \sum_{k=0}^j \Gamma_j, \qquad \beta_j = (8+n) \sum_{x \in \Lambda} \big( w_{j+1}^2 (x) - w_j^2 (x) \big) .
	\label{eq:wjdefi}
\end{align}

\item Given $g >0$,  let $\tilde{g}_0 = g$ and define for each $j\ge 0$
\begin{align}
	\tilde{g}_{j+1} = \begin{cases}
		\tilde{g}_j - \beta_j \tilde{g}_j^2 & (d=d_{c,u})  \\
		g & (d > d_{c,u}) .
	\end{cases}
\end{align}
If $g$ is sufficiently small,  it satisfies $\tilde{g}_{j+1} \le \tilde{g}_{j} \le 2\tilde{g}_{j+1}$,  and it is sufficient for applying \cite{FSmap},  comparing with \cite[(1.45)]{FSmap}.

\item Let $\kt$,  $\epsilon$,  $\epsilon_p >0$ be sufficiently small constants that may depend on $n,d,\eta$,  but not on any other parameter.
Also,  let $C_{\cD}, \cM$,  $p_\Phi$ be sufficiently large constants that may depend on $n,d,\eta$,  but not on any other parameter.
$\kt$ and $C_{\cD}$ are chosen in Proposition~\ref{prop:nuptcctty}, \ref{prop:nupcctty} and \ref{prop:obsRG} and $\epsilon_p$ in Proposition~\ref{prop:theFRD}.
For the other constants,  specific choices matter in the construction of the RG map,  but we will pay less attention to them here.  See \cite[Section~1.8]{FSmap} for specific choices.

\item A natural parameter for the decay of effective potentials under the RG flow is
\begin{align}
	\scale_j = L^{-(d-4+2\eta) j}  .
\end{align}
It appears with various exponents,  such as
\begin{align} \label{eq:kaekbekpedefi}
	\kae = \begin{cases}
		3 & (d=4) \\
		\frac{2d-7+2\eta}{d-4+2\eta} (1-\epsilon) & (d > 4) ,
	\end{cases}
	\qquad \kbe = \frac{2(1+\kae)}{3} , \qquad \kpe = \kae  -\kbe .
\end{align}
They satisfy $\kae \in (2,3]$ and $\kpe \in (0,1/2)$,  which will be used in Lemma~\ref{lemma:EVWox},  and specific choices are motivated by \cite{FSmap}.

\item The observable scale is
\begin{align}
	j_{\ox} = \min \big\{ j \ge 0 : 3 \cdot 2^d L^j > \dist_{\infty} (\o,\x) \big\} .
\end{align}
Define the fluctuation field scale $\ell_j = (\ell_{j,\bulk},  \ell_{j,\sigma} ,  \ell_{j, \ssigma}) \in \R^3$ and large field scale  $h_j = (h_{j,\bulk},  h_{j,\sigma} ,  h_{j, \ssigma})$ respectively,  by 
\begin{align} \label{eq:ellhdefi}
\begin{split}
	&
	\ell_{j, \bulk} = \ell_0 L^{-\frac{d-2+\eta}{2}j} ,  \quad\;
	\ell_{j,\sigma} = \begin{cases} 
		\tilde{g}_j L^{(1-\frac{3}{2}\eta) j \wedge j_\ox} 2^{(j-j_\ox)_+} & (d=4) \\
		\tilde{g}_j L^{-(d-5+ \eta) j} & (d \ge 5)
	\end{cases} ,\quad\;
		\ell_{j,\ssigma} = \ell_{j,\sigma}^2  ,\\
	&
	h_{j,\bulk} = \tilde{g}_j^{-\frac{1}{4}} L^{-\frac{d}{4}j}, \quad\;
	h_{j,\sigma} = \tilde{g}_j^{1/4} L^{\frac{d}{4} j} , \quad  \;
	h_{j,\ssigma} = \tilde{g}_j^{1/2} L^{\frac{d}{2} j \wedge j_{\ox}} \big( L^{\frac{d}{2} (1-\epsilon') - (d-4+2\eta) \kpe} \big)^{(j-j_\ox)_+} 
\end{split}
\end{align}
where $\ell_0 = L^{(d+p_\Phi)/2}$ and $\epsilon'>0$ is chosen sufficiently small in Lemma~\ref{lemma:EVWox} and \ref{lemma:ANox}.

\item Given $\tilde{m}^2 \ge 0$,  we consider the domain of $\ba^{(\emptyset)}$ given by
\begin{align}
	\II_j (\tilde{m}^2) :=
		\begin{cases}
		[0, L^{- (2-\eta) j}] & (\tilde{m}^2=0) \\
		[ \tilde{m}^2 / 2 ,  2 \tilde{m}^2  ] & (\tilde{m}^2 >0)  .
	\end{cases} \label{eq:IIdefi}
\end{align}
$\tilde{m}$ plays the role of the mass,  and the associated decay rate is encoded inside
\begin{align}
	j_{\tilde{m}^2} = \min\{ j\ge 0 : L^{(2-\eta)j} \tilde{m}^2 \ge 1 \}, \qquad \tilde{\chi}_j (\tilde{m}^2) = 2^{-(j-j_{\tilde{m}^2})_+}
\end{align}
where $(x)_+ = \max\{x,0\}$.
Similarly,  if we are given a collection $\vec{\ba} = (\ba^{(\emptyset)}, \cdots)$,  then we define $j_{\vec{\ba}}$ and $\chi_j (\vec{\ba})$ exactly the same,  but just with $\tilde{m}^2$ replaced by $\ba^{(\emptyset)}$.
The scales $j_{\tilde{m}^2}$ and $j_{\vec{\ba}}$ are called the mass scales.

\item For a normed space $S$ and $r >0$,  let $B_r (S)$ be the open ball in $S$ centred at the origin. 

\item $f\lesssim g$ means $f \le C$ for an $L$-independent constant $C$,  and $f \asymp g$ means $f\lesssim g \lesssim f$.
When $f \le C_L g$ for an $L$-dependent constant $C_L$,  we also denote $f \le O_L (g)$.

\item If $\zeta$ is a Gaussian random variable with mean 0 and covariance $C$,  expectation with respect to $\zeta$ is denoted $\E_C^{\zeta}$.
\end{itemize}

\section{Gaussian integrals and effective potential functions}
\label{sec:gssintsundeffpots}

In this section,  we aim to restate the $|\varphi|^4$ measure in terms of Gaussian integral of an effective potential function. 
Due to counterterms arising from the RG process,  both the covariance of the Gaussian integral and the effective potentials need to be modified from the original form of the Hamiltonian \eqref{eq:HgnuN}. 
These are introduced in Section~\ref{sec:modotcov} and \ref{sec:effpotes},  respectively.  
Then the Gaussian integral in the modified covariance is rephrased in terms of progressive integrals,  given by Section~\ref{sec:FRDdefi}.  They generate a flow of a pair $(V_j, K_j)$ of coordinates representing the RG flow,  where $V_j$ will be an effective potential and $K_j$ is an error term that will be dealt in Section~\ref{sec:polacts}.
Then in the final integral of Section~\ref{sec:finintgrl},  the moment generating functions and the correlation functions can be expressed in terms of these RG coordinates.  
These expressions will be used as key inputs to the proofs of the main theorems of this article.

\subsection{Fractional Laplacian} \label{sec:fracLap}

On a finite torus,  we define the fractional Laplacian using its Fourier representation.
Let $\lambda (p)$ ($p\in \Lambda^*$) be the Fourier symbol of $-\Delta$ (see Appendix~\ref{sec:fa} for the conventions for the Fourier transforms).

Explicitly,  $\lambda (p) = 2 \sum_{i=1}^d ( 1 - \cos ( p_i ))$,  where $p_i$ is the $i^{\rm th}$ component of $p \in \Lambda^*$ (well-defined modulo $2\pi$).  Then for any $\alpha \in \R$,  we define
\begin{align}
	(\varphi,  (-\Delta)^{\alpha} \varphi) = \frac{1}{|\Lambda|} \sum_{p \in \Lambda^*} e^{i x \cdot p} \lambda^{\alpha} (p) |\hat{\varphi} (p) |^2  .
\end{align}
The same formula holds on $\Z^d$ via
\begin{align}
	(\varphi,  (-\Delta )^{\alpha} \varphi) = \frac{1}{(2\pi)^d} \int_{[0,2\pi)^d} \lambda^{\alpha} (p) | \hat{\varphi} (p) |^2 \rd p .
\end{align}

\subsection{Modifications to the covariance}
\label{sec:modotcov}

In the RG analysis,  we need quadratic counterterms added on $(-\Delta)^{1-\eta/2}$.  They are labelled by (local) derivative indices,  as in \cite[Section~2.2]{FSmap},  which is a collection $\km = (m_k, \alpha_k)_{k=1}^{p (\km)}$ where $m_k = (\mu_{k,1}, \cdots, \mu_{k, i_k}) \in ( \hat{e} )^{i_k}$ for some $k, i_k\ge 0$ and $\alpha_i \in [n]$.  It defines a lcoal field monomial
\begin{align}
	M_x^{(\km)} (f) = \prod_{k=1}^{p (\km)} \nabla^{(m_k)}  f_x^{(\alpha_k)} = \prod_{k=1}^{p (\km)} \nabla^{\mu_{k,1} } \cdots \nabla^{ \mu_{k,i_k}}  f_x^{(\alpha_k)}
		\label{eq:fieldpolys}
\end{align}
for $f : \Lambda \rightarrow \R^n$,  where $f_x^{(\alpha)}$ is the $\alpha^{\rm th}$ component of $f_x \in \R^n$.  The degree of total derivative is denoted $q (\km) = \sum_{k=1}^{p (\km)} i_k$. 
A positive derivative index is $\km = (m_i, \alpha_i)_{i=1}^{p (\km)}$ with each $m_i \in \hat{e}_+$,  and the collection of positive derivative index is denoted $\ko_+$.  For each $\km \in \ko_+$,  we let $\kl (\km)$ be the derivative index with each occurrence of $\nabla^{\mu} \nabla^{\mu} \varphi_x$ replaced by $\nabla^{-\mu} \nabla^{\mu} \varphi_x$.  The symmetrise field monomial with index $\km \in \ko_+$ is
\begin{align}
	S^{(\km)}_x = \frac{1}{|\Sigma_{\rm axes}|} \sum_{\Theta \in \Sigma_{\rm axes}} {\lambda} (\Theta ,  \km) \Theta M^{(\kl(\km))}   \label{eq:sympolys}
\end{align}
where $\Sigma_{\rm axes}$ is the set of permutations of $\hat{e}$ generated by flips $e_i \leftrightarrow -e_i$ and 
\begin{align}
	\lambda(\Theta, \km) = \begin{cases}
		+1 \text{ if } \Theta_{\operatorname{axes}} \text{ flips even number of indices in } (\mu_{k,i})_{k \le p(\km), \,  i \le i_k} \\
		-1 \text{ if } \Theta_{\operatorname{axes}} \text{ flips odd number of indices in } (\mu_{k,i})_{k \le p(\km), \,  i \le i_k} .
	\end{cases}
\end{align}
We only use some particular subsets of $\ko_+$.  At the moment,  we just define
\begin{align}
	\ko_2 = \kA_0 , \qquad \ko_{2,\nabla} = \kA_1 \cup \kA_2 \cup \kA_3 \label{eq:ko2ko2nabla}
\end{align}
where
\begin{align}
	\kA_0 &= \big\{ \km \in \ko_+ \; : \; p(\km) = 2,  \; q(\km) = 0  \big\} ,
	\\
	\kA_1 &= \big\{ \km \in \ko_+ \; : \; p(\km) = 2,  \;  q(\km) \in (0, d-2+\eta) \cap 2\Z \big\} ,
	\\
	\kA_2 &= \big\{ \km \in \ko_+ \; : \; p(\km) = 2,  \;  q( \km ) = d-2+\eta \in 2\Z \big\} ,	
	\\
	\kA_3 &= \big\{ \km \in \ko_+ \; : \; p(\km) = 2,  \;  q(\km) \in (d-2+\eta,2d-6 ] \cap 2\Z \big\}  .
\end{align}
and
\begin{align}
\begin{split}
	\kB_0 = \ko_{4} 
		&= \big\{ \km \in \ko_\bulk : p(\km) = 4 ,  \; q(\km) = 0 \big\}  , \\
	\kB_1 = \ko_{4,\nabla} 
		&= \big\{ \km \in \ko_\bulk : p(\km) = 4 ,  \; q(\km) \in (0,  d-3+\eta ) \cap 2\Z  \}  \big\}   .
\end{split}		\label{eq:kotfnabla}
\end{align}
We also consider a symmetry on the collection of these indices given by
\begin{align}
\begin{split}
	&  \{ \vec{\ba} \in \R^{\kA_0 \cup \kA_1} : F_x (\varphi) = \sum_{\kA_0 \cup \kA_1} \ba^{(\km)} S_x^{(\km)} (\varphi) \;\; \text{is invariant under} \\
	& \qquad \qquad\qquad  \text{lattice symmetries of $\Lambda$ and $O(n)$-symmetries of $\varphi$}  \} .
\end{split}
	\label{eq:kAsymmetries}
\end{align}

\begin{remark} \label{remark:kAsymmetries}
When $\vec{\ba}$ satisfies \eqref{eq:kAsymmetries},  then there exists $\ba^{(\emptyset)}$ such that
\begin{align}
	\sum_{\km \in \kA_0} \ba^{(\km)} S_x^{(\km)} (\varphi) &= \ba^{(\emptyset)} |\varphi_x|^2 .
\end{align}
Thus we will pretend that $\kA_0$ only has a single element and denote it by $\emptyset$.

Similarly,  there exist $\ba^{(\Delta)}$ and $\ba^{(\nnabla)}$ such that
\begin{align}
	F'_{x} (\varphi) := \sum_{\km \in \kA_1}^{q(\km)=2} \ba^{(\km)} S_x^{(\km)} (\varphi) &= \ba^{(\nnabla)} \nabla \varphi \cdot \nabla \varphi + \ba^{(\Delta)} \varphi \cdot \Delta \varphi .
\end{align}
By summation by parts, 
\begin{align}
	\sum_x F'_{x} (\varphi) := \big( - \ba^{(\Delta)} + \ba^{(\nnabla)} \big) \varphi \cdot ( - \Delta ) \varphi ,
\end{align}
and we will denote $\bar{\ba}_{\Delta} = -\ba^{(\Delta)} + \ba^{(\nnabla)}$.
\end{remark}

Given $\vec{\ba} = (\ba^{(\km)} \in \R : \km \in  \kA_0 \cup \kA_1 )$ satisfying  \eqref{eq:kAsymmetries},  we define
\begin{align}
	(\varphi ,  \L_\eta^{(\vec{\ba})} \varphi ) 
		&= \big( \varphi,  ( (-\Delta)^{1-\eta/2} + \ba^{(\emptyset)} ) \varphi \big)
		+ \sum_{x \in \Lambda} \sum_{\km \in \kA_1} \ba^{(\km)} S^{(\km)}_x (\varphi ) 
		\label{eq:Lapkadefi} \\
	C^{(\vec{\ba})} 
		&= (\L_{\eta}^{(\vec{\ba})})^{-1} .   \label{eq:Cvecba}
\end{align}
To guarantee the positivity of $C^{(\vec{\ba})}$,  we restrict the domain of $\vec{\ba}$ by 
\begin{align}
	\vec{\ba} \in \HB_{\epsilon_p} 
		= \big\{ ( \ba^{(\km)})_{\km \in \kA_0 \cup \kA_1} \text{ satisfies \eqref{eq:kAsymmetries}}
		: \ba^{(\emptyset)} \in [0,\epsilon_p] ,  \;  \max_{\km \in \kA_1 } |\ba^{(\km)}| \le \epsilon_p
		&   \big\} .
\end{align}
Then by Lemma~\ref{lemma:Lapptv-res},  we have $C^{(\vec{\ba})} \ge 0$,  and it can be considered as a covariance matrix on $\Lambda$.

\subsection{Covariance decomposition}
\label{sec:FRDdefi}

When we apply the RG map constructed in \cite{FSmap},  each RG map involves a convolution integral in a Gaussian measure with a covariance $\Gamma_j$ given at scale $j$.  Namely,  we consider decomposition
\begin{align}
	C^{(\vec{\ba})}
		= \begin{cases}
	 	\sum_{j=1}^{\infty} \Gamma_j & (\Lambda = \Z^d) \\
	 	\sum_{j=1}^{N-1} \Gamma_j + \Gamma_j^{\Lambda_N} + t_N Q_N & (\Lambda = \Lambda_N) 
	 \end{cases}
	 \label{eq:LetaFRD}
\end{align}
where $t_N  \in \R_{>0}$ and $Q_N : \Lambda_N \times \Lambda_N \rightarrow \R$ is given by $Q_N (x,y) = L^{-dN}$ with the properties as in the next proposition.
The results were assumed in \cite{FSmap},  but we prove them in Section~\ref{sec:covdcmp}.

\begin{proposition} \label{prop:theFRD}
Let $d \ge 3$,  $\eta \in [0,2)$ and $\vec{\ba} \in \HB_{\epsilon_p}$ for sufficiently small $\epsilon_p$.  Then there exist covariance matrices $\Gamma_j$ on $\Z^d$ such that \eqref{eq:LetaFRD} and the following hold.

\begin{enumerate}
	\item (Symmetries) $\Gamma_j : \Z^d \times \Z^d \rightarrow \R$ is a covariance matrix invariant under isometries,
	i.e., $\Gamma_t \ge 0$ and $\Gamma_t (E (x) , E(y)) = \Gamma_j (x,y)$ for any isometry $E : \Z^d \rightarrow \Z^d$.
	\item (Finite range property) $\Gamma_j$ has range $< L^j$ in the $\ell^1$-metric,  i.e.,  $\Gamma_j (x,y) =0$ whenever $\norm{x-y}_{\ell^{1}} \ge L^j$. 
	\item (Upper bound) For each $k, k_x,k_y \ge 0$ with $k_x + k_y = k$,  
	\begin{align}
		\big\| \nabla_x^{k_x} \nabla_y^{k_y} \Gamma_{j+1} (x,y) \big\|_{\ell^\infty (\Z^d \times \Z^d)} \le C_k \frac{L^{-(d-2+\eta)j}}{1 + \ba^{(\emptyset)} L^{(2-\eta)j}}	 .  \label{eq:Gammajbounds1}
	\end{align}
	uniformly in $j,  L$ for some constant $C_k$.	
\end{enumerate}
When $\ba^{(\emptyset)} >0$ and $\Lambda = \Lambda_N$,  then there exist $t_N >0$ and a covariance $\Gamma_N^{\Lambda}$ on $\Lambda_N$,  that satisfy \eqref{eq:LetaFRD} with $\Gamma_j$ the projections of those of $\Z^d$ and satisfy the following.
\begin{enumerate}
	\item $\Gamma_N^{\Lambda}$ satisfies the same symmetries and the upper bounds as $\Gamma_N$.  

	\item $t_N \in (0, (\ba^{(\emptyset)})^{-1})$ and there exists $C >0$ such that $t_N >  (\ba^{(\emptyset)})^{-1} - CL^{(2-\eta) N}$.
\end{enumerate}
Moreover,  $\Gamma_j$ is continuous in $\vec{\ba} \in \HB_{\epsilon_p}$ and $\Gamma_N^{\Lambda}$ and $t_N$ are continuous in $\vec{\ba} \in \HB_{\epsilon_p} \backslash \{ \ba^{(\emptyset)} = 0 \}$.
\end{proposition}

We allow $\vec{\ba}$ to vary in an RG map,  so we fix $\tilde{m}^2 \ge 0$ in each RG map and take $\vec{\ba} \in \II_j (\tilde{m}^2)$ (recall \eqref{eq:IIdefi}) and define
\begin{align}
	\AA_j (\tilde{m}^2) = \{ \vec{\ba} \in \HB_{\epsilon_p} 
		:  \ba^{(\emptyset)} \in \II_j (\tilde{m}^2)  \}  .
		\label{eq:Ajdefi}
\end{align}
Letting
\begin{align}
	\kc_{j} = \tilde\chi_{j-1}^{1/2} L^{-\frac{1}{2} (d-2+\eta) (j-1)} ,
	\label{eq:kcdefi}
\end{align}
\eqref{eq:Gammajbounds1} can be restated as
\begin{align}
	\norm{\nabla^{k_x}_x \nabla^{k_y}_y \Gamma_{j+1}}_{\ell^{\infty}} \lesssim \kc_{j+1}^2 L^{-(|k_x| + |k_y|) j} ,\quad \; \vec{\ba} \in \AA_j (\tilde{m}^2)
	\label{eq:Gammajbounds2}
\end{align}
when $|k_x| + |k_y| \le  2 p_{\Phi} + 2d$.

Now,  recall from \cite[Corollary~1.2]{FSmap} that,  for covariance matrices $C_1, C_2$ and independent centred Gaussian random variables $\varphi_1 \sim \cN (0, C_1)$,  $\varphi_2 \sim \cN(0, C_2)$ and $\varphi \sim \cN(0,C_1 + C_2)$,
\begin{align}
	\E^{\varphi}_{C_1 + C_2} [F(\varphi)] = \E^{\varphi_1}_{C_2} \E^{\varphi_2}_{C_2} [F(\varphi_1 + \varphi_2)] \label{eq:C1C2F}
\end{align}
(recall the notation from Section~\ref{sec:notation})
whenever both sides are integrable. 

We can apply this identity and the covariance decomposition to decompose integrals in $C^{(\vec{\ba})}$:
if $\Lambda = \Lambda_N$ and $\ba^{(\emptyset)} >0$,  by \eqref{eq:C1C2F} and \eqref{eq:LetaFRD},
\begin{align}
	\E_{C^{(\vec{\ba})}} [F (\varphi)] = \E_{t_N Q_N} \E_{\Gamma_N^\Lambda}  \cdots \E_{\Gamma_1} [ F(\zeta_1 + \cdots + \zeta_N + \zeta_{\hat{N}})]
	\label{eq:Cdcmpcnvl}
\end{align}
for independent Gaussian random variables $\zeta_j \sim \cN (0, \Gamma_j)$ ($j < N$),  $\zeta_N \sim \cN(0, \Gamma_N^\Lambda)$ and $\zeta_{\hat{N}} \sim \cN(0,  t_N Q_N)$.
We also abbreviate
\begin{align}
	\E_{C^{(\vec{\ba})}} = \E_{\hat{N}} \E_N \cdots \E_1
	,\qquad  \E_j \equiv \E_{\Gamma_j} ,  \quad\; \E_{\hat{N}} = \E_{t_N Q_N}
	\label{eq:EhatNN1}
\end{align}

\subsection{Effective potential}
\label{sec:effpotes}

Quadratic counterterms added to the covariance should be subtracted from the potential function to correctly account for the $|\varphi|^4$-measure.  Thus given the coefficients $\vec{\ba} \in \R^{\kA_0 \cup \kA_1}$ of $\L_\eta^{(\vec{\ba})}$,  we consider potential function satisfying
\begin{align}
	\tilde{V} (\Lambda , \varphi) := \sum_{x \in \Lambda} \tilde{V}_{0,x} (\varphi) = V_{g,\nu, N} (\varphi) - \sum_{x \in \Lambda} \sum_{\km \in \kA_0 \cup \kA_1} \ba^{(\km)} S_x^{(\km)} (\varphi) .  \label{eq:Vtildedefi}
\end{align}

\begin{lemma} \label{lemma:Vtilde}
Consider $C^{(\vec{\ba})}$ given by \eqref{eq:Cvecba} and $V_0$ satisfying \eqref{eq:Vtildedefi}.  Then for any $F(\varphi)$ with appropriate integrability condition,
\begin{align}
	\langle F(\varphi) \rangle_{g,\nu, N} = \frac{\E_{C^{(\vec{\ba})}} [ F(\varphi) \exp (-\tilde{V} (\Lambda, \varphi))]  }{\E_{C^{(\vec{\ba})}} [ \exp (-\tilde{V} (\Lambda, \varphi))]} .
\end{align}
\end{lemma}
\begin{proof}
The identity is immediate if one realises that the quadratic function in \eqref{eq:Vtildedefi} subtracted from $V_{g,\nu,N}$ exactly cancels $\L^{(\vec{\ba})}_\eta - (-\Delta)^{1-\eta/2}$.
\end{proof}

We study the evolution of the modified potential function $\tilde{V}$ upon convolutions with decomposed  Gaussian measures given by Section~\ref{sec:FRDdefi}.
The intermediate integrals are expressed in terms of effective potential functions,  and some new polynomial terms arise that can not expressed by indices $\km \in \kA_0 \cup \kA_1$ upon integrals of \eqref{eq:Cdcmpcnvl}.
Thus there is a need to extend the space of effective potentials beyond \eqref{eq:Vtildedefi},  by adding observable fields and higher order terms.

\subsubsection{Observable fields}

We extend the effective potential using observable fields to express correlation functions concisely,  see Proposition~\ref{prop:2ptfncrestt}.
Observable fields are simply distinct elements $\sigma_\o$ and $\sigma_\x$ that generate a commutative ring $R$ via
\begin{align}
	1 =: \sigma_{\bulk}, \qquad \sigma_\o^2 = \sigma_\x^2 = 0, \qquad \sigma_\o \sigma_\x = \sigma_\x \sigma_\o =: \sigma_{\ox} \neq 0 .
\end{align}
For any Abelian group $M_\bulk$,  we can consider a graded $R$-module given by
\begin{align}
	M = M_\bulk \oplus M_\o \oplus M_\x \oplus M_\ox , \qquad M_* = \sigma_* M_\bulk
\end{align}
for $* \in \{\bulk , \o,\x,\ox\}$.
We let $\pi_{*}$ for the projection on each respective space, 
and each $m \in M$ can be denoted
\begin{align}
	m = m_\bulk + \sigma_\o m_\o  + \sigma_\x m_\x + \sigma_\ox m_\ox , \qquad m_\bulk,m_\o,m_\x,m_\ox \in M_\bulk .
\end{align}
We make this extension for a number of algebraic structures that appear throughout this article without particular mentioning.

\subsubsection{Indices for higher order terms}
\label{sec:indfhot}

Indices \eqref{eq:ko2ko2nabla}--\eqref{eq:kotfnabla} define the bulk effective potential
\begin{align}
	V_{\bulk} = V_{2} + V_{2,\nabla} +  V_{4} + V_{4,\nabla} \in  \cV_2 + \cV_{2,\nabla} + \cV_4 + \cV_{4,\nabla} = \cV_{\bulk}
\end{align}
by
\begin{align}
	& V_2 = \sum_{\km \in \ko_2} \nu^{(\km)} S^{(\km)}  ,  \qquad 
	V_{2,\nabla} = \sum_{\km \in \ko_{2,\nabla}} \nu^{(\km)} S^{(\km)} , \\
	& V_4 = \sum_{\km \in \ko_4 } g^{(\km)} S^{(\km)} ,\qquad
	V_{4,\nabla} = \sum_{\km \in \ko_{4,\nabla}} g^{(\km)} S^{(\km)}
\end{align}
and require additional symmetry that $V_{\bulk,x} (\varphi)$ is invariant under lattice symmetries of $x \in \Lambda$ and $O(n)$-symmetries of $\varphi$.
Projections on each subspace is denoted $\pi_*$ for $* \in \{ 2,(2,\nabla),4, (4,\nabla) \}$.

The observable effective potential is expressed as $V_{\sigma} = \sigma_\o V_\o + \sigma_\x V_\x$ where for each $\hash \in \{ \o,\x\}$,  $V_\hash$ has form $V_{1} + V_{1,\nabla} \in \cV_1 + \cV_{1,\nabla}$ where
\begin{align}
	V_{1,x} = \sum_{\km \in \ko_1} \lambda_\hash^{(\km)} S^{(\km)}_x \one_{x = \hash}  , \qquad
	V_{1,\nabla,x} = \sum_{\km \in \ko_{1,\nabla}} \lambda_\hash^{(\km)} S^{(\km)}_x \one_{x = \hash}
\end{align}
with labels
\begin{align}
	&\ko_1 = \{ \km \in \ko_+ \; : \; p(\km) =1,  \; q(\km) =0 \},  \label{eq:ko1defi} \\
	& \ko_{1,\nabla} = \{ \km \in \ko_+ \; : \; p(\km) =1,  \; q(\km) \in (0, (d-2+\eta)/2 ) \}  .  \label{eq:ko1nabladefi}
\end{align}
We require $V_{\hash}$ to satisfy certain symmetries: $V_{\hash} (- \varphi) = -V_{\hash} (\varphi)$ and $V_{\hash} (R\varphi) = V_{\hash} (\varphi)$ when $R \in O(n)$ acts via $(R\varphi)_x = \varphi_x$ and fixes each $\varphi^{(1)}_x$.
Projections on each space is denoted $\pi_*$ for $* \in \{ (1, (1,\nabla) \}$.
The set of effective potentials $V = V_\bulk + V_{1} + V_{1,\nabla}$ is denoted $\cV$.
For $V \in \cV$ and $X \subset \Lambda$,  we denote $V(X) = \sum_x V_x$.  
We also extend $\cV$ by adding constant terms
\begin{align}
	U = u + V \in  \cU = (\R + \sigma_\ox \R)^{\Lambda} + \cV
\end{align}
where $u = u_\bulk + \sigma_\ox u_\ox$ has form
\begin{align}
	u_{\bulk} (\Lambda) = u_{\bulk} |\Lambda| , \qquad u_{\ox,x} (\varphi) = q_\o \one_{x=\o} + q_\x \one_{x=\x}
\end{align}
for some $q_\o, q_\x \in \R$.
We denote $\pi_0 : U \mapsto u_\bulk$ and $\pi_{\ox} : U \mapsto \sigma_\ox u_\ox$.

\begin{remark} \label{remark:Vsymmetries}
Just as in Remark~\ref{remark:kAsymmetries},
although $\ko_1$,  $\ko_2$ and $\ko_4$ are not composed of a single element,  but by the symmetries mentioned above,  we can express
\begin{align}
	V_1 (\varphi)_x = \lambda_{\hash}^{(\emptyset)} \varphi_x^{(1)} \one_{x=\hash} , \quad 
	V_2 (\varphi)_x = \frac{\nu^{(\emptyset)}}{2}  |\varphi_x|^2 , \quad 
	V_4 (\varphi)_x = \frac{g^{(\emptyset)}}{4} |\varphi_x|^4 
\end{align}
for some $\nu^{(\emptyset)}$,  $g^{(\emptyset)}$ and $\lambda_{\hash}^{(\emptyset)}$.  
Thus we will pretend that $\ko_1$,  $\ko_2$ and $\ko_4$ are composed of a single element and denote them by $\emptyset$ when there is no source of confusion.

Again as in Remark~\ref{remark:kAsymmetries},  the part of $V_{2,\nabla}$ with two derivatives ($q(\km)=2$) can be expressed as
\begin{align}
	\nu^{(\nnabla)} \nabla \varphi \cdot \nabla \varphi + \nu^{(\Delta)} \varphi \cdot \Delta \varphi .
\end{align}
Since they can be equated,  we consider operation $\mathbb{V}^{(0)} : \cV_2 \rightarrow \cV_2$ that maps $(\nu^{(\Delta)},  \nu^{(\nnabla)}) \mapsto (\bar{\nu}^{(\Delta)},0) = (\nu^{(\Delta)}- \nu^{(\nnabla)}, 0 )$.
This operation also extends to $\mathbb{V}^{(0)} : \cV \rightarrow \cV$ naturally.
\end{remark}

We also restrict the size of the coefficients as the following. 
For a parameter $C_{\cD} > 0$ and $\alpha \ge 1/2$,
\begin{align}
\label{eq:cDbulk}
		\begin{split}
	& \cD_{j,\bulk}  (\alpha)
		= \Big\{ (\nu_{j}^{(\km_1)},  g_{j}^{(\km_2)})  : 
			|\nu_{j}^{(\km_1)}| \le \alpha  C_{\cD} L^{(q (\km_1) - 2 + \eta) j} \scale_j \tilde{g}_j		\text{ if } \km_1 \in \kA_0 \cup \kA_1 \cup \kA_2 , \\
			& \qquad\qquad\qquad\qquad |\nu_{j}^{(\km_1)}| \le \alpha  C_{\cD} \scale_j^{-\kt} \tilde{g}_j 
			\text{ if } \km_1 \in \kA_3  ,  \;\;\;
			 g^{(\emptyset)}_j / \tilde{g}_j \in ( \alpha C_{\cD})^{-1},    \alpha C_{\cD}),  \\
			& \qquad\qquad\qquad\qquad 
			| g_{j}^{(\km_2)}| \le \alpha  C_{\cD} \scale_j^{-\kt} \tilde{g}_j^{3/2}  \text{ if } \km_2 \in \ko_{4,\nabla}
			\Big\}
	\end{split} \\
	\begin{split} \label{eq:Dstsigmadefi}
	& \cD_{j, \sigma} (\alpha)
		= \Big\{ (\lambda_{j, \o}^{(\km)} , \lambda_{j, \x}^{(\km)} )_{\km \in \ko_1 \cup \ko_{1,\nabla}}   \;  : \;  
		|\lambda^{(\km)}_{j, \hash}| <  \alpha C_{\cD} L^{q(\km) j}  \text{ if } q (\km) \le 1 , \\
		& \qquad\qquad\qquad\qquad\qquad\qquad\qquad   |\lambda^{(\km)}_{j, \hash}| < \alpha C_{\cD} \scale_j^\kt L^{(2-\eta) j}  \text{ if } q (\km) \ge 2
		\Big\} 
\end{split}		
\end{align}
and define
\begin{align}
	\cD_j = \cD_{j, \bulk} (\alpha) \times \cD_{j, \sigma} (\alpha), \qquad \cD_j^{(0)} (\alpha)= \mathbb{V}^{(0)} \cD_j (\alpha) .
	\label{eq:cDzerodefi}
\end{align}
When $\alpha$ is omitted,  it means $\alpha =1$.

\subsection{Final integral} \label{sec:finintgrl}

Motivated by the successive integrals \eqref{eq:Cdcmpcnvl}, we define
$Z_0 (\varphi) = \exp ( -V_0 (\Lambda, \varphi) )$ and
\begin{align} \label{eq:Zjdfn}
	Z_{j} (\varphi) = \E_{j} \cdots \E_{1} \big[ \exp\left( - V_0 (\Lambda, \zeta_1 + \cdots + \zeta_j + \varphi ) \right) \big] , 
	\qquad \varphi \in (\R^n)^{\Lambda}
\end{align}
for $j= 1, \cdots, N$. They satisfy inductive relations
\begin{align}
	Z_{j+1} (\varphi) = \E_{j+1} \thetaz Z_j (\varphi),  \qquad \zeta \sim \cN(0,\Gamma_{j+1})
	\label{eq:Zjinductive}
\end{align}
where we denoted $\thetaz F (\varphi) = F(\varphi + \zeta)$.
Similarly, we also inductively define
\begin{align}
	Z_{0,\bulk} (\varphi) = e^{-V_{0,\bulk} (\Lambda, \varphi)}, 
	\qquad
	Z_{j+1 ,  \bulk} (\varphi) = \E_{j+1} \thetaz Z_{j,\bulk} (\varphi)
\end{align}
for $j+1 \le N$. 
Given $Z_N$,  the final integral $\E_{\hat{N}}$ (recall \eqref{eq:EhatNN1}) with covariance $t_N Q_N$ can be expressed as an one-dimensional integral,  as the we can see in the next result.

\begin{proposition}
\label{prop:mgfref}

Let $V_{0,\bulk} \in \cV$,  $\vec\ba \in \HB_{\epsilon_p}$,  $g^{(\emptyset)} >0$ and $g^{(\km)} =0$ for $\km \in \ko_{4,\nabla}$.
Then for $f \in (\R^n )^{\Lambda}$ and $\ba^{(\emptyset)} >0$,
\begin{align}
	\frac{\E_{C^{(\vec{\ba})}} [ e^{-V_{0,\bulk} (\Lambda, \varphi) + (f,\varphi) }  ] } {\E_{C^{(\vec{\ba})}} [ e^{-V_{0,\bulk} (\Lambda, \varphi) }  ]}
		= e^{\frac{1}{2} (f,  C^{(\vec{\ba})} f)}  \label{eq:mgfref1}
		\frac{\int_{\R^n} Z_{N,\bulk} ( y \one +  C^{(\vec{\ba})} f ) e^{-\frac{1}{2} t_N^{-1} L^{dN} |y|^2  } \rd y }{\int_{\R^n} Z_{N,\bulk}  ( y \one ) e^{-\frac{1}{2} t_N^{-1} L^{dN} |y|^2  } \rd y } \\
		= e^{\frac{1}{2} (f,  w_N  f)}  \label{eq:mgfref2}
		\frac{\int_{\R^n} Z_{N,\bulk} ( y \one +  w_N f ) e^{(f, y \one)} e^{-\frac{1}{2} t_N^{-1} L^{dN} |y|^2  } \rd y }{\int_{\R^n} Z_{N,\bulk}  ( y \one ) e^{-\frac{1}{2} t_N^{-1} L^{dN} |y|^2  } \rd y }
\end{align}
where $\one$ is the unit-valued function on $\Lambda$.
If $\ba^{(\emptyset)} = 0$,  then \eqref{eq:mgfref2} holds with $t_N^{-1} = 0$.
\end{proposition}

\begin{proof}
We drop $\bulk$ in the proof for brevity. 
We first consider $\ba^{(\emptyset)} > 0$ so that $t_N (\vec{\ba}) \in (0, \ba_\bulk^{-1})$
and denote $\E_N \cdots \E_1 = \E_{\le N}$.
Then by \eqref{eq:EhatNN1},
\begin{align}
	\E_{C^{(\vec{\ba})}} \left[  e^{-V_{0} (\Lambda, \varphi) + (f,\varphi) } \right]
		= \E_{\hat{N}} \E_{\le N} \left[  e^{-V_{0} (\Lambda, \zeta_{\le N} + \zeta_{\hat{N}} ) + (f, \zeta_{\le N} + \zeta_{\hat{N}} ) } \right]  .
		\label{eq:mgfref3}
\end{align}

Both sides are integrable because $V_{0}$ contains a term that grows with quartic order in $\varphi$.
By Gaussian change of variable $\zeta_{\le N} \mapsto \zeta_{\le N} + C^{(\vec{\ba})} f$,
\begin{align}
		&= e^{\frac{1}{2} (f,C^{(\vec{\ba})} f)} \E_{\hat{N}} \,  \E_{\le N} \left[  e^{-V_{0} (\Lambda, \zeta_{\le N} + \zeta_{\hat{N}} + C^{(\vec{\ba})} f )  } \right]  \nnb
		&= e^{\frac{1}{2} (f,C^{(\vec{\ba})} f)} \E_{\hat{N}} \,  Z_{N} \big( \zeta_{\hat{N}} + C^{(\vec{\ba})} f \big)  .
\end{align}
Now, notice that $\zeta_{\hat{N}} =^d \tilde{Y} \one$ for some $\R^n$-valued Gaussian random variable $\tilde{Y} \in \cN(0,  L^{-dN} t_N)$,
so we can rewrite the expectation as
\begin{align}
	\propto e^{\frac{1}{2} (f, C^{(\vec{\ba})} f)} \int_{\R^n}  Z_{N} \big( y \one + C^{(\vec{\ba})} f \big)  e^{-\frac{1}{2} t_N^{-1} L^{dN} |y|^2} \rd y
	,
\end{align}
so we obtain \eqref{eq:mgfref1} for $\ba^{(\emptyset)} >0$.

For \eqref{eq:mgfref2},  we apply Gaussian change of variable $\zeta_{\le N} \mapsto \zeta_{\le N} + w_N f$ on \eqref{eq:mgfref3} to obtain
\begin{align}
	&= e^{\frac{1}{2} (f, w_N f)} \E_{\hat{N}} \left[ e^{(f, \zeta_{\hat{N}}} \E_{\le N} \left[  e^{-V_{0,\bulk} (\Lambda, \zeta_{\le N} + \zeta_{\hat{N}} + w_N f ) ) } \right] \right] \nnb
	&= e^{\frac{1}{2} (f,\Gamma_{\le N} f)} \E_{\hat{N}} \, e^{(f, \zeta_{\hat{N}} ) } Z_{N,\bulk} \big( \zeta_{\hat{N}} + w_N f \big)  
\end{align}
and we obtain \eqref{eq:mgfref2} for $\ba^{(\emptyset)} >0$ by substitution $\zeta_{\hat{N}} = \tilde{Y} \one$.

For $\ba^{(\emptyset)} = 0$, 
observe that $\Gamma_j$ is continuous in $\vec{\ba}$ due to Proposition~\ref{prop:theFRD} and because $t_N > (\ba^{(\emptyset)} )^{-1} - C L^{-2N}$, 
we have $t_N^{-1} \rightarrow 0$ as $\ba^{(\emptyset)} \downarrow 0$.
Thus we can take the limit $\ba^{(\emptyset)} \downarrow 0$ and apply the Dominated convergence theorem, 
whose uniform integrability is again guaranteed by the quartic growth of $V_0$.
\end{proof}

A similar statement can be used for the two-point function.

\begin{proposition} \label{prop:2ptfncrestt}
Under the assumptions of Proposition~\ref{prop:mgfref},  also let $\pi_{\hash} V_0 = \sigma_\hash \lambda^{(\emptyset)}_{0, \hash} \varphi^{(1)}_\hash$ for both $\hash \in \{\o,\x\}$.  Then
\begin{align}
	\lambda^{(\emptyset)}_{0, \o} \lambda^{(\emptyset)}_{0, \x}	\frac{\E_{C^{(\vec{\ba})}} [ \varphi_\o^{(1)} \varphi_\x^{(1)} e^{-V_{0,\bulk} (\Lambda, \varphi)  }  ] } {\E_{C^{(\vec{\ba})}} [ e^{-V_{0,\bulk} (\Lambda, \varphi) }  ]}
		& = 
		\frac{\int_{\R^n} Z_{N,\ox} ( y \one ) e^{-\frac{1}{2}  t_N^{-1} L^{dN}  |y|^2  } \rd y }{\int_{\R^n} Z_{N,\bulk}  ( y \one ) e^{-\frac{1}{2} t_N^{-1} L^{dN} |y|^2  } \rd y }  .
\end{align}
If $\ba^{(\emptyset)}=0$,  then the identity holds with $t_N^{-1} = 0$.
\end{proposition}
\begin{proof}
Observe that
$\lambda^{(\emptyset)}_{0, \o} \lambda^{(\emptyset)}_{0, \x}
	\varphi_\o^{(1)} \varphi_\x^{(1)} e^{-V_{0,\bulk} (\Lambda,\varphi)}  = \pi_\ox e^{-V_0 (\Lambda,\varphi)}$.
Thus by the definition of $Z_0$,  we have
\begin{align}
	\lambda^{(\emptyset)}_{0, \o} \lambda^{(\emptyset)}_{0, \x}	\frac{\E_{C^{(\vec{\ba})}} [ \varphi_\o^{(1)} \varphi_\x^{(1)} e^{-V_{0,\bulk} (\Lambda, \varphi)  }  ] } {\E_{C^{(\vec{\ba})}} [ e^{-V_{0,\bulk} (\Lambda, \varphi) }  ]} = \frac{\E_{C^{(\vec{\ba})}} [Z_{0, \ox} (\varphi)] }{\E_{C^{(\vec{\ba})}} [Z_{0, \bulk} (\varphi)] }
\end{align}
and the conclusion follows from the same steps as Proposition~\ref{prop:mgfref}. 
\end{proof}

By combining Lemma~\ref{lemma:Vtilde} and Proposition~\ref{prop:mgfref} or \ref{prop:2ptfncrestt},  
we can express the moment generating function $\langle e^{(f,\varphi)} \rangle_{g, \nu, N}$ and the two-point function $\langle \varphi_\o \cdot \varphi_\x \rangle_{g,\nu, N}$ in terms of integrals of $Z_N$.

\section{Polymer activities}
\label{sec:polacts}

After integrating out fluctuations below scale $j$ \eqref{eq:Zjdfn},
$Z_j$ describes the spin system at scale $j$.  We will approximate $Z_j$ with an effective potential $V_j \in \cV_j$ with error $K_j$,  expanded via a cluster-type expansion \eqref{eq:circleproduct}.
Error coordinate $K_j$ will be defined as a polymer activity,  residing in space $K_j$.
We will explain these terminologies in this section,  mostly restating notations defined in \cite{FSmap}.

For each $j \in \{ 0, \cdots, N \}$,  we let $B_{0,j} = [ - \lfloor \frac{L^j - 1}{2} \rfloor , \lfloor \frac{L^j }{2} \rfloor ]^d \subset \Lambda$ and let $\cB_j$,  the set of $j$-blocks,  be the set of $(L^j \Z)^d$-translations of $B_{0,j}$.
Any union of $\cB_j$ is called a $j$-polymer,  and the set of $j$-polymer is denoted $\cP_j$.
For $X \subset \Lambda$,  $\cB_{j} (X) = \{ B \in \cB_j : B \subset X \}$ and $\cP_{j} (X) = \{ Y \in \cP_j : Y \subset X \}$.

We use a collections of functions $(K_j (X,\varphi) : X \in \cP_j)$ and $(I_j (b,\varphi) : b \in \cB_j )$ of smooth functions in $\varphi$.  The polymer expansion of $(I_j, K_j)$ is the operation
\begin{align}
	(I_j \circ K_j) (\Lambda,\varphi) = \sum_{X \in \cP_j} \prod_{b \in \cB_j (\Lambda \backslash X)} I_j (b,\varphi) K_j (X,\varphi) .
	\label{eq:circleproduct}
\end{align}
In \cite{FSmap},  for each $j\ge 0$,  we consider a pair $(I_j, K_j)$ such that
\begin{align}
	Z_j = e^{-u_j (\Lambda)} (I_j \circ K_j ) (\Lambda,\varphi) .
	\label{eq:Zjpolyexp}
\end{align}
We will defined $I_j = \cI_j (V_j)$ as a function of $V_j \in \cV$ in Section~\ref{sec:WIcoords} and it can be thought of as the effective Boltzmann factor corresponding to the effective potential $V_j$,  while $K_j$ can be considered to be an error term.

\subsection{Norms on function spaces}

Let $\Lambda_1, \cdots, \Lambda_n$ be copies of $\Lambda$ and let $\Lambda_{\rm b} = \Lambda_1 \sqcup \cdots \sqcup \Lambda_n$.
A lattice polynomial is a function $g = (g^{(r)})_{r \ge 0} \in \prod_{r=0}^{\infty} \Phi^{(r)} =: \Phi$ where $\Phi^{(r)}$ is the set of functions $g^{(r)} : ( \Lambda_{\rm b} ) ^{r} \rightarrow R$.  
Given $\kh \in \{\ell, h\}$ (recall \eqref{eq:ellhdefi}),  let
\begin{align}
	\norm{g^{(r)}}_{\kh, \Phi_j^{(r)}}
		&= \max_{n \le p_{\Phi}} L^{nj} \norm{\nabla^{n} g^{(r)}}_{\ell^{\infty} (\Lambda)} ,  \\
	\norm{g}_{\kh,  \Phi_j} 
		&= \sup_{r \ge 0} \kh_{\bulk}^{-r} \norm{g^{(r)}}_{\kh, \Phi_j^{(r)}} 
\end{align}
(see \cite[Section~2.2]{FSmap} for detailed explanation).
Let $K(X,\varphi)$ be an $\R$-valued smooth function of $\varphi \in (\R)^{\Lambda_N}$ for each $X \in \cP_j$.  Then we define
\begin{align}
	\norm{K (X,\varphi)}_{\kh,  T_j (X,\varphi)} 
		&= \sup\Big\{ 
			\sum_{n=0}^{\infty} \frac{1}{n!} D^n_{\varphi} K (X, \varphi ; g^{(r)})
			: g \in \Phi ,  \; \norm{g}_{\kh, \Phi_j } \le 1 \Big\}  .
\end{align}
If $K = K_\bulk + \sigma_\o K_\o + \sigma_\x K_\x + \sigma_\ox K_\ox$ is $\R + \sigma_\o \R + \sigma_\x \R+ \sigma_\ox \R$-valued,  then we extend this via
\begin{align}
	\norm{K (X,\varphi)}_{\kh,  T_j (X,\varphi)}  
		&= \norm{K_\bulk (X,\varphi)}_{\kh,  T_j (X,\varphi)} + \kh_{\sigma}  \norm{K_\o (X,\varphi)}_{\kh,  T_j (X,\varphi)} \nnb
		& \qquad + \kh_\sigma \norm{K_\x (X,\varphi)}_{\kh,  T_j (X,\varphi)} + \kh_\ssigma \norm{K_\ox (X,\varphi)}_{\kh,  T_j (X,\varphi)} .
\end{align}

\subsection{Space of error terms $K$}

We equip norm $\norm{\cdot}_{\cW_j}$ on $K_j$ as in \cite[Section~2.6.1]{FSmap} and let $\cK_j$ be the space of $K_j$ defined as in \cite[Section~2.7]{FSmap}.
We do not write the specific definitions,  but we just use the following results. 
In the next lemma,  we use $\R \cK_j = \{ r K : K \in \cK_j,  \; r \in \R \}$.

\begin{lemma} \label{lemma:KNbnd}
There is a ($L$-dependent) constant $C_{\rg} > 0$ such that for $K \in \R\cK_j$ and any $j \le N$,  
\begin{align}
	 K \in \cK_{j} \quad \Leftrightarrow \quad \norm{K}_{\cW_j} \le C_{\rg} \tilde{\chi}_j^{3/2} \tilde{g}_j^3 \scale_j^{\kae}
\end{align}
and whenever $\norm{K_N}_{\cW_N} < \infty$,
\begin{align}
	| K_N (\Lambda, 0)  | & \le O_L (1) \norm{K_N}_{\cW_N}
	\\
	| K_N (\Lambda, \varphi) | & \le O_L (1) \tilde{g}_N^{-\frac{9}{4}} \scale_N^{ - \kae + \kpe} \tilde{G}_N (\Lambda,  \varphi) e^{-\kappa L^{-dN} \norm{\varphi / h_{N,\bulk}}_{\ell^2}^2 } \norm{K_N}_{\cW_N}  \label{eq:KNbnd2}
\end{align}
where $\tilde{G}_N (\Lambda, \varphi)$ is as in \cite[(2.22)]{FSmap} and $\kappa >0$ is a constant (independent of any other parameter).
\end{lemma}
\begin{proof}
They follow directly from the definition \cite[(2.38)]{FSmap} of the $\norm{\cdot}_{\cW_N}$-norm.
\end{proof}

The $\ell^2$-decay estimate in \eqref{eq:KNbnd2} is the revolution of \cite{FSmap} compared to that of \cite{BBS5}.
We also clarify some properties of the `regulator' $\tilde{G}_N$.

\begin{lemma} \label{lemma:logtildeGNbnd}
For any $c \in \R^n$,  we have $\tilde{G}_N (\Lambda, \varphi + c\one) = \tilde{G}_N (\Lambda, \varphi) \le G_N (\Lambda, \varphi)$ and
\begin{align}
	\log G_N (\Lambda, \varphi)
		\lesssim L \norm{\varphi}^2_{\ell_N, \Phi_N}	
		 \lesssim L \ell_{N,\bulk}^{-2} L^{-dN} \sum_{n \le d + p_{\Phi}} L^{2n N} \norm{\nabla^n \varphi}^2_{\ell^2 (\Lambda)} .
	\label{eq:logtildeGNbnd}
\end{align}
\end{lemma}
\begin{proof}
The first relation is due to \cite[(2.15),  (2.21),  (2.22)]{FSmap}.  For \eqref{eq:logtildeGNbnd},  $G_N$ is bounded in terms of $\norm{\varphi}_{\ell_N, \Phi_N}$ due to \cite[(2.21)]{FSmap} and \cite[Lemma~B.2]{FSmap} bounds $\norm{\varphi}_{\ell_N, \Phi_N}$ in terms of $\norm{\cdot}_{\ell^2}$-norms as desired. 
\end{proof}

\begin{corollary} \label{cor:logtildeGNbndcor}
For any $c \in \R^n$,  we have $\tilde{G}_N (\Lambda, c\one)  \le 1$.
\end{corollary}
\begin{proof}
This is direct from Lemma~\ref{lemma:logtildeGNbnd},  since $\tilde{G}_N (\Lambda, c\one) = \tilde{G}_N (\Lambda,0) \le 1$.
\end{proof}

We also need a topological property on the space of $K_j$.

\begin{lemma} \label{lemma:RcKisBanach}
$\R \cK_j$ is a Banach space equipped with norm $\norm{\cdot}_{\cW_j}$ and $\cK_j$ is a closed ball in $\R \cK_j$.
\end{lemma}
\begin{proof}
This statement is basically \cite[Proposition~1.8]{BBS5},  but just with a slightly modified norm.  It follows from \cite[Proposition~A.1]{BBS5},  which holds for any regulator that is positive and continuous in $\varphi$,   so we have the same conclusion for our $\norm{\cdot}_{\cW_j}$.
\end{proof}

Finally,  we can define the RG domain as
\begin{align}
	\D_j (\alpha) = \cD^{(0)}_j (\alpha) \times \cK_j (\alpha) :=  \cD^{(0)}_j (\alpha)  \times \alpha \cK_j , \qquad \D_{\bulk,j} (\alpha) = \pi_\bulk \D_j (\alpha)
	\label{eq:RGdomain}
\end{align}
for $\alpha \le 1$,  where $\cD^{(0)}_j$ was defined in \eqref{eq:cDzerodefi}. 
As before,  we omit $\alpha$ when $\alpha =1$.

\subsection{Norms on effective potentials}
\label{sec:effpots}

In \cite[Section~4]{FSmap},  two distinct norms $\norm{\cdot}_{\cL_j (\ell)}$ and $\norm{\cdot}_{\cV_j (\ell)}$ are defined on $\cV$ at scale $j$.  They satisfy the following bounds. 

\begin{definition} \label{defi:cLnorm}
We equip spaces $\cV_\bulk, \cV_\o,\cV_\x$ and $\cU_\ox$ with norm
\begin{align}
	\norm{V_\bulk}_{\cL_j (\ell)} 
		& = L^{jd} \max \big\{ \ell_{j, \bulk}^2 L^{- q (\km_1 ) j} | \nu^{(\km_1)} | ,  \;
		\ell_{j, \bulk}^4 L^{- q (\km_2) j} | g^{(\km_2)}|  \nnb
		& \qquad\qquad\qquad\quad :\,  \km_1 \in \ko_2 \cup \ko_{2,\nabla} ,  \,  \km_2 \in \ko_4 \cup \ko_{4,\nabla}  \big\} ,
		\\
	\norm{\sigma_\hash V_\hash}_{\cL_j (\ell)} 
		& = \ell_{j, \bulk} \ell_{j, \sigma} \max\{ L^{-q(\km) j} |\lambda_\#^{(\km)} | : \km \in \ko_1 \cup \ko_{1,\nabla} \}  ,   \qquad \# \in \{ \o, \x \} \\
	\norm{\sigma_\ox U_\ox}_{\cL_j (\ell)} 
		& = \ell_{j, \ssigma} \big( |q_\o | + |q_\x | \big)
\end{align}
and for $U = u_\bulk + \sum_{* \in \{\bulk,\o,\x\}} \sigma_* V_* + \sigma_\ox U_\ox \in \cU$,  
\begin{align}
	\norm{U}_{\cL_j (\ell)} = L^{jd} |u_\bulk| + \sum_{* \in \{ \bulk,\o,\x \}}  \norm{\sigma_* V_*}_{\cL_j (\ell)} + \norm{\sigma_\ox U_\ox}_{\cL_j (\ell)} .
	\label{eq:cLnorm}
\end{align}
\end{definition}

\begin{definition} \label{defi:cVnormdefi}
Define
\begin{align}
	\label{eq:cVnormdefi}
	\norm{V_\bulk}_{\cV_j (\ell)} 
		& = \ell_{j,\bulk}^2  \max\big\{ L^{(d - q (\km_1)) j}  |\nu^{(\km_1)}|  \,  :\,  \km_1 \in \kA_0 \cup \kA_1 \cup \kA_2 \big\} 
		\nnb &\quad +
		\ell_{j,\bulk}^2 \scale_j^{\kt} L^{(d - 2 [\varphi] ) j} \max\big\{  |\nu^{(\km_1)}|  \,  :\,  \km_1 \in \kA_3 \big\} 
		\nnb &\quad +
		\ell_{j,\bulk}^4  L^{dj }  \max\big\{ 
		|g^{(\km_2)}|  \,  :\,   \,  \km_2 \in \ko_4  \big\} 
		\nnb &\quad +
		\ell_{j,\bulk}^4 \scale_j^{\kt} L^{ dj} \max\big\{ 
		 |g^{(\km_2)}|  \,  :\,   \,  \km_2 \in \ko_{4,\nabla}  \big\}
\end{align}
and for $\hash \in \{ \o ,  \x\}$,
\begin{align}
	\norm{\sigma_\hash V_{\hash}}_{\cV_j (\ell)} &= 
		\ell_{j,\bulk} \ell_{j,\sigma}  \max\{ L^{-q (\km) j} |\lambda^{(\km)}_\hash  |  \, : \,  q (\km) \le 1 \}  \\
		& \quad + 
		\ell_{j,\bulk} \ell_{j,\sigma} \scale_j^\kt \max\{ L^{(-2+\eta) j } |\lambda^{(\km)}_\hash  |  \, : \,  q (\km) \ge 2 \} .
\end{align}
For generic $U = u_\bulk +  \sum_{* \in \{\bulk,\o,\x\}} \sigma_* V_* + \sigma_\ox U_\ox \in \cU$,
\begin{align}
	\norm{U}_{\cV_j (\ell)} = L^{jd} |u_\bulk| + \sum_{* \in \{ \bulk,\o,\x \}}  \norm{\sigma_* V_*}_{\cV_j (\ell)} +  \norm{\sigma_\ox U_\ox}_{\cL_j (\ell)}   .
\end{align}
\end{definition}

\begin{lemma} \cite[Lemma~4.5 and 7.1]{FSmap} \label{lemma:VincDsize}
Following relations hold.  
\begin{itemize}
\item If $V \in \cV$,  then $\norm{V}_{\cL_j (\ell)} \asymp \max_{b \in \cB_j} \norm{V(b)}_{\ell_j, T_j (0)}$.
\item If instead $V \in \cD_j (\alpha)$ for $\alpha \le 1$,  we have $\norm{V_j}_{\cV_j (\ell)} \lesssim \ell_0^4 \tilde{g}_j \scale_j$.
\end{itemize}
\end{lemma}

\begin{lemma} \cite[Lemma~4.7]{FSmap} 
\label{lemma:cVnrmvslnrm}
For $V \in \cV_j$,
\begin{align}
	\norm{V}_{\cL_j (\ell)}
		\lesssim \norm{V}_{\cV_j (\ell) }
		& \lesssim \scale_j^{-1+\kt} \norm{V}_{\cL_j (\ell)}
		\label{eq:cVnrmvslnrm1} , \\
	\norm{V - \E_{j+1} \theta V}_{\cV_{j} (\ell)}
		& \lesssim \ell_0^{-2}  \tilde\chi_j \norm{V}_{\cV_j (\ell)}
		\label{eq:cVnrmvslnrm2} 
\end{align}
and
\begin{align}
	\norm{V}_{\cL_{j+1} (\ell) } \le L^2 \norm{V}_{\cL_j (\ell) } , \qquad 	
	\norm{V}_{\cV_{j+1} (\ell)} \le L^2 \norm{V}_{\cV_j (\ell)}  .
		\label{eq:cVnrmvslnrm3} 
\end{align}
\end{lemma}

\begin{lemma} \cite[Lemma~7.18]{FSmap} \label{lemma:EthV12}
If $V$ is a local monomial of degree $\le k$, 
then for $\kh \ge \kc_{j+1}$,
\begin{align}
\label{eq:EthV1}
		\norm{\E_{j+1} (\theta V - V ) (b)}_{\kh,  T_j(0)}
		& \lesssim \Big( \frac{\kc_+}{\mathfrak{h}} \Big) \norm{V (b)}_{\kh,  T_j (0)}
		, \\
\label{eq:EthV13}	
		\norm{\Cov_{j+1} [ V (b) ; V(b') ] }_{\kh,  T_j(0)}
		& \lesssim  \Big( \frac{\kc_+}{\mathfrak{h}} \Big)^{2} \norm{V (b)}_{\kh,  T_j (0)}^2
\end{align}
with the constants only depending on $k$.
\end{lemma}

For polynomials of $(\varphi_x)_{x \in \Lambda}$,  we can bound $\norm{\cdot}_{\kh, T_j (\varphi)}$ just from $\norm{\cdot}_{\kh, T_j (0)}$.

\begin{lemma} \cite[Proposition~3.10]{BBS1} 
\label{lemma:polynorm}
If $F (\varphi)$ is a polynomial of degree $A \ge 0$ and order of derivatives $\le p_\Phi$, then
\begin{align}
	\norm{F}_{\kh,  T_j (\varphi) } \le \norm{F}_{\kh,  T_j (0) } \big( 1 + \norm{\varphi}_{\kh , \Phi_j } \big)^A   .
\end{align}
\end{lemma}

\subsection{Stabilisation and $W_{j,V}$}
\label{sec:WIcoords}

Given $\cM >0$,  we define $\pexp (x) = \sum_{k=0}^\cM x^k / k!$.
For $V \in \cV$ and $b \in \cB_j$,  let
\begin{align}
	V^{(1)} = (1-\pi_{4,\nabla}) V , \qquad V^{(2)} = \pi_{4,\nabla} V 
\end{align}
(recall $\pi_{4,\nabla}$ from Section~\ref{sec:indfhot})
and for $X \in \cP_j$,  define the stabilisation of $e^{-V}$ as
\begin{align}
	e^{-V^{(\bs_j)} (X)} = e^{-V^{(1)} (X)} \prod_{b \in \cB_j (X)} \pexp (-V^{(2)} (b)) .
\end{align}
This enables the quartic terms except $|\varphi|^4$ stay away from the exponential,  and this guarantees integrability.
We also define $W_{j,V} (X) = \sum_{x \in X} W_{j,V,x}$ as in \cite[Definition~4.10]{FSmap}.  It is given as an explicit quadratic function of $V$ and scale $j$,  and $W_{j,Vx} (\varphi)$ is a polynomial of degree $\le 6$ in $\varphi$ that only depends on $\{ \varphi_y : |x-y| \le L^j \}$. 
Its motivation is explained in \cite[Section~4.4]{FSmap}.
They define the coordinate $I_j$ via
\begin{align}
	I_j (b) = \cI_j (V_j) (b) = e^{-V^{(\bs_j)} (b)} (1 + W_{j,V} (b)) ,
	\qquad
	I_j^X = \prod_{b \in \cB_j (X)} I_j (b)
	\label{eq:Ijdefi}
\end{align}
for $b \in \cB_j$.
The specific definition of $W$ will not be needed,  but we will only need a bound Lemma~\ref{lemma:FWPbounds3obs} and its projection $\pi_\ox W$ stated later in Definition~\ref{defi:pioxWP}.

\section{RG map}
\label{sec:RGmap}

To control the sequence $(Z_j)_{j}$ generated by progressive integrals \eqref{eq:Zjinductive},  we expanded $Z_j$ via \eqref{eq:Zjpolyexp},  and we will now consider the RG map $\Phi_{j+1}$ that relates the polymer expansions at scale $j$ and $j+1$.  
We summarise the main results on the RG map constructed in \cite[Section~5--10]{FSmap}.  

Generally put,  an RG map at scale $j$ is a map
\begin{align}
	\Phi_{j+1} = (\Phi_{j+1}^U, \Phi_{j+1}^K) : (V_j,K_j) \mapsto (\delta u_{j+1},  V_{j+1},  K_{j+1}) 
\end{align}
for $(V_j, V_{j+1} ) \in \cD_j^{(0)} \times \cD_{j+1}^{(0)}$ (recall \eqref{eq:cDzerodefi}),  $(K_j , K_{j+1} ) \in \cK_j \times \cK_{j+1}$ and $\delta u_{j+1} \in (\R + \sigma_\ox \R)^{\cB_{j+1}}$ such that
\begin{align}
	\E_{j+1} \theta [  ( I_j \circ_j K_j  ) (\Lambda)] = e^{-\delta u_{j+1} (\Lambda)} (I_{j+1} \circ_{j+1} K_{j+1}) (\Lambda)
	\label{eq:contrlldRGalg}
\end{align}
when $I_j = \cI_j (V_j)$ and $I_{j+1} = \cI_{j+1} (V_{j+1})$.  
A specific construction was given in \cite[Section~5]{FSmap} and its leading contribution is summarised by the perturbative map $\Phi_{j+1}^{\pt} (V)$,  Definition~\ref{def:WP}. 
We will not state the specific definition of the RG map,  but the bounds on the error terms $\Phi_{j+1}^K$ and
\begin{align}
	R_{j+1}^U (V,K) = \Phi_{j+1}^U (V,K) - \mathbb{V}^{(0)} \Phi_{j+1}^\pt (V) 
	\label{eq:RjUdefi}
\end{align}
will be stated in Section~\ref{sec:controlledRGmap}.

\subsection{Definition of $\Phi^U_{j+1}$}
\label{sec:bbonPhipt}

We define the RG map on the effective potential is defined in terms of the perturbative RG map. 
The perturbative RG map is a quadratic function that is defined on the whole space of effective potentials (not just $\cD_j$). 

\begin{definition} \label{def:WP} 
\cite[Definition~4.13,  (4.68)]{FSmap}
The \emph{pertubative RG map} for $V_j \in \cV$ is
\begin{align}
	\Phi_{j+1}^\pt : V_j \mapsto \E_{j+1} \theta V_{j} - P_{j,V}  ,
	\label{eq:PhiptUdf}
\end{align}
where $P_{j,V}$ is a polymer function of degree $\le 6$ that is a quadratic form of $V$ as defined in \cite[(4.63)]{FSmap} and satisfies bound Lemma~\ref{lemma:FWPbounds3obs}.  We also denote
\begin{align}
	U_{j+1}^{\pt} = \Phi_{j+1}^\pt (V_j)  ,\qquad V_{j+1}^{\pt} = (1- \pi_0 - \pi_\ox) U_{j+1}^{\pt} 
\end{align}
(recall Section~\ref{sec:effpotes} for $\pi_0$ and $\pi_{\ox}$).

For $Q_j \in \cV$ defined as in \cite[(7.80)]{FSmap} (as a function of $(V_j, K_j)$),  let $\hat{V}_j = V_j - Q_j$ and
\begin{align}
	\Phi_{j+1}^U (V_j, K_j) = \mathbb{V}^{(0)} \Phi_{j+1}^\pt (\hat{V}_j ) .
\end{align}
\end{definition}

As for the case of $W_j$,  we do not need specific definitions of $P_{j,V}$ and $Q_j$,  but we need some their properties. 
Bound on $P_{j,V}$ and $Q_j$ are stated in Lemma~\ref{lemma:FWPbounds3obs} and \ref{lemma:Qjbound},  respectively,  and the explicit form of $\pi_\ox P_{j,V}$ is given in Definition~\ref{defi:pioxWP}.

\begin{lemma}
\label{lemma:FWPbounds3obs}
For $V, V' \in \cV$,  there are bilinear functions $W_j (V_x,  V'_y)$ and $P_j (V_x,  V'_y)$ that are polynomials of $\varphi$ and satisfy
\begin{align}
	W_{j,V,x} = \sum_{y \in \Lambda} W_j (V_x,  V_y),  \qquad P_{j,V,x} = \sum_{y \in \Lambda} P_j (V_x,  V_y)
\end{align}

Also,  for $b \in \cB_j$,   $\kh \in \{ \ell, h \}$ and sufficiently large $L$,
\begin{align}
	\max\left\{
		\begin{array}{r}
		\sum_{x \in b,  \,  y \in \Lambda} \norm{W_j (V_x ; V'_{y})}_{\kh_j ,  T_j (0)}  \\
		\sum_{x \in b,  \,  y \in \Lambda} \norm{P_j (V_x ; V'_{y})}_{\kh_j ,  T_j (0)}
		\end{array}
		\right\}
		\le O_L (1) \tilde{\chi}_j \Big( \frac{\ell_{\bulk,j}}{\kh_{\bulk,j}} \Big)^6 \norm{V}_{\cV_j (\ell)} \norm{V'}_{\cV_j (\ell)}
		\nonumber
\end{align}
\end{lemma}
\begin{proof}
The bound on $W_{j}$ is \cite[Lemma~7.5]{FSmap}. 
For the bound on $P_{j}$,  notice that $P_{j}$ is a linear combination of $W_{j}$ and $\Cov_{\pi,+}$ (defined in \cite[(4.56)]{FSmap}) in \cite[(4.63)]{FSmap},  and $\Loc_x$ in the reference is a bounded linear operator due to \cite[Proposition~3.4]{FSmap}.
Now,  $\Cov_{\pi,+} (V_x, V_y)$ is bounded by \eqref{eq:EthV13},  so we have the desired bound.
\end{proof}

When we assume $V, V' \in \cD_j$,  then Lemma~\ref{lemma:VincDsize} and \ref{lemma:FWPbounds3obs} yield
\begin{align}
	& \sum_{x \in b,  \,  y \in \Lambda} \norm{W_{j,V} (V_x ; V'_{y})}_{\kh_j ,  T_j (0)} \le O_L (1) \tilde{\chi}_j \times \begin{cases}
	 	\tilde{g}_j^{2} \scale_j^2 & (\kh = \ell) \\
		\tilde{g}_j^{1/2} \scale_j^{1/2} & (\kh = h) .
	 \end{cases}
	 \label{eq:FWPbounds3obsapplied}
\end{align}
This immediately implies an estimate on the deviation of $\Phi^\pt_{j+1}$ from $\mathbb{E}_{j+1} \theta$.

\begin{lemma}  \label{lemma:VptmthV}

If $V \in \cV$,  then
\begin{align}
	\norm{ \Phipt_{j+1} (V) - \E_{j+1} \theta V}_{\cL_{j+1} (\ell)}
		& \le O_L (1) \tilde\chi_{j+1}  \norm{V}^2_{\cV_{j} (\ell)}
		\label{eq:VptmthV1}
		\\
	\norm{\Phipt_{j+1} (V) - \E_{j+1} \theta V}_{\cV_{j+1} (\ell)}
		& \le O_L (1) \tilde\chi_{j+1} \scale_{j+1}^{-1 + \kt} \norm{V}^2_{\cV_{j} (\ell)} .  \label{eq:VptmthV2}
\end{align}
and for $\norm{D_V F} = \sup\{ \norm{D_V F(\dot{V})} : \norm{\dot{V}}_{\cV_j (\ell)} \le 1 \}$,
\begin{align}
	\norm{D_V ( \Phipt_{j+1} (V) - \E_{j+1} \theta V ) }_{\cL_{j+1} (\ell)}
		& \le O_L (1) \tilde\chi_{j+1}  \norm{V}_{\cV_{j} (\ell)}
		\label{eq:VptmthV3}
		\\
	\norm{D_V ( \Phipt_{j+1} (V) - \E_{j+1} \theta V  ) }_{\cV_{j+1} (\ell)}
		& \le O_L (1) \tilde\chi_{j+1} \scale_{j+1}^{-1 + \kt} \norm{V}_{\cV_{j} (\ell)} .
		\label{eq:VptmthV4}
\end{align}
\end{lemma}
\begin{proof}
We will only show \eqref{eq:VptmthV1} and \eqref{eq:VptmthV3}
as the other two follow from these by \eqref{eq:cVnrmvslnrm1}.

Recalling Definition~\ref{def:WP},  $\Phipt_{j+1} (V)_x - \E_{j+1} \theta V_x = -P_{j,V,x}$, we just need to bound $P_{j,V}$.
But by Lemma~\ref{lemma:FWPbounds3obs},  since $P_{j,V}$ is a quadratic function of $V$,
\begin{align}
	\norm{P_{j,V} (B) }_{\ell_{j+1} , T_{j+1} (0)}
		& \le 
		O_L (1) \chi_{j+1} \norm{V}_{\cV_{j} (\ell)}^2 \\
	\norm{D_V P_{j,V} (B)}_{\ell_{j+1} , T_{j+1} (0)}
		& \le 
		O_L (1) \chi_{j+1} \norm{V}_{\cV_{j} (\ell)} 
\end{align}
for $B \in \cB_{j+1}$,  as desired.
\end{proof}

Finally,  we need an estimate on $Q_j$ to control the full RG map on the effective potential.

\begin{lemma} \cite[(7.87)]{FSmap} 
\label{lemma:Qjbound}
For $(V_j, K_j) \in \mathbb{D}_j (\alpha)$ ($\alpha \le 1$),  we have
\begin{align}
	\norm{Q_j }_{\cL_j (\ell)} \le O_L (1) \tilde{\chi}_j^{3/2} \tilde{g}_j^3 \scale_j^{\kae}
\end{align}
\end{lemma}

\subsection{Controlled RG map}
\label{sec:controlledRGmap}

We assume the following estimate on the RG map.
We make the base lattice $\Lambda$ explicit in the following definition by denoting $\Phi_{j+1}^{\Lambda}$ and $\cK^{\Lambda}_j$.

\begin{definition} \cite[Definition~1.6]{FSmap}
 \label{defi:contrlldRG}
Let $j < N < \infty$ and $\Lambda = \Lambda_N$.  \emph{Controlled RG map} at scale $j$ is a function
\begin{align}
\begin{split}
	\Phi^{\Lambda}_{j+1} = (\Phi^U_{j+1} , \Phi^K_{j+1}) : \cD^{(0)}_j \times \cK_j^{\Lambda} \times \AA_j (\tilde{m}^2) 
		& \rightarrow \left( (\R + \sigma_\ox \R)^{\Lambda} \times \cV^{(0)} \right) \times \R\cK^{\Lambda}_{j+1} ,  \\
	(V_j,K_j) 
		& \mapsto \big( (\delta u_{j+1} , V_{j+1} ), K_{j+1} \big) ,
\end{split}	
\end{align}
such that \eqref{eq:contrlldRGalg} holds
when $I_{j'} = \cI_{j'} (V_{j'})$ for each $j' \in \{ j,  j+1 \}$,
and bounds \eqref{eq:controlledRG22}--\eqref{eq:controlledRG24} hold
for some $j,N$-independent,  $L$-dependent constants $(M_{p,q} )_{p,q\ge 0}$:
if $R_{j+1}^U$ is defined by \eqref{eq:RjUdefi},
\begin{align}
	\norm{D^p_{V_\bulk} D^q_K R_{j+1}^{U} }_{\ell_{j+1},  T_{j+1} (0)}
		& \le M_{p,q} \times \begin{cases}
			\tilde\chi_{j+1}^{3/2} \tilde{g}_{j+1}^{3} \scale_{j+1}^{\kae-(1-\kt) p} & (p\ge 0,  \; q =0) \\
			\scale_{j+1}^{-(1-\kt) p}  & (p \ge 0,  \; q=1)  	\\
			\scale_{j+1}^{-2(1-\kt)}  & (p \ge 0,  \; q=2) \\
			0   &  (p \ge 0,  \; q \ge 3) 
			,
		\end{cases} \label{eq:controlledRG22} \\
	\norm{D_{V_\bulk}^p D_K^q \Phi_{j+1}^{K}}_{\cW_{j+1}}
		& \le M_{p,q} \times 
			\begin{cases}
			\tilde\chi_{j+1}^{3/2} \tilde{g}_{j+1}^{3-p} \scale_{j+1}^{\kae - p}  & (p \ge 0,  \; q= 0)  	\\
			\tilde{g}_{j+1}^{-p - \frac{9}{4} (q-1)} \scale_{j+1}^{-p - \kbe (q-1)} & (p \ge 0,  \; q \ge 1) 
		\end{cases} \label{eq:controlledRG23}
\end{align}
and when $j+1 < N$,  with the same $C_{\rg}$ as in Lemma~\ref{lemma:KNbnd},
\begin{align}
	\norm{D^q_K \Phi_{j+1}^{K}}_{\cW_{j+1}}
		\le 
		\begin{cases}		
			C_{\rg} \tilde\chi_{j+1}^{3/2} \tilde{g}_{j+1}^{3} \scale_{j+1}^{\kae} & (q = 0) \\
			\frac{1}{32} L^{-\max\{ 1/2 ,  (d-4 +2\eta) \kae \}}  & (q = 1)  .
		\end{cases} 
		\label{eq:controlledRG24}
\end{align}
Moreover,  $D^p_{V_\bulk} D^q_K R_{j+1}^{U}$ and $D_{V_\bulk}^p D_K^q \Phi_{j+1}^{K}$ are continuous in $(\ba_\emptyset, \ba) \in \AA_j (\tilde{m}^2)$.
\end{definition}

\begin{definition} \label{defi:contrlldRG2}
Let $\Phi^{\Lambda}_{j+1}$ be a controlled RG map at scale $j$.  It is said to \emph{respect the graded structure} if
$\tilde{\pi} \circ \Phi^{\Lambda}_{j+1} = \Phi^{\Lambda}_{j+1} \circ \tilde{\pi}$ for each $\tilde{\pi} \in \{ \pi_\bulk,  \pi_\bulk + \pi_\o ,  \pi_\bulk + \pi_\x \}$.
\end{definition}

\begin{theorem} \cite[Theorem~1.4]{FSmap}
\label{thm:contrlldRG}
Assume $\eta \in [0,1/2)$,  $d \ge d_{c,u} = 4-2\eta$,  let $L$ be sufficiently large and $\Gamma_{j+1}$ be as in Proposition~\ref{prop:theFRD}.
Then a controlled RG map exists at any scale $j < N < \infty$.
Also,  the RG map respects the graded structure.
\end{theorem}

The RG map can also be extended to the infinite lattice in the following sense.
To compare polymer activities on $\Z^d$ and $\Lambda_N$,  we let $c_N : \Z^d \rightarrow \Lambda_N$ be a local isometry such that $c_N (\o) = \o$ and $c_N (\x) = \x$.  We say that $K_j \in \cK_j^{\Z^d}$ projects to $K'_j \in \cK_j^{\Lambda_N}$ if 
\begin{align}
	K_j (X,  \varphi \circ c_N ) = K'_j ( c_N (X),  \varphi)
\end{align}
whenever $X$ is contained in a hypercube of sidelength $L^{N-1}$.

\begin{theorem} \cite[Theorem~A.4]{FSmap}
\label{thm:infvolRGmap}
Under the assumptions of Theorem~\ref{thm:contrlldRG},  for any $j\ge 0$,  there exists a map
\begin{align}
	\Phi_{j+1}^{\Z^d} = (\Phi^U_{j+1} , \Phi^K_{j+1}) : \cD^{(0)}_j \times \cK^{\Z^d}_j \times \AA_j (\tilde{m}^2) 
		\rightarrow \left( (\R + \sigma_\ox \R)^{\Lambda} \times \cV^{(0)} \right) \times \R\cK^{\Z^d}_{j+1}
\end{align}
that satisfies \eqref{eq:controlledRG22}--\eqref{eq:controlledRG24} and its restriction to finite volume is identical to $\Phi_{j+1}^{\Lambda_N}$ in the following sense: 
if $j +1 < N$ and $K_j \in \cK^{\Z^d}_j$ projects to $K'_j  \in \cK^{\Lambda_N}_j$, 
and let $(U_{j+1},K_{j+1}) = \Phi^{\Z^d}_{j+1} (U_j, K_j)$ and $(U'_{j+1},K'_{j+1}) = \Phi^{\Lambda_N}_{j+1} (U_j, K'_j)$,  then $K_{j+1}$ also projects to $K'_{j+1}$ and $U_{j+1} = U'_{j+1}$.
\end{theorem}

We will need the infinite volume RG map $\Phi_{j+1}^{\Z^d}$ in Section~\ref{sec:cotcpfdf} and \ref{sec:stabobsflow} to construct the critical point and prove the stability of the RG flow.  
On the other hand,  proof of the main theorems are all based on the finite volume RG map,  and we use the projection property in Theorem~\ref{thm:infvolRGmap} to obtain the stability of the finite volume RG flow.

\section{Critical point for $d> d_{c,u}$}
\label{sec:cotcpfdf}

In this section,  we construct the critical initial value of the RG flow when $d > 4$ for $\eta =0$ and $d \ge 4$ for $\eta \in (0,1/2)$,
i.e.,  choice of $V_0$ such that the flow generated by $\Phi_{j+1}$ satisfies $\norm{V_j}_{\cV_j (\ell)},  \norm{K_j}_{\cW_j} \rightarrow 0$ as $j\rightarrow \infty$.  Thus for sufficiently large $N$,  the polymer expansion \eqref{eq:Zjpolyexp} can be approximated by $Z_N \approx \exp(-u_N (\Lambda))$.
References for the construction when $d= d_{c,u}$ is given in Appendix~\ref{sec:critdRGflow}.

Construction of the critical point only requires the bulk part of the RG flow,  so we take $\lambda_\o = \lambda_\x \equiv 0$ in this section. and we always have $(U,K) = (U_\bulk, K_{\bulk})$.
The discussion on the stability of the observable part is deferred to Section~\ref{sec:stabobsflow}.

We will first construct the stable manifold of the dynamical system generated by the perturbative maps $(\Phi_{j+1}^{\pt})_{j\ge 0}$ in Section~\ref{sec:cotpcp},  and then extend the construction to the full RG map by interpolating with the perturbative flow using an ODE in Section~\ref{sec:cotcpiibRm}.
Then the critical point is obtained in Section~\ref{sec:cpnt},  using elementary topological arguments.  These results are summarised in Theorem~\ref{thm:theStableManifold}.
As before,  we use $\vec{\nu} = (\nu^{(\km_1)})_{\km_1 \in \ko_2 \cup \ko_{2,\nabla}}$,  $\vec{g} = (g^{(\km_2)})_{\km_2 \in \ko_4 \cup \ko_{4,\nabla}}$ 
and $\vec{\lambda} = (\lambda^{(\km_3)})_{\km_3 \in \ko_1 \cup \ko_{1,\nabla}}$ to denote the coefficients of $V$.

\begin{definition} \label{def:RGflow}
Let  $j \le N$ on $\Lambda = \Lambda_N$,  $j< \infty$ on $\Lambda = \Z^d$ and $\alpha \in [1/2,1]$.
\begin{enumerate}
\item Sequence $( \vec{\nu}_k , \vec{g}_k ,  \vec{\lambda}_k ,  K_k)_{k \le j} \in \prod_{k =0}^j \D_j (\alpha)$ is called the \emph{RG flow of length $j$} if there exist $(\delta u_{k})_{k=1}^j$ with
\begin{align}
	(\delta u_{k+1} ,   \vec{\nu}_{k+1} , \vec{g}_{k+1},  \vec{\lambda}_{k+1} ,  K_{k+1} ) = \Phi^{\Lambda}_{k+1} ( \vec{\nu}_k , \vec{g}_k ,  \vec{\lambda}_k ,  K_k )
	\quad \text{for all} \quad k < j
	.
\end{align}
The sequence is called the \emph{infinite RG flow} if the same holds with $j = \infty$.

We use the term \emph{bulk RG flow} for a sequence $( \vec{\nu}_k , \vec{g}_k, K_{\bulk,k})_{k \le j} \in \prod_{k=0}^j \D_{\bulk,k} (\alpha)$ with the same property.

\item For $j< \infty$,  sequence $( \vec{\nu}_k , \vec{g}_k )_{k \le j} \in \prod_{k =0}^j \cD_{\bulk,k}^{(0)} (\alpha)$ is called the \emph{bulk perturbative RG flow of length $j$} if there exist $(\delta u_{k})_{k=1}^j$ with
\begin{align}
	(\delta u_{k+1} ,   \vec{\nu}_{k+1} , \vec{g}_{k+1}  ) = \mathbb{V}^{(0)} \Phi_{k+1}^{\pt} ( \vec{\nu}_k , \vec{g}_k )
	\quad \text{for all} \quad k < j .
\end{align}
\end{enumerate}
\end{definition}

\begin{theorem} \label{thm:theStableManifold}
Let $\alpha =1$,  $g$ be sufficiently small and $L$ be sufficiently large.

Then there exist convergent sequences 
\begin{align}
	( \ba^{(\emptyset)}_{c,k} ,  \nu^{(\emptyset)}_{c,k} )_{k \in \N \cup \{\infty\}} \subset  \HB_{\epsilon_p} \times \R
\end{align}
such that $\ba_{c,\infty}^{(\emptyset)} =0$ and satisfies the following stability condition:
whenever $\vec{\nu}, \vec{\ba} \in \R^{\kA_0 \cup \kA_1}$ are given by
\begin{align}
\begin{cases}
	\nu^{(\emptyset)} =  \nu_{c,k}^{(\emptyset)}  \text{ and } \nu^{(\km)} = - \ba_{c,k}^{(\km)} \text{ for each } \km \in \kA_1 \\
	\vec{\ba} = \vec{\ba}_{c,k}
\end{cases}
\end{align}
the bulk RG flow of infinite length exists with $\vec{\ba}$ and initial condition
\begin{align}
	\vec{\nu}_0 = \vec{\nu},  \quad g^{(\emptyset)}_0 = g,  \quad g^{(\km_2)} =0 \;\; (\km_2 \in \ko_{4,\nabla}),  \quad K_0 \equiv 0 .
\end{align}
Moreover,  $\nu^{(\emptyset)}_{c,\infty},  \ba^{(\km)}_{c,\infty} = O(g)$ for each $\km \in \kA_1$.   
\end{theorem}

\begin{definition} \label{defi:nuc}
The choice of $\nu_{c,\infty}^{(\emptyset)}$ is called the critical point and also denote
\begin{align}
	\nu_c = \nu^{(\emptyset)}_{c,\infty}  . 
\end{align}
\end{definition}

$\nu_c$ serves as the point of phase transition of the $|\varphi|^4$ model used in Theorem~\ref{thm:infvol2pt}--\ref{thm:NGlimit},  and $\vec{\ba}_{c,\infty}$ serves as the collection of coefficients of counterterms.

\subsection{More projections}

Recall that $\pi_2, \pi_4, \pi_{2,\nabla}$ and $\pi_{4,\nabla}$ are projections on $\cV_2,\cV_4,\cV_{2,\nabla}$ and $\cV_{4,\nabla}$,  respectively.  Also,  we use $\vec{\nu} = \{ \nu^{(\km_1)} \}_{\km_1 \in \ko_2 \cup \ko_{2,\nabla}}$ and $\vec{g} = \{ g^{(\km_2)} \}_{\km_2 \in \ko_4 \cup \ko_{4,\nabla}}$.
But we will need further decomposition of the projections.
Their existences are guaranteed by \cite[Remark~4.1]{FSmap}.

\begin{definition}
\label{def:bbphipt}
On $\cV_{\bulk}$ we define projections
\begin{align}
	p^{(\km_1)}
		& : (\vec{\nu}, \vec{g})  \mapsto \nu^{(\km_1)} ,  \qquad
	q^{(\km_2)} : (\vec{\nu}, \vec{g})  \mapsto g^{(\km_2)}  
\end{align}
for $\km_1 \in \ko_2 \cup \ko_{2,\nabla}$ and $\km_2 \in \ko_4 \cup \ko_{4,\nabla}$.
For subsets $A \subset \{ 0,1,2,3 \}$ and $B \subset \{0,1\}$,  we define projections $p^{A}$ and $q^{B}$ by 
\begin{align}
	p^{A}
		& : (\vec{\nu}, \vec{g})  \mapsto \vec{\nu}^{A} := (\nu^{(\km_1)} )_{\km_1 \in \cup_{i \in A} \kA_i }, \\
	q^{B} 
		&: (\vec{\nu}, \vec{g} ) \mapsto 
		\vec{g}^{B} :=
	 	(g^{(\km_2)} )_{\km_2 \in \cup_{i \in B} \kB_i  }  .
\end{align}
Also,  we denote for $V \in \cV_\bulk$
\begin{align}
	V^{(\km)} = \begin{cases}
		 p^{(\km_1)} V & (\km_1 \in \ko_2 \cup \ko_{2,\nabla}) \\
		 q^{(\km_2)} V & (\km_2 \in \ko_4 \cup \ko_{4,\nabla}) ,
	\end{cases}   \qquad 
	\cV^{(\km)} = \begin{cases}
		 p^{(\km_1)} \cV_\bulk & (\km_1 \in \ko_2 \cup \ko_{2,\nabla}) \\
		 q^{(\km_2)} \cV_\bulk & (\km_2 \in \ko_4 \cup \ko_{4,\nabla}) 
	\end{cases} 
\end{align}
\end{definition}

For example,  when
\begin{align}
	V_x (\varphi) = \frac{1}{2} \nu^{(\emptyset)} |\varphi_x|^2 + \frac{1}{4} g^{(\emptyset)} |\varphi_x|^4 + \frac{1}{2} \sum_{\km_1 \in \ko_{2,\nabla}} \nu^{(\km_1)} S^{(\km_1)}_{x} (\varphi) +  \frac{1}{4} \sum_{\km_2 \in \ko_{4,\nabla}} g^{(\km_2)} S^{(\km_2)}_{x} (\varphi) ,
\end{align}
we may write
\begin{align}
	(p^{0,1} V )_x (\varphi) = \frac{1}{2} \nu^{(\emptyset)} |\varphi_x|^2 + \frac{1}{2} \sum_{\km_1 \in \kA_1} \nu^{(\km_1)} S_x^{(\km_1)} (\varphi)
	.
\end{align}
They satisfy $\pi_2 = p^0$,  $\pi_{2,\nabla} = p^{1,2,3},  \pi_4 = q^0$,  $\pi_{4,\nabla} = q^{1}$.

\subsection{Stable manifold I.  perturbative map}
\label{sec:cotpcp}

Construction of the critical point only requires the infinite volume limit RG map $\Phi_{j+1}^{\Z^d}$,  so we  just write $\Phi_{j+1}$ in the rest of Section~\ref{sec:cotcpfdf}.
Just for the construction of stable manifolds,  we need new spaces of $(V_j,K_j)$.

\begin{definition} \label{eq:spaceXYZdefi}
For $V \in \cV_{\bulk}$ and $K \in \cK_{j,\bulk}$,  we define norms
\begin{align}
	\norm{V}'_{j} = \scale_j^{-1 + \kt} \norm{V}_{\cV_j (\ell)} , 
	\qquad \norm{K}''_{j} = \scale_j^{- \kae + \kt} \norm{K}_{\cW_j}
\end{align}
and Banach spaces
\begin{align}
	& X_j = \Big\{ \sum_{\km \in \kA_0 \cup \kA_1} V^{(\km)} (\vec{\ba}) : V^{(\km)} \in C \big( \HB_{\epsilon_p} ; \cV^{(\km)} \big) \Big\} , \\
	& Y_j = \Big\{ \sum_{\km \in \kA_2 \cup \kA_3 \cup \kB_0 \cup \kB_1}  V^{(\km)} (\vec{\ba}) : V^{(\km)} \in C \big( \HB_{\epsilon_p} ; \cV^{(\km)} \big) \Big\}  ,\\
	& Z_j = \Big\{ K (\vec{\ba}) : K \in C \big( \HB_{\epsilon_p} ; \R \cK_j \big)  \Big\}   ,
\end{align}
i.e.,  they are spaces of continuous function of $\vec{\ba} \in \HB_{\epsilon_p}$.
They are endowed with norms
\begin{align}
\begin{split}
	\norm{x_j}_{X_j} = \sup_{\vec{\ba} \in \HB_{\epsilon_p}} \norm{x_j (\vec{\ba})}'_j,  \qquad \norm{y_j}_{Y_j} &= w_Y \sup_{\vec{\ba} \in \HB_{\epsilon_p}} \norm{y_j (\vec{\ba})}'_j , \\ 
	\norm{z_j}_{Z_j} &= w_Z \sup_{\vec{\ba} \in \HB_{\epsilon_p}}  \norm{z_j (\vec{\ba})}''_j  
\end{split}
\end{align}
for some constants $w_Y, w_Z >0$,  that are chosen sufficiently large later in the proof of Proposition~\ref{prop:Stdefi} and \ref{prop:nupcctty}.

Projection on each space is denoted $\proj_X,\proj_Y$ and $\proj_Z$,  respectively. 
\end{definition}

We first need to construct the fixed point of the perturbative flow.  We recall the bulk perturbative RG flow from Definition~\ref{def:RGflow}(ii),
and $C_{\cD}$ was in the definition of $\cD_j$.

\begin{proposition}
\label{prop:nuptcctty}
Let $\alpha \in [1/2,1]$,  $g$ be sufficiently small,  $\vec{\ba} \in \HB_{\epsilon_p}$ and $C_{\cD}$ be sufficiently large (independently of $L$).  Then there exists $ \vec\nu_{ptc}  = (\nu^{(\km_1)}_{ptc} )_{\km_1 \in  \kA_0 \cup \kA_1 } $ such that the following hold: with initial condition \begin{align}
	\begin{cases}
		\nu^{(\km_1)}_0  = \nu_{ptc}^{(\km_1)} & (\km_1 \in \kA_0 \cup \kA_1 ) \\
		\nu^{(\km_1)}_0 = g^{(\km_2)}_0  = 0 & (\km_1 \in \kA_2 \cup \kA_3   ,  \; \km_2 \in \kB_1) \\
		g_0^{(\emptyset)} = g  , & {}
	\end{cases}
	\label{eq:nupccttyIC}
\end{align}
the bulk perturbative RG flow of infinite length exists,  and $V_j  \in \cD_{j,\bulk} (\alpha)$ for each $j$.  Moreover,  $\vec{\nu}_{ptc}$ is a continuous function of $\vec{\ba}$ and differentiable function of $g$ satisfying $\vec{\nu}_{ptc} = O(g)$.
\end{proposition}

\begin{proof}
Let us denote $\bar{\Phi}^\pt_{j+1} = \mathbb{V}^{(0)} \Phi^\pt_{j+1}$.
(Recall that $\mathbb{V}^{(0)}$ transfers the constant part and transfers the coefficient of $\nabla \varphi \cdot \nabla \varphi$ to the coefficient of $-\varphi \cdot \Delta \varphi$.)
For any $\bmrho >0$,  $\bar\Phi^\pt_{j+1}$ can be considered as a map
\begin{align}
\begin{split}
	\bar\Phi^\pt_{j+1} \equiv (\bar\Phi_{j+1}^{\pt, X},  \bar\Phi_{j+1}^{\pt,Y} )  : B_{\bmrho} (X_j) \times  B_{\bmrho} (Y_j ) \rightarrow X_{j+1} \times Y_{j+1}  , \\
	(x_j, y_j ) \mapsto ( \proj_X \bar\Phi^{\pt}_{j+1} (x_j,y_j ),  \proj_Y \bar\Phi^{\pt}_{j+1} (x_j,y_j ) ) .
\end{split}
\end{align}
By expanding them in linear order,  we obtain
\begin{align}
	\bar\Phi_{j+1}^{\pt, X} (x_j, y_j) &= A_j x_j + B_j  y_j  + N^X_j (x_j, y_j) \\
	\bar\Phi_{j+1}^{\pt, Y} (x_j, y_j) &=  C_j y_j  +  N^Y_j (x_j, y_j ), 
\end{align}	
where $N^{S}_j (0,0) = D N^S_j (0,0) = 0$ for each $S \in \{ X,Y \}$ and 
\begin{align}
\begin{split}
	A_j x_j = \proj_X \E_{j+1} \theta x_j =  p^{0,1} \E_{j+1} \theta x_j,  & \quad 
	B_j y_j = \proj_X \E_{j+1} \theta y_j =  p^{0,1} \E_{j+1} \theta y_j,  \\
	C_j y_j = \proj_Y \E_{j+1} \theta y_j  &=  \big( p^{2,3} + q^{0,1} \big) \E_{j+1} \theta y_j 
\end{split}
\end{align}
The crucial observation is that the linear part is triangular,  so by a standard theory,  we only need to verify a few estimates on $A_j,B_j,C_j$ to construct the stable manifold: 
(1) $N_j^X$ and $N_j^Y$ are continuously differentiable uniformly in $j$,  (2) $A_j$ is invertible,  (3) $\sup_{j,k} \norm{A_j^{-1}} \norm{C_k} <1$,  (4) $\sup_j \norm{C_j} < L^{-(d-4+\eta) \kt}$ and (5) $\sup_j \norm{B_j} < \infty$. 
Then by \cite[Theorem~2.16]{MR2523458} (just with smoothness replaced by continuous differentiability),  there exists $\vec\nu_{ptc} (\vec{\ba})$ that is continuously differentiable in $g$,  $\vec{\ba} \in \HB_{\epsilon_p}$,  satisfies for each $j\ge 0$
\begin{align}
	(x_{j+1}, y_{j+1}) = (\bar\Phi_{j+1}^{\pt, X},  \bar\Phi_{j+1}^{\pt,Y} ) (x_j, y_j) 
\end{align}
and $\norm{x_j}_{X_j}, \norm{y_j}_{Y_j} \rightarrow 0$ exponentially as $j \rightarrow \infty$ when the initial condition is given by \eqref{eq:nupccttyIC}. 
In particular,  there exist $C' ,\mu > 0$ such that $\norm{(x_j, y_j)}_{X_j \times Y_j} \le C'  L^{-(d-4+\eta) \mu j }$.  Actually,  we can deduce more from the proof: the whole sequence $(x_j, y_j)_{j\ge 1}$ is a continuous differentiable function of $g$,  and we can take the rate of exponential decay to be given by any number smaller than the decay rate given by condition (4). 
In particular,  we can take any $\mu \in (\frac{1}{2}\kt ,  \kt)$ and get $\norm{(x_j, y_j)}_{X_j \times Y_j} \lesssim g \scale_j^{\mu}$,  or equivalently
\begin{align}
	\norm{x_j + y_j}_{\cV_j (\ell)} \lesssim g \scale_j^{1 -\kt + \mu } .
\end{align}
Reflecting on the condition $(\vec{\nu}_j , \vec{g}_j ) \in \cD_{j,\bulk}^{(0)} (\alpha)$ of Definition~\ref{def:RGflow}~(ii), 
these do not directly imply that $(V_j)_j$ form a bulk perturbative RG flow,  but we can use this as an apriori estimate to improve the bounds in Lemma~\ref{lemma:bootstrap} so that $\norm{x_j}_{\cV_j (\ell)} \lesssim \ell_0^2 g \scale_j$ and $\norm{y_j - y_0}_{\cV_j (\ell)} \le O_L (1) g^2 \scale_j$ for each $j\ge 0$,  thus $x_j + y_j \in \cD^{(0)}_{j,\bulk} (\alpha)$ when $C_{\cD}$ is sufficiently large and $g$ is sufficiently small. 

Verifications of (1)--(5) follow from simple calculations.
For (1),  we can use Lemma~\ref{lemma:cVnrmvslnrm},  \ref{lemma:VptmthV} and the fact that $N_j^X$ and $N_j^Y$ are quadratic forms to see that
\begin{align}
	\norm{D N_j^X}_{X_{j+1}} + \norm{D N_j^Y}_{Y_{j+1}} &\le O_L (\bmrho)
	\label{eq:NjXYbnds}
\end{align}
whenever $\norm{x_j}_{X_j},\norm{y_j}_{Y_j} \lesssim \bmrho$.
For (2),  we just need an observation that $A_j$ does not change the coefficients of monomials when we think of $x_j$ as an element of $\cV$.
For (3) and (4),  since $A_j$ and $C_j$ do not change the coefficients of monomials, 
\begin{align}
\begin{cases}
	\norm{p^{(\km_1)} A_j x_j}_{X_{j+1}} = L^{(d-2 + \eta - q(\km_1) )} \norm{p^{(\km_1)} x_j}_{X_j}  & (\km_1 \in \kA_0 \cup \kA_1) \\
	\norm{p^{(\km_1)} C_j y_j}_{Y_{j+1}} =  L^{-(d-4 + 2\eta) \kt} \norm{p^{(\km_1)} y_j}_{Y_j} & (\km_1 \in \kA_2) \\
	\norm{p^{(\km_1)} C_j y_j}_{Y_{j+1}} = L^{-2 (d-4 + 2\eta) \kt} \norm{p^{(\km_1)} y_j}_{Y_j} & (\km_1 \in \kA_3) \\	
	\norm{q^{(\km_2)} C_j y_j}_{Y_{j+1}} = L^{-(d-4 + 2\eta) \kt} \norm{q^{(\km_2)} y_j}_{Y_j} & (\km_2 \in \ko_4) \\
	\norm{q^{(\km_2)} C_j y_j}_{Y_{j+1}} =  L^{-2(d-4 + 2\eta) \kt }  \norm{q^{(\km_2)} y_j}_{Y_j} & (\km_2 \in \ko_{4,\nabla}) .
\end{cases}
	\label{eq:Phiptlinbnds} 
\end{align}
When $\km_1 \in \kA_0 \cup \kA_1$,  we have $q(\km_1) \le \lceil d-2+\eta \rceil - 1$,  so for $z (d,\eta) := d-2+\eta - \lceil d-2+\eta \rceil +1 >0$,
\begin{align}
	\norm{p^{(\km_1)} A_j x_j}_{X_{j+1}} \ge L^{z(d,\eta) } \norm{p^{(\km_1)} x_j}_{X_j}
\end{align}
Putting together the other cases,  we have $\norm{C_j}  \le L^{-(d-4 + \eta) \kt}$ and $\norm{A_j^{-1}} \norm{C_k}  < L^{-z(d,\eta) + (d-4+\eta)\kt } < 1$ when $\kt$ is chosen $\kt < z(d,\eta) / (d-4+\eta)$.
Finally,  for (5),  since $\norm{\Eplus \theta V}_{\cV_{j} (\ell)} \lesssim \norm{V}_{\cV_j (\ell)}$ again by Lemma~\ref{lemma:cVnrmvslnrm},  we see
\begin{align}
	\norm{B_j y_j}_{X_{j+1}} \le O_L (1) w_Y^{-1} \norm{y_j}_{Y_j} . \label{eq:Phiptlinbnds3} 	
\end{align}
\end{proof}

The following lemma is a bootstraping step used in the proof of Proposition~\ref{prop:nuptcctty}.

\begin{lemma}	\label{lemma:bootstrap} 
Let $g$ be sufficiently small,  $\mu \in ( \frac{\kt}{2}, \kt)$,  and suppose $(x_j, y_j)_{j\ge 0} \in X_j \times Y_j$ are such that
\begin{align}
	\norm{x_0}_{X_0} \lesssim \ell_0^2 g ,  \quad \norm{y_0}_{Y_0} \lesssim \ell_0^4 g  ,  \qquad
	\norm{x_j}_{X_j}, \norm{y_j}_{Y_j} \le C_1 g \scale_j^{\mu} \;\;	
	 \text{for some } C_1 > 0
\end{align}
and $(x_{j+1}, y_{j+1}) = (\bar\Phi_{j+1}^{\pt, X},  \bar\Phi_{j+1}^{\pt,Y} ) (x_j, y_j)$ for each $j \ge 0$.  Then for each $j\ge 0$,
\begin{align}
	\norm{x_j}_{\cV_j (\ell)} \lesssim \ell_0^2 g \scale_j , \qquad 
	\norm{y_j - y_0 }_{\cV_j (\ell)} \le O_L (1) g^2 \scale_j .
	\label{eq:bootstrap}
\end{align}
\end{lemma}

\begin{proof}
The assumptions imply,  for some $C_0 >0$
\begin{align}
	\norm{x_j + y_j}_{\cV_j (\ell)} \le C_1 g \scale_j^{1+ \mu - \kt}, \qquad \norm{x_0 + y_0}_{\cV_0 (\ell)} \le C_0 g  .
\end{align}
We first assume as an induction hypothesis that
\begin{align}
	 	\norm{ y_j}_{\cV_j (\ell)} \le C_0 g \scale_j \prod_{k < j} (1 + g^{1/2} \scale_k^{2 \mu - \kt} ) . \label{eq:bootstrapIH}
\end{align}
The bound trivially holds for $j =0$.  Since $y_j =  (p^{2,3} + q^{0,1}) \E_{j+1} \theta (x_j + y_j)$,  we can apply Lemma~\ref{lemma:VptmthV} to see that
\begin{align}
	\norm{y_{j+1} - y_j }_{\cV_{j+1} (\ell)} \le O_L (1) \scale^{-1 + \kt} \norm{x_j + y_j}_{\cV_j (\ell)}^2
		\le C g^2 \scale_j^{1+ 2\mu - \kt} ,  \label{eq:yjrecursive}
\end{align}
and since $\norm{y_j}_{\cV_{j+1} (\ell)} \le L^{-(d-4+2\eta)} \norm{y_j}_{\cV_{j} (\ell)}$,  \eqref{eq:yjrecursive} and \eqref{eq:bootstrapIH} imply
\begin{align}
	\norm{y_{j+1}}_{\cV_{j+1} (\ell)} &\le C_0 g \scale_{j+1} \prod_{k < j} (1 + g^{1/2} \scale_k^{2\mu-\kt}) + C g^2 \scale_j^{1+ 2\mu - \kt} \nnb
	&\le C_0 g \scale_{j+1} \prod_{k \le j} (1 + g^{1/2} \scale_k^{2\mu - \kt})
\end{align}
for sufficiently small $g$,  completing the induction.  Thus again by taking $g$ sufficiently small,  we obtain
\begin{align}
	\norm{y_j}_{\cV_j (\ell)} \le 2 C_0 g \scale_j .
	\label{eq:bootstrapy} 
\end{align}

Then we bound $x_j$'s,  which follow the recursive relation
\begin{align}
	x_{j+1} = x_j + p^{0,1} \E_{j+1} \theta y_j + N_j^X (x_j, y_j) .
\end{align}
If $(\nu_j^{(\km_1)})_{\km_1 \in \kA_0 \cup \kA_1} \in \R^{\kA_0 \cup \kA_1}$ are the coefficients of $x_j$,  the only possible solution with asymptotic condition $\nu_j^{(\km_1)} \rightarrow 0$ should satisfy
\begin{align}
	\frac{1}{2} \nu_j^{(\km_1)} S^{(\km_1)} = - \sum_{k \ge j} p^{(\km_1)}  \Big( \E_{k+1} \theta y_k + N_k^X (x_k, y_k) \Big) ,
\end{align}
and comparing the coefficients,
\begin{align}
	| \nu_j^{(\km_1)} | 
		&\le C \ell_0^2 \sum_{k\ge j} L^{(q(\km_1) - 2 + \eta) k} \norm{p^{(\km_1)} ( \E_{k+1} \theta y_k +  N_k^X (x_k, y_k) ) }_{\cL_{k}} \nnb
		&\le C' \sum_{k\ge j} L^{(q(\km_1) - 2 + \eta) k} (C_0 g \scale_k + C_3 g^2 \scale_k^{1+ 2 \mu- \kt} ) \nnb
		&\le 2 C' (2 C_0 + C_3 g) g \scale_j L^{(q(\km_1) - 2 + \eta) j}
		\label{eq:bootstrapnujkm1}
\end{align}
where in the second line,  the first term follows from Lemma~\ref{lemma:cVnrmvslnrm} and \eqref{eq:bootstrapy},
the final term follows from \eqref{eq:bootstrapy},
and $C'$ can be $L$-dependent.  
Also,  since $\norm{S^{(\km_1)}}_{\cL_j (\ell)} \lesssim \ell_0^2 L^{(2-\eta-q(\km_1)) j}$,  again by taking $g$ sufficiently small,  
\begin{align}
	\norm{x_j}_{\cV_j (\ell)} 
		\lesssim \max_{\km_1 \in \kA_1 \cup \kA_2} \big\|  \nu_j^{(\km_1)} S^{(\km_1)} \big\|_{\cL_j (\ell)}  
		\lesssim O_L (1) g \scale_j ,
	\label{eq:bootstrapx}	
\end{align}
so putting together,  we have $\norm{x_j}_{\cV_j (\ell)} \le O_L (1) g \scale_j$.

Using these bounds,  we bootstrap again to obtain \eqref{eq:bootstrap}.  
If we let $\delta_j = \norm{y_{j} - y_0}_{\cV_{j} (\ell)}$,  then
\begin{align}
	\norm{y_{j} - y_0}_{\cV_{j+1} (\ell)} \le L^{-(d-4+2\eta)} \norm{y_{j} - y_0}_{\cV_{j} (\ell)}
		=   L^{-(d-4+2\eta)} \delta_j
\end{align}
Also,  by the first inequality of \eqref{eq:yjrecursive},  but alternatively using $\norm{x_j + y_j}_{\cV_j (\ell)} \lesssim C_0 g \scale_j$ as an input,  we now obtain
\begin{align}
	\norm{y_{j+1} - y_j}_{\cV_{j+1} (\ell)} \le O_L (1) \scale^{-1+\kt}_j \norm{x_j + y_j}^2_{\cV_j (\ell)}  \le O_L (1) g^2 \scale_j^{1 + \kt} ,
\end{align}
so $\delta_j$ satisfies recursion $\delta_0 = 0$ and
\begin{align}
	\delta_{j+1} \le  L^{-(d-4+2\eta)} \delta_j + O_L (1) g^2 \scale_j^{1 + \kt} .
\end{align}
Since $\scale_{j+1} / \scale_j = L^{-(d-4+2\eta)}$ and $\kt >0$,  we deduce $\delta_j \le O_L (1) g^2 \scale_j$,  as desired. 

To bound $x_j$,  we go back to the first line of \eqref{eq:bootstrapnujkm1},  and observe that $p^{(\km_1)} \E_{k+1} \theta y_k = p^{(\km_1)} (-y_k + \E_{k+1} \theta y_k )$ for $\km_1 \in \kA_0 \cup \kA_1$. 
Since the bound on $\delta_j$ gives $\norm{y_j}_{\cV_j (\ell)} \lesssim \ell_0^4 g\scale_j$,  \eqref{eq:cVnrmvslnrm2} gives
\begin{align}
	|\nu_j^{(\km_1)}| \lesssim L^{(q(\km_1)-2+\eta)j} g \scale_j ,
\end{align}
without $L$-dependent constants.  This gives $\norm{x_j}_{\cV_j (\ell)} \lesssim \ell_0^2 g \scale_j$.
\end{proof}

\begin{remark}
Why do we not apply this argument directly to the full bulk RG map? This is because the domain $\cD^{(0)}_{j,\bulk} (\alpha)$ does not allow $V_j \equiv 0$,  while the perturbative map is defined on the whole $\cV_{\bulk}$.  
It is the aposteori estimate Lemma~\ref{lemma:bootstrap} that verifies that $V_j$ generated by the perturbative maps stays inside $\cD^{(0)}_{j,\bulk} (\alpha)$,
but the stable manifold theorem of \cite[Theorem~2.16]{MR2523458} requires $V_j \equiv 0$ to be included inside the domain.

This happened likewise in \cite{MR3317791} when $d=4$,  where the RG flow was decomposed into sum of quadratic parts and higher order terms.  The quadratic part is defined on the whole RG domain,  and the higher order terms were treated as perturbation.  However,  the polynomial $S^{(\km_2)} (\varphi)$ is marginal when $\km_2 \in \ko_4$ when $d=d_{c,u}$,  so it is necessary to determine the sign of the quadratic part.  In $d >d_{c,u}$,  treatment of the quadratic part becomes simpler,  but the Banach spaces are more complicated.
\end{remark}

\subsection{Stable manifold II.  full bulk RG map}
\label{sec:cotcpiibRm}

The stable manifold of the full bulk RG map can be constructed by adding perturbations to $\vec{\nu}_{ptc}$.
We follow the formulation of \cite{MR3317791}.  For this purpose,  we consider the flow $A_j^{\pt} = (x_j^\pt, y_j^\pt, z_j^\pt) \in X_j \times Y_j \times Z_j$ (see Definition~\ref{eq:spaceXYZdefi}) defined by initial condition \eqref{eq:nupccttyIC} together with $z_0^\pt =0$,  and
\begin{align}
	A_{j+1}^{\pt} = \tilde{\Phi}^0_{j+1} (A_j^{\pt}) := ( \bar\Phi_{j+1}^{\pt ,X} ,   \bar\Phi_{j+1}^{\pt ,Y} ,  \Phi_{j+1}^K  )  (x_j^{\pt}, y_j^{\pt}, z_{j}^\pt) 
\end{align}
with the same $(\bar\Phi_{j+1}^{\pt ,X} ,   \bar\Phi_{j+1}^{\pt ,Y})$ as in the proof of Proposition~\ref{prop:nuptcctty}---we interpret $\bar{\Phi}_{j+1}^{\pt, *} (x_j, y_j, z_j) = \bar{\Phi}_{j+1}^{\pt, *} (x_j, y_j)$ for both $* \in \{X,Y\}$---and $\Phi_{j+1}^K$ as in \eqref{eq:controlledRG23}--\eqref{eq:controlledRG24}.
By Proposition~\ref{prop:nuptcctty} and \eqref{eq:controlledRG24},  we see for sufficiently small $g$ and $\vec{\ba} \in \HB_{\epsilon_p}$,
\begin{align}
	\underline{A}^\pt  := ( A_j^\pt )_{j\ge 0} \in \prod_{j\ge 0} \D_j (\alpha) .
\end{align}
We interpolate $(A_j^\pt)_{j\ge 0}$ with the stab;e manifold of the full bulk RG map using interpolated maps
\begin{align}
	\tilde{\Phi}^t_{j+1} (A_j) = \tilde{\Phi}^0_{j+1} (A_j) + t \delta\Phi_{j+1} (A_j) .   \label{eq:tildePhidefi}
\end{align}
where we abbreviated $A_j = (x_j, y_j, z_j)$,  $\delta \Phi_{j+1} := \Phi_{j+1}  -  \tilde{\Phi}^0_{j+1}$ and $\Phi_{j+1}$ is the full bulk RG map. 
In coordinates,
\begin{align}
	\delta \Phi_{j+1}  : A_j \mapsto \left( p^{0,1} R_{j+1}^U (A_j) ,  \,  (p^{2,3} + q^{0,1} ) R_{j+1}^U (A_j) ,  \,  0 \right)
\end{align}
and $\delta \Phi_{j+1}$ is well-defined for $A_j \in \D_j (\alpha)$.
Solving for the family of stable manifolds on $t \in [0,1]$ is equivalent to solving an ordinary differential equation,  as we explain below.
Suppose there exists a sequence of $A^t_j = (x_j^t, y_j^t, z_j^t) \in \D_j (\alpha)$ that is continuously differentiable in $t$ and satisfies $A_{j+1}^t = \tilde{\Phi}_{j+1}^t (A_j^t)$ for all $j\ge 0$.  If we denote $t$-derivatives using dots,  and differentiate \eqref{eq:tildePhidefi} in $t$,  sequence $\underline{A}^t = (A_j^t)_{j\ge 0}$ satisfies
\begin{align}
	\dot{A}^t_{j+1} = D \tilde\Phi_{j+1}^t (A_j^t) \dot{A}_j^t + \delta\Phi_{j+1} (A_j^t) , 
	\label{eq:dotArelatoin}
\end{align}
so we can consider a system of linear equations 
\begin{align}
	B_{j+1} = D \tilde\Phi_{j+1}^t (A_j^t) B_j + \Psi_{j+1}  , \qquad \Psi_{j+1} \in \D_{j+1}    (\alpha) \label{eq:BDArelation}
\end{align}
with variables $\underline{B} = (B_j)_{j\ge 0}$ and parameters $\underline{A}^t = (A_j^t)_{j\ge 0}$ and $\underline\Psi = (\Psi_{j})_{j\ge 1}$.  
If \eqref{eq:BDArelation} has a (linear) solution map $\underline{S}^t = (S_j^t)_{j\ge 0}$ that satisfies the asymptotic condition $B_j \rightarrow 0$ as $j\rightarrow \infty$ with
\begin{align}
	\underline{B} = \underline{S}^t ( \underline{A}^t) \underline{\Psi} ,  
\end{align}
then \eqref{eq:dotArelatoin} is equivalent to the ODE
\begin{align}
	\underline{\dot{A}}^t = \underline{S}^t ( \underline{A}^t)  \underline{\delta \Phi} (\underline{A}^t) ,  \qquad \underline{A}^0 = \underline{A}^{\pt} .  \label{eq:dotAODE}
\end{align}
To implement this argument,  our first goal is to prove the existence of the solution map $S^t$ and next goal is to prove estimates that guarantee the existence of the solution of \eqref{eq:dotAODE} on $t \in [0,1]$.
To this end,  we define $(\mathbb{A}_0 , \mathbb{A})$,  the domain and codomain of the solution map.

\begin{definition} \label{defi:xyzspace}
We denote $\xyz$ for the sequence of $\underline{A} = (A_j)_{j\ge 0} = (x_j, y_j, z_j)_{j\ge 0} \in \prod_{j\ge 0} X_j \times Y_j \times Z_j$ such that
\begin{align}
	\norm{\underline{A}}_{\xyz} := \sup_{j \ge 0} \norm{A_j}_{X_j \times Y_j \times Z_j}  < \infty 
\end{align}
and $y_0 = z_0 = 0$. 
Let $\xyz_0 \subset \xyz$ be the subspace with $x_0 = y_0 = z_0 =0$.
Due to Lemma~\ref{lemma:RcKisBanach},  $\xyz$ and $\xyz_0$ are Banach spaces.
\end{definition}

In what follows,  norms on linear operators are operator norms,  and we do not explicitly write the base spaces as long as they are clear from the context.

\begin{proposition} \label{prop:Stdefi}
For $t \in [0,1]$,   $\underline{\Psi} = (\Psi_{j})_{j\ge 0} \in \xyz_0$ and $\underline{A} = ( A_j )_{j\ge 0} \in \prod_{j=0}^{\infty} \D_j (\alpha)$,
there exists a linear map $\underline{S}^t (\underline{A}) = (S_j^t (A_j) )_{j \ge 0} : \xyz_0 \rightarrow \xyz$ such that $\underline{B} = \underline{S}^t (\underline{A}) \underline{\Psi}$ satisfies
\begin{align}
	B_{j+1} = D \tilde{\Phi}^t_{j+1} (A_j) B_j + \Psi_{j+1} \quad \text{for each} \quad j\ge 0.  \label{eq:BDPhiBPsi}
\end{align}
Moreover,  it satisfies
\begin{itemize}
\item[(1)] $\norm{\underline{S}^t (\underline{A}) } \lesssim 1$;  and
\item[(2)] if $\norm{\Psi_j}_{X_j \times Y_j \times Z_j} \le O_L (1) g^{9/4} \scale_j^{\kt}$ for each $j$,  then $\norm{D_{\underline{A} } \underline{S}^t (\underline{A}) \underline{\Psi} }\le O_L (1)$.
\end{itemize}
\end{proposition}
\begin{proof}
Let $\underline{D \tilde{\Phi}}^{t} : \xyz \rightarrow \xyz$ be given by $\big( \underline{D \tilde{\Phi}}^{t} \underline{B} \big)_{j+1} = D \tilde{\Phi}^t_{j+1} (A_j) B_j$ and $\big( \underline{D \tilde{\Phi}}^{t} \underline{B} \big)_{0} = B_0$ so \eqref{eq:BDPhiBPsi} can be stated alternatively as
\begin{align}
	\underline{B} = \underline{D \tilde{\Phi}}^{t} \underline{B} + \underline{\Psi} .
\end{align}
If $\underline{I}$ is the identity map on $\mathbb{A}$,  we would like to define $\underline{S}^t (\underline{A})$ as the inverse of $\underline{I} - \underline{D \tilde{\Phi}}^{t} : \mathbb{A} \rightarrow \mathbb{A}_0$.  
In the proof,  we will denote $\underline{A} = (x_j, y_j, z_j)_{j\ge 0}$,  $\underline{\Psi} = (\Psi_j^X, \Psi_j^Y, \Psi_j^Z)_{j\ge 0}$ and
\begin{align}
	D \tilde{\Phi}_{j+1}^t (A_j) = \begin{pmatrix}
		L^{XX}_{j+1} & L^{XY}_{j+1} & L^{XZ}_{j+1} \\
		L^{YX}_{j+1} & L^{YY}_{j+1} & L^{YZ}_{j+1} \\
		L^{ZX}_{j+1} & L^{ZY}_{j+1} & L^{ZZ}_{j+1} 
	\end{pmatrix} .
\end{align}
We verify some properties of $\underline{I} - \underline{D \tilde{\Phi}}^{t}$.

\noindent\medskip \emph{Injectivity.}

Since $x_j + y_j \in \cD_j (\alpha)$, by Lemma~\ref{lemma:VptmthV},
\begin{align}
	\norm{ D \bar{\Phi}^{\pt}_{j+1} (x_j + y_j) - D \bar{\Phi}^{\pt}_{j+1} (0)} = \big\| D \big( \bar{\Phi}^{\pt}_{j+1} (x_j + y_j) - \Eplus \theta (x_j + y_j) \big)  \big\|  \le O_L (1) \tilde{g}_j \scale_j^{\kt} ,  
\end{align}
and by \eqref{eq:controlledRG22},
\begin{align}
	\norm{D_{x_j} ( \tilde{\Phi}^{t}_{j+1} -  \tilde{\Phi}^{0}_{j+1} ) (A_j) } &\le \norm{D_{x_j} R_{j+1}^U (A_j)} \le O_L (w_Y) \tilde{g}_{j+1}^3 \scale_{j+1}^{\kae - 1 + \kt} , \\
	\norm{D_{y_j} ( \tilde{\Phi}^{t}_{j+1} -  \tilde{\Phi}^{0}_{j+1} ) (A_j) } &\le \norm{D_{y_j} R_{j+1}^U (A_j)} \le O_L (1) \tilde{g}_{j+1}^3 \scale_{j+1}^{\kae - 1 + \kt} , \\	
	\norm{D_{z_j} ( \tilde{\Phi}^{t}_{j+1} -  \tilde{\Phi}^{0}_{j+1} ) (A_j) } &\le \norm{D_{z_j} R_{j+1}^U (A_j)} \le O_L (w_Z^{-1} w_Y) \scale_{j+1}^{\kae -1} .
\end{align}
Combined with the estimates \eqref{eq:Phiptlinbnds}--\eqref{eq:Phiptlinbnds3} on $D \bar\Phi^{\pt} (0)$,  
taking sufficiently large $L$ and small $\tilde{g}$, 
we have
\begin{alignat}{5}
	\norm{ L_{j+1}^{XX} } &\ge \frac{1}{2} L^{z(d,\eta)} , \qquad
	&\norm{ L_{j+1}^{XY} } &\le O_L (w_Y^{-1}) ,  \qquad
	& \norm{ L_{j+1}^{XZ} } &\le O_L (w_Z^{-1})    ,  \\
	\norm{ L_{j+1}^{YX} } &\le O_L (w_Y) \tilde{g}_{j+1}   , \qquad
	&\norm{ L_{j+1}^{YY} } &\le 2 L^{-(d-4 +2\eta) \kt}  , \qquad
	& \norm{ L_{j+1}^{YZ} } &\le O_L (w_Y w_Z^{-1})  .
\end{alignat}
Also,  by \eqref{eq:controlledRG23}--\eqref{eq:controlledRG24}, 
\begin{align}
	\norm{ L_{j+1}^{ZX} } \le O_L (w_Z)  \tilde{g}_{j+1}^2 , \quad \norm{L_{j+1}^{ZY}} \le O_L (w_Z w_Y^{-1})  \tilde{g}_{j+1}^2 ,    \quad \norm{L_{j+1}^{ZZ}} \le L^{-(d-4+2\eta) \kt} . \label{eq:LXYZbnds1}
\end{align}
Thus by taking $\tilde{g}_j^{-3/2} \gg w_Y \tilde{g}_j^{-1} \gg w_Z \gg w_Y \gg O_L (1)$,  we can safely say that
\begin{align}
	\norm{ L_{j+1}^{XX} } \ge \frac{1}{2} L^{z(d,\eta)} , \quad \norm{ L_{j+1}^{XY} } ,   \; \cdots \norm{L_{j+1}^{ZZ}} \le L^{-(d-4 + 2\eta) \kt}	\label{eq:LXYZbnds2}
\end{align}
and for sufficiently large $L$,
\begin{align}
	\norm{ ( \underline{I} - \underline{D \tilde{\Phi}}^{t} ) \underline{B} } \ge \frac{1}{2} \norm{\underline{B}}   \label{eq:ImLXYZbnds}
\end{align}
for any $\underline{B} \in \mathbb{A}$.
This proves the injectivity of $\underline{I} - \underline{D \tilde{\Phi}}^{t}$.

\noindent\medskip \emph{Surjectivity.} Having fixed $\underline{\Psi}$, consider a map $\underline{F}_{\Psi} : \xyz \rightarrow \xyz$ given by,  when $\underline{A}' = \underline{F}_{\Psi} \, \underline{A}$,
\begin{align}
	x'_{j} &= \big(L^{XX}_{j+1} \big)^{-1} \big( x_{j+1} - L^{XY}_{j+1} y_j - L^{XZ} z_j - \Psi_{j+1}^X \big) \\
	y'_{j+1} &= L^{YX}_{j+1} x_j + L^{YY}_{j+1} y_j + L^{YZ}_{j+1} z_j + \Psi_{j+1}^Y  \\
	z'_{j+1} &= L^{ZX}_{j+1} x_j + L^{ZY}_{j+1} y_j + L^{ZZ}_{j+1} z_j + \Psi_{j+1}^Z 
\end{align}
with $y'_0 = z'_0 = 0$.  (Lower bound on $L_{j+1}^{XX}$ in \eqref{eq:LXYZbnds2} shows that it is invertible because it is an endomorphism on the finite dimensional space $\R^{\cA_0 \cup \cA_1}$.)
By the estimates \eqref{eq:LXYZbnds2},  $\underline{F}$ is a contraction for sufficiently large $L$,  so the Contraction Mapping Theorem implies the existence of a fixed point $\underline{A}_{\Psi}$ of $\underline{F}_{\Psi}$,  and equivalently $(\underline{I} - \underline{D \tilde{\Phi}}^{t}) \underline{A}_{\Psi} = \underline{\Psi}$.
\vspace{5pt}

\noindent\medskip \emph{Bounds (1),  (2).} 
By the Open Mapping Theorem and bijectivity proved above,  $\underline{S}^t (\underline{A}) = (\underline{I} - \underline{D \tilde{\Phi}}^{t} (\underline{A}) )^{-1}$ is a bounded linear map.  Also,  by \eqref{eq:ImLXYZbnds},  it satisfies
\begin{align}
	\norm{\underline{S}^t (\underline{A}) } &\le 2 , \\
	\norm{ D_{\underline{A}} \underline{S}^t (\underline{A}) \underline{\Psi} } &= \norm{ \underline{S}^t (\underline{A}) ( D_{\underline{A}} \underline{D \tilde{\Phi}}^{t} ) \underline{S}^t (\underline{A}) \underline{\Psi} } \le 4 \norm{D_{\underline{A}} \underline{D \tilde{\Phi}}^{t} \underline{\Psi}}  , 
\end{align}
so it is enough to bound $D_{\underline{A}} \underline{D \tilde{\Phi}}^{t} \underline{\Psi}$.
(It should not be confused with the derivative of $\underline{D \tilde{\Phi}}^{t}$ as a linear map on $\underline{B} \in \mathbb{A}$.  It is the derivative in $\underline{A} = (A_j)_{j\ge 0}$.) Following the process that is used to prove \eqref{eq:LXYZbnds2},  using derivative bounds \eqref{eq:controlledRG22}--\eqref{eq:controlledRG24},  we obtain
\begin{align}
	& \norm{D L_{j+1}^{XX}} ,  \; \norm{D L_{j+1}^{XY} },  \cdots \; \norm{D L_{j+1}^{YZ}}  \le O_L (1) ,  \\
	& \norm{D L_{j+1}^{ZY}} ,  \;  \norm{D L_{j+1}^{ZX}} ,  \; \norm{D_{x_j} L_{j+1}^{ZZ}} ,  \; \norm{D_{y_j} L_{j+1}^{ZZ}} \le O_L (1) \tilde{g}_{j+1}^{-1} \scale_{j+1}^{-\kt} ,  \\
	& \norm{D_{z_j} L_{j+1}^{ZZ}} \le O_L (1) \tilde{g}_{j+1}^{-\frac{9}{4}} \scale_{j+1}^{\kae-\kbe-\kt} ,
\end{align}
(when $w_Y$ and $w_Z$ are chosen to be only $L$-dependent)
giving the bound (2) when multiplied with the assumed bound on $\underline{\Psi}$.
\end{proof}

By the argument presented above Definition~\ref{defi:xyzspace},  these bounds on the solution map $\underline{S}^t$ is enough to solve the ODE \eqref{eq:dotArelatoin}.

\begin{proposition} \label{prop:nupcctty}
Let $\alpha =1$,  $g$ be sufficiently small and $\vec{\ba} \in \HB_{\epsilon_p}$.
Then there exists $ \vec\nu_{pc}  = (\nu^{(\km)}_{pc} )_{\km \in  \kA_0 \cup \kA_1 }$ such that the following hold: whenever $(V_0 , K_0)$ is determined by coefficients
\begin{align}
	\begin{cases}
		(\nu^{(\km_1)}_0 )_{\km_1 \in  \kA_0 \cup \kA_1 } = \vec\nu_{pc} \in \R^{\kA_0 \cup \kA_1} & {} \\
		\nu^{(\km_1)}_0 = g^{(\km_2)}_0 =  0 & (\km_1 \in \kA_2 \cup \kA_3   ,  \; \km_2 \in \ko_{4,\nabla}) \\
		g^{(\emptyset)}_0 = g  ,  \quad 		K_0 = 0 ,    & {} 	\\
	\end{cases}
\end{align}
the bulk RG flow of infinite length exists.  Moreover,  $\vec{\nu}_{pc}$ is a continuous function of $\vec{\ba}$ and
a differentiable function of $g$ satisfying $| \vec{\nu}_{pc} - \vec{\nu}_{ptc} | = O(g^3)$.
\end{proposition}

\begin{proof}
By the Picard-Lindel\"of Theorem, to show that ODE \eqref{eq:dotAODE} has a solution on $t \in [0,1]$,  it is enough to show that $\underline{A} \mapsto \underline{S}^t (\underline{A}) \underline{\delta \Phi} (\underline{A})$ is a Lipschitz continuous function.  But since 
\begin{align}
	\norm{\delta \Phi_{j+1} (A_j) }_{X_j \times Y_j \times Z_j}  \le  \norm{R_{j+1}^U (A_j)}_{X_j \times Y_j} \le O_L (1) g^3 \scale_j^{\kae -1 + \kt}  ,   \label{eq:nupcctty2} \\
	\norm{D_{\underline{A}} \underline{\delta \Phi} (\underline{A} )} \le  \sup_{j\ge 0}  \norm{D R_{j+1}^U (A_j) }_{X_j \times Y_j} \le O_L ( \scale_j^{\kae - 2+\kt} ) \label{eq:nupcctty3}
\end{align}
along with bounds Proposition~\ref{prop:Stdefi}~(1),(2),  we see 
\begin{align}
	\big\| D_{\underline{A}} \big( \underline{S}^t ( \underline{A}) \underline{\delta \Phi} (\underline{A}) \big)  \big\|  \le C_L .
	\label{eq:nupcctty5} 
\end{align}
Thus we can consider a solution $\underline{A}^t = (A^t_j)_{j\ge 0}$ of \eqref{eq:dotAODE} on $t \in [0,  \epsilon']$ for some $\epsilon' >0$ (that only depends on $C_L$).
Then by Proposition~\ref{prop:Stdefi}~(1) and the bound on $\underline{\delta \Phi}$, 
\begin{align}
	\norm{ \underline{A}^t - \underline{A}^{\pt} } 
		\le \epsilon' \sup_{t \in [0,\epsilon']} \norm{ \underline{S}^t (\underline{A}^t) \underline{\delta \Phi} (\underline{A}^t) } \le O_L (\epsilon' g^3) , \qquad t \in [0,\epsilon'] .
		\label{eq:nupcctty6} 
\end{align}
If we have chosen parameters so that $\underline{A}^{\pt} \in \prod_{j=0}^{\infty} \D_j (1/2)$,  then by choosing $g$ sufficiently small,  we see $A_j^{t} \in \D_j ( \frac{1}{2} (1 + \epsilon') \alpha )$ for all $t \in [0,\epsilon']$,  i.e.,  it stays inside the RG domain.  
This can be repeated $\lfloor 1/ \epsilon' \rfloor$ number of times to obtain $\underline{A}^t \in \prod_{j=0}^{\infty} \D_j (1)$ on the full interval $t \in [0,1]$ and
\begin{align}
	\norm{ \underline{A}^1 - \underline{A}^{\pt} } 
		\le O_L (g^3) .   \label{eq:nupcctty7} 
\end{align}

Continuity of $\underline{A}^1$ in $\vec{\ba}$ follows from the definition of the space $\xyz$.
Differentiability of $\underline{A}^t$ in $g$ follows since it can be constructed from the ODE
\begin{align}
	\frac{\rd}{\rd t} \partial_g \underline{A}^t = D_{\underline{A}} \big( \underline{S}^t ( \underline{A}^t) \underline{\delta \Phi} (\underline{A}^t) \big) \partial_g \underline{A}^t .
\end{align}
and the initial condition $\underline{A}^{\pt}$ is a differentiable function of $g$,  due to Proposition~\ref{prop:nuptcctty}.
Also,  by \eqref{eq:nupcctty7},  we obtain $| \vec{\nu}_{ptc} - \vec{\nu}_{pc} | \le O_L (g^3 )$.
\end{proof}

\subsection{Fixed point argument} \label{sec:cpnt}

Construction of the critical point is now almost direct from the stable manifold. 
In the next theorem,  $\vec{N}_c ( \ba^{(\emptyset)} )$ gives a set of critical points for the RG flow with square mass $\ba^{(\emptyset)}$.
We use the convention that $2^{X}$ is the power set of $X$,  i.e.,   $Y \in 2^X$ if and only if $Y \subset X$.

\begin{proposition} \label{prop:cpconstr}

Let $g$ be sufficiently small and denote $\bmrho_0 = C_{\cD} g$.
Let $p^1 \vec{\ba} = ( \ba^{(\km)} : \km \in \kA_1  )$ and $\bar{B}_{\epsilon_p /2} = \{ p^1 \vec{\ba} : \norm{ p^1 \vec{\ba} } \le \epsilon_p /2 \}$.

Then there exists a set-valued function
\begin{align}
	\vec{N}_c :  [0,\epsilon_p / 2] \rightarrow 2^{ [-\bmrho_0 , \bmrho_0] \times \bar{B}_{\epsilon_p /2} } ,
	\qquad \ba^{(\emptyset)} \mapsto \vec{N}_c ( \ba^{(\emptyset)} )
\end{align}
with the following property.  
For each $\tilde\ba^{(\emptyset)} \in [0, \epsilon_p /2]$ and $\vec{b} \in \vec{N}_c (\tilde\ba^{(\emptyset)})$,
if $\vec{\nu}, \vec{\ba} \in \R^{\kA_0 \cup \kA_1}$ are given by
\begin{align}
\begin{cases}
	\nu^{(\emptyset)} =  b^{(\emptyset)}  \text{ and } \nu^{(\km)} = - b^{(\km)} \text{ for each } \km \in \kA_1 \\
	\ba^{(\emptyset)} = \tilde\ba^{(\emptyset)}  \text{ and }  \ba^{(\km)} = b^{(\km)} \text{ for each } \km \in \kA_1
\end{cases}
\end{align}
then for $\vec{\nu}_{pc}$ given by Proposition~\ref{prop:nupcctty},
\begin{align}
	\vec{\nu}  = \vec{\nu}_{pc} ( \vec\ba ) ,
	\label{eq:cpconstr}
\end{align}
i.e.,  $(\nu^{(\emptyset)} ,  - p^1 \vec{\ba} )$ is critical for the RG flow with mass $(\tilde\ba^{(\emptyset)} ,   p^1 \vec{b})$. 
Moreover, the graph of $\vec{N}_c$ is compact.
\end{proposition}
\begin{proof}
Having fixed $\tilde\ba^{(\emptyset)} \in [0, \epsilon_p / 2]$, 
we consider the function
\begin{align}
	F ( p^1 \vec{b} ) = - p^1 \vec{\nu}_{pc} ( \tilde\ba^{(\emptyset)} ,  p^1 \vec{b} ) , \qquad p^1 \vec{b} \in \bar{B}_{\epsilon_p /2 } ,
\end{align}
where $\vec{\nu}_{pc}$ is that constructed in Proposition~\ref{prop:nupcctty}.  Since $|\vec{\nu}_{ptc} - \vec{\nu}_{pc}| \le O(g^3)$ and $\vec{\nu}_{ptc} = O(g)$ (by Proposition~\ref{prop:nuptcctty}),  we also have $|\vec{\nu}_{pc}| \le O(g)$,  so we see that $- p^1 \vec{\nu}_{pc} ( \tilde\ba^{(\emptyset)} ,  p^1 \vec{b} ) \in \bar{B}_{\epsilon_p / 2}$ for sufficiently small $g$,
i.e.,  $F$ can be considered as a function $F: \bar{B}_{\epsilon_p /2 } \rightarrow \bar{B}_{\epsilon_p /2 }$.

Since the domain $\bar{B}_{\epsilon_p /2}$ is a convex compact set and $F$ is continuous by Proposition~\ref{prop:nupcctty},
there exists a fixed point $p^1 \vec{b}_F$ to $F$ by the Brouwer's fixed point theorem.  Then we let $b^{(\emptyset)}_F = p^0 \vec{\nu}_{pc} ( \tilde\ba^{(\emptyset)},  p^1 \vec{b}_F )$,
so if we set $\vec\nu = \vec{\nu}_{pc} ( \tilde\ba^{(\emptyset)},  p^1 \vec{b}_F )$
(in particular $\nu^{(\emptyset)} = b^{(\emptyset)}_F = p^0 \vec{\nu}_{pc} ( \tilde\ba^{(\emptyset)},  p^1 \vec{b} )$),
then it satisfies
$\nu^{(\km)} = - b^{(\km)}_F$ for any $\km \in \kA_1$,  as desired.
We denote $\vec{N}_c (\tilde\ba^{(\emptyset)})$ for the set of such fixed points $\vec{b}_F = (b_F^{(\emptyset)},  p^1 b_F)$.

For the compactness of the graph of $\vec{N}_c$,  let $\ba^{(\emptyset)}_{k} \rightarrow \ba^{(\emptyset)}$ be a converging sequence in $[0, \epsilon_p / 2]$
and let $(\nu^{(\emptyset)}_k,  p^1 \vec{\ba}_k ) \in \vec{N}_c (\ba^{(\emptyset)}_{k})$ be convergent,  and denote $(\nu^{(\emptyset)}_{\infty},   p^1 \vec{\ba}_{\infty} )$ for the cluster point.
Then by continuity of $\vec{\nu}_{pc}$ and the definition of $\vec{N}_c (\ba_k^{(\emptyset)})$,
\begin{align}
	- p^1 \vec{\nu}_{pc} (\ba^{(\emptyset)}, p^1 \vec{\ba}_{\infty}) 
		= - \lim_{k \rightarrow \infty} p^1 \vec{\nu}_{pc} (\ba^{(\emptyset)}_k ,  p^1 \vec{\ba}_{k}) 
		= \lim_{k \rightarrow \infty}  p^1 \vec{\ba}_{k} = p^1 \vec{\ba}_{\infty}
\end{align}
and also
\begin{align}
	p^0 \vec{\nu}_{pc} (\ba^{(\emptyset)} ,  p^1 \vec{\ba}_{\infty} )
		= \lim_{k \rightarrow \infty} p^0 \vec{\nu}_{pc} (\ba^{(\emptyset)}_k ,  p^1 \vec{\ba}_k ) 		
		= \lim_{k \rightarrow \infty} \nu^{(\emptyset)}_k 		
		=	\nu^{(\emptyset)}_{\infty}  ,
\end{align}
thus $(\nu^{(\emptyset)}_{\infty},   p^1 \vec{\ba}_{\infty} ) \in \vec{N}_c (\ba^{(\emptyset)})$ as desired.
\end{proof}

It is expected that $\vec{N}_c$ is actually a single-valued function, 
but we cannot prove this only using abstract arguments.
Instead,  we make a choice of the critical point using the compactness of the graph of $\vec{N}_c$.

\begin{proof}[Proof of Theorem~\ref{thm:theStableManifold}]
The case $d= d_{c,u}$ is dealt by Proposition~\ref{prop:stablemanifold-dcu}.

We are only left to consider $d > d_{c,u}$,  that uses the construction above.
We fix a sequence $( \ba^{(\emptyset)}_{c,k} ,  \nu^{(\emptyset)}_{c,k},  p^1 \vec{\ba}_{c,k} ) \in \R \times \R \times \R^{\kA_1}$ such that $\ba_{c,k}^{(\emptyset)} \downarrow 0$,
\begin{align}
	( \nu^{(\emptyset)}_{c,k},  p^1 \vec{\ba}_{c,k} ) \in \vec{N}_c (\ba^{(\emptyset)}_{c,k}) \quad \text{for each } k 
\end{align}
and $( \nu^{(\emptyset)}_{c,k},  p^1 \vec{\ba}_{c,k})$ is convergent---existence of such a sequence is guaranteed by the compactness of the graph of $\vec{N}_c$.
Denote $\ba_{c,\infty}^{(\emptyset)} = 0$ and
\begin{align}
	( \nu^{(\emptyset)}_{c,\infty} ,  p^1 \vec{\ba}_{c,\infty} ) = \lim_{k \rightarrow \infty} ( \nu^{(\emptyset)}_{c,k},  p^1 \vec{\ba}_{c,k}) \in \vec{N}_c (0) .
\end{align}
These satisfy the desired properties due to Proposition~\ref{prop:cpconstr}.

For the final bound,  by the bounds on $\vec{\nu}_{pc}$ (see Proposition~\ref{prop:nuptcctty} and \ref{prop:nupcctty}),
we have both $\nu^{(\emptyset)}_{c, \infty}$ and $p^1 \vec{\ba}_{c, \infty}$ of order $O(g)$.   
\end{proof}

\subsection{Asymptotic of the quartic term}

In this final part,  we make a precise estimate on the asymptotic of the coefficient of the $|\varphi|^4$-term.  This is essential for computing the plateau and the scaling limits. 

\begin{lemma} \label{lemma:PhijUmEVj}
Let $(V_j, K_j) \in \D_j$ and $U_{j+1} = \Phi_{j+1}^{U} (V_j, K_j)$.  Then
\begin{align}
	\norm{U_{j+1} - \mathbb{V}^{(0)} \mathbb{E}_{j+1} \theta V_j }_{\cL_j (\ell)}
		\le O_L (1) \tilde{\chi}_j \norm{V_j}^2_{\cV_j (\ell)} \le O_L (1) \tilde{\chi}_j \tilde{g}_j^2 \scale_j^2 .
\end{align}
\end{lemma}
\begin{proof}
The first inequality follows from Lemma~\ref{lemma:VptmthV} and \eqref{eq:controlledRG22},  since
\begin{align}
	\Phi_{j+1}^U (V_j ,  K_j ) = \mathbb{V}^{(0)} \big( \Phi_{j+1}^\pt (V_j) - \E_{j+1} \theta V_j \big) + \mathbb{V}^{(0)} \E_{j+1} \theta V_j + R_{j+1}^U (V_j, K_j) .
\end{align}
The second inequality follows from Lemma~\ref{lemma:VincDsize}.
\end{proof}

\begin{lemma} \label{lemma:gjasympcritd}
Let $g > 0$ be sufficiently small,  $L$ be sufficiently large and $\vec{\ba} \in \HB_{\epsilon_p}$.
Let $(V_{j,\bulk},K_{j,\bulk})$ be a bulk RG flow of infinite length with $g_0^{(\emptyset)} = g$. 
\begin{align}
	g_j^{(\emptyset)} =
		\begin{cases} 
			(\bfb j)^{-1} (1 + O(j^{-2} \log j)) & (d= d_{c,u} ,  \;  j < j_{\ba}) \\
			g_{\infty} + O_L ( 2^{- (j - j_\ba)_+ } g_{\infty}^2) & (d = d_{c,u} ,  \;  j \ge j_{\ba}) \\
			g_{\infty} + O_L ( \scale_j g^2  ) & (d > d_{c,u})
		\end{cases}	
\end{align}
for some $\bfb , g_{\infty} > 0$ such that $\bfb = \frac{n+8}{16\pi^2}$ and $g_{\infty} \sim ( \bfb |\log \ba^{(\emptyset)} | )^{-1}$ as $\ba^{(\emptyset)} \downarrow 0$ for $d = d_{c,u}$ and $g_{\infty} = g + O(g^2)$ for $d >d_{c,u}$.
\end{lemma}
\begin{proof}
The case $d = d_{c,u}$,  $j < j_{\ba}$ follows from \cite[Lemma~4.8]{MR3459163},  
where in the reference,  we set $a,b \in \Z^4$ so that $j = j_{ab}$ and $j_{m} = j_{\ba}$.

For the case $d = d_{c,u}$,   $\ba^{(\emptyset)} > 0$,  we use \cite[Lemma~4.7]{MR3459163} to see that $\lim_{j\rightarrow \infty} g^{(\emptyset)}_j \rightarrow g_\infty$ for some $g_{\infty}$.  Also,  due to Lemma~\ref{lemma:PhijUmEVj}
\begin{align}
	g^{(\emptyset)}_{k+1} = g^{(\emptyset)}_k + O_L (1) \tilde{\chi}_{k+1} \tilde{g}_k^2 
\end{align}
while $\tilde{\chi}_{j+1} = 2^{-(j+1 - j_{\ba})_+}$ by definition,  so the desired error term follows after summing over $k \ge j$.

Finally,  for $d > d_{c,u}$,  Lemma~\ref{lemma:PhijUmEVj} again gives
\begin{align}
	g^{(\emptyset)}_{k+1} = g^{(\emptyset)}_k + O_L (1) \tilde{\chi}_{k+1} \tilde{g}_k^2 \scale_k ,
\end{align}
so the asymptotic follows after summing over $k\ge j$.
\end{proof}

\section{Scaling limits}
\label{sec:mainresults}

We are now equipped with all the tools for proving the scaling limits of Theorem~\ref{thm:WNlimit} and \ref{thm:NGlimit}. 
To outline,  we apply Lemma~\ref{lemma:Vtilde} and Proposition~\ref{prop:mgfref} to state the moment generating function in terms of integrals of $Z_{N,\bulk}$.  
If we tune the initial condition of the RG flow to the critical values,  then we can approximate $Z_{N,\bulk} \approx \exp (-u_{N,\bulk} (\Lambda))$ as in the introduction of Section~\ref{sec:cotcpfdf},  and this almost concludes the Theorem~\ref{thm:WNlimit} and \ref{thm:NGlimit}(ii).
For the argument Theorem~\ref{thm:NGlimit}(i),  we just need an additional step to preserve the quartic term in the limit.
We make these arguments precise in this section.

\subsection{Finite volume RG theorem}

To prove the main statements,  we will use the RG flow generated by the initial conditions given by Theorem~\ref{thm:theStableManifold}. 
Since the theorem only makes reference to the RG flow in the infinite volume,  we need a special treatment on the final RG step of $j+1 =N$.  The estimates should be stated on a slightly larger domain
\begin{align}
\label{eq:cDbulktilde}
	& \tilde\cD_{N,\bulk}
		=  \Big\{ (\nu_{N}^{(\km_1)},  g_{N}^{(\km_2)})  :  |\nu_{N}^{(\km_1)}| \le  L^d \cdot C_{\cD} L^{(q (\km_1) - 2 +\eta) N} \scale_N \tilde{g}_N \text{ if } \km_1 \in \kA_0 \cup \kA_1 \cup \kA_2 ,  \\
			& \qquad\qquad\qquad\;\; |\nu_{N}^{(\km_1)}| \le L \cdot \alpha  C_{\cD} \scale_N^{-\kt} \tilde{g}_N 
			\text{ if } \km_1 \in \kA_3  ,  \nnb
			& \qquad\qquad\qquad\;\;  g^{(\emptyset)}_N / \tilde{g}_N \in \big( (2 C_{\cD})^{-1},   2 C_{\cD} \big),  \;\;
			| g_N^{(\km_2)}| \le 2 C_{\cD} \scale_N^{-\kt} \tilde{g}_N^{3/2}  \text{ if } \km_2 \in \ko_{4,\nabla}
			\Big\} . \nonumber
\end{align}

\begin{proposition} \label{prop:theRGthm}
Let $\eta \in [0,1/2)$,  $d\ge d_{c,u}$ and $g >0$ be sufficiently small,  $(\vec{\ba}_{c,k})_{k \in \N \cup \{\infty\}}$ be as in Theorem~\ref{thm:theStableManifold} and $\nu_c$ be as in Definition~\ref{defi:nuc}.
There exist convergent sequences $(\epsilon_k, \epsilon'_k )_{k \in \N \cup \{\infty \}}$ such that $\epsilon_k \rightarrow 0$,  $\epsilon'_k \downarrow 0$ and satisfy the following.
On $\Lambda = \Lambda_N$,  consider $\vec{\ba}$ given by $\ba^{(\km)} = \ba_{c,k}^{(\km)}$ for each $\km \in \kA_1$,  $\ba^{(\emptyset)} = \epsilon'_k$ and
\begin{align}
	Z_{N,\bulk} (\varphi) = \E_{w_N} \theta \big[ \exp (-V_{0,\bulk} (\Lambda,\varphi)) \big] 
\end{align}
where $V_0$ is defined using $\nu_0^{(\emptyset)} = \nu_c + \epsilon_k$, 
$\nu_0^{(\km_1)} = - \ba^{(\km_1)}_{c,k}$ for $\km_1 \in \ko_{2,\nabla}$,
$g^{(\emptyset)}_0 = g$ and $g^{(\km_2)} = 0$ for $\km_2 \in \ko_{4,\nabla}$.  Then
\begin{align}
	Z_{N,\bulk} (\varphi) = e^{-u_{N,\bulk} ( \Lambda_N )} \left( I_{N,\bulk} (\varphi) + K_{N,\bulk} \right) \label{eq:theRGthm1}
\end{align}
for some $u_{N,\bulk} (\Lambda_N) \in \R$,  $V_{N,\bulk} \in \tilde\cD_N^{(0)}$ and $K_{N,\bulk} \in M \cK_N$ where $I_{N,\bulk} = \cI_N (V_{N,\bulk})$ is as in Section~\ref{sec:WIcoords} and $M >0$ is a (possibly $L$-dependent) constant.
\end{proposition}

\begin{proof}
We drop $\bulk$ in the proof for brevity.
We set the initial conditions with reference to Theorem~\ref{thm:theStableManifold}.  With $(\vec{\ba}_{c,k},  \nu^{(\emptyset)}_{c,k} )_{k}$ as in the theorem,  let
$\epsilon'_k = \ba_{c,k}^{(\emptyset)}$ and  $\epsilon_k = \nu_{c,k}^{(\emptyset)} - \nu_c$.
Then the infinite volume RG flow exists by the theorem.  
Also,  by Theorem~\ref{thm:infvolRGmap},  the finite volume RG flow up to scale $j < N$ can be considered as a projection of the infinite volume RG flow,  so we see that $(V_{j}, K_{j}) \in \D_j$ for all $j < N$.  

We are now left to consider $(\delta u_{N}, V_{N}, K_{N}) := \Phi_N (V_{N-1}, K_{N-1})$.
By Lemma~\ref{lemma:PhijUmEVj},
\begin{align}
	\norm{\mathbb{V}^{(0)} \E_N \theta V_{N-1} - V_N}_{\cL_N (\ell)} \le O_L (1) \tilde{\chi}_N \tilde{g}_N^2 \scale_N^{2} ,
\end{align}
and also by \eqref{eq:EthV1},
\begin{align}
	\norm{\E_N \theta V_{N-1} - V_{N-1}}_{\cL_N (\ell)} \le O_L (1) \tilde{\chi}_N \tilde{g}_N \scale_N .
\end{align}
If we use the additional observation that $(\pi_4 + \pi_{4,\nabla}) (\E_N \theta V_{N-1} - V_{N-1} ) = 0$,  then we see that the two bounds above and Definition~\ref{defi:cLnorm} imply $V_N \in \tilde{\cD}_{N}^{(0)}$.
Finally,  \eqref{eq:controlledRG23} with $p=q=0$ implies
\begin{align}
	\norm{K_{N}}_{\cW_N} \le M_{0,0} \tilde{\chi}_{N}^{3/2} \tilde{g}_N^3 \scale_N^{\kae} ,
\end{align}
thus by Lemma~\ref{lemma:KNbnd},  $K_{N} \in (M_{0,0} / 2C_{\rg}) \cK_N$.
\end{proof}

\begin{lemma} \label{lemma:gNasymp}
Let $\sb$ and $g_{\infty}$ be as in Lemma~\ref{lemma:gjasympcritd}.  Then under the assumptions of Proposition~\ref{prop:theRGthm},
\begin{align}
	g_N^{(\emptyset)} =
		\begin{cases} 
			(\bfb N)^{-1} (1 + O(N^{-2} \log N)) & (d= d_{c,u} ,  \;  N < j_{\ba}) \\
			g_{\infty} + O_L ( 2^{- (N - j_\ba)_+ } g_{\infty}^2) & (d = d_{c,u} ,  \;  N \ge j_{\ba}) \\
			g_{\infty} + O_L ( \scale_N g^2  ) & (d > d_{c,u})
		\end{cases}	\label{eq:gNasymp}
\end{align} 
\end{lemma}
\begin{proof}
By Proposition~\ref{prop:theRGthm} and Lemma~\ref{lemma:gjasympcritd},  we see that $g_{N-1}$ satisfies \eqref{eq:gNasymp}.  Also,  by Lemma~\ref{lemma:VptmthV},  we have
\begin{align}
	| g_N^{(\emptyset)} - g_{N-1}^{(\emptyset)}| \le O_L ( \tilde{\chi}_N \tilde{g}_N^2 \scale_N ) ,
\end{align}
which give \eqref{eq:gNasymp}.
\end{proof}

In the proofs of Theorem~\ref{thm:WNlimit} and \ref{thm:NGlimit} we see below,  we only consider the bulk part of the RG flow,  so we take $\vec{\lambda}_{\o,0} = \vec{\lambda}_{\x,0} = 0$.

\subsection{Proof of Theorem~\ref{thm:WNlimit}--white noise limit}

We write the rescaled field as $\f_N = f_N / \sa_N$,  where we recall $f_N$ from \eqref{eq:fN}.
We take Lemma~\ref{lemma:Vtilde} and the first equality of Proposition~\ref{prop:mgfref} as the starting point.  
We apply the initial conditions determined by Proposition~\ref{prop:theRGthm},  with $k<\infty$.  In particular,  we take $\nu = \nu_c + \epsilon_k$ and obtain
\begin{align}
	\langle e^{L^{dN/2} (\varphi,f_N)} \rangle_{g,\nu_{c} + \epsilon_k,  N}
		& = e^{\frac{1}{2} (\f_N,   \tilde\f_N ) } 
		\frac{\int_{\R^n} Z_{N,\bulk} ( y \one + \tilde\f_N) e^{- \frac{1}{2} t_N^{-1} L^{dN} |y|^2  } \rd y}
		{\int_{\R^n} Z_{N,\bulk} ( y \one ) e^{- \frac{1}{2} t_N^{-1} L^{dN} |y|^2  } \rd y}    \nnb
		& = e^{\frac{1}{2} (\f_N,   \tilde\f_N ) } 
		\frac{\int_{\R^n} Z_{N,\bulk} ( v_N) e^{- \frac{1}{2} |z|^2  } \rd z}
		{\int_{\R^n} Z_{N,\bulk} ( t_N^{1/2} L^{-Nd/2} z \one ) e^{- \frac{1}{2} |z|^2  } \rd z} \label{eq:mgfasintfracres}  
\end{align}
with change of variable $z = t_N^{-1/2} L^{Nd/2} y$
where $\tilde{\f}_N = C^{(\vec\ba)} \f_N$ and $v_N  = t_N^{1/2} L^{-Nd/2} z \one + \tilde{\f}_N$.

The next purely computational lemma is proved in Appendix~\ref{sec:covcomp}.

\begin{lemma} \label{lemma:fNtNlmt}
Following hold for $\f_N = L^{\frac{d}{2} N} f_N$.
\begin{enumerate}
\item $\lim_{N\rightarrow \infty} (\f_N,   \tilde\f_N) = ( \ba^{(\emptyset)} )^{-1} \norm{f}_{L^2 (\T^d)} $.
\item For any $r \in [0,\infty]$ and each $n \in \N$,  there is a constant $C_n >0$ such that
\begin{align}
	\norm{\nabla^n \tilde\f_N}_{\ell^{2+r}} \le C_{n} L^{-nN} L^{-\frac{d r}{2(2 +r)} N} \left( L^{(2-\eta)N} + ( \ba^{(\emptyset)} )^{-1} \right)^{\frac{r}{2+r}} ( \ba^{(\emptyset)} )^{-\frac{2}{2+r}} \norm{\nabla^n f}_{\ell^{\infty}}  
\end{align}
uniformly in $N$,  with $\frac{r}{2+r} = 1$ when $r = \infty$.
\end{enumerate}
\end{lemma}

In the proof below,  we take $\max_{n \le d + p_\Phi} \norm{\nabla^n f}_{\ell^{\infty}}$ as a constant (which is finite because of the assumption $f \in \cS (\T^d ; \R^n)$) and do not write it explicitly.

\begin{proof}[Proof of Theorem~\ref{thm:WNlimit}]

With the parameters chosen as in Proposition~\ref{prop:theRGthm} with $k<\infty$,
we want to prove that
\begin{align}
	\lim_{N\rightarrow \infty} \langle e^{(\f_N,\varphi)} \rangle_{g,\nu,N} \rightarrow \exp \Big( \frac{1}{2\ba^{(\emptyset)}} \norm{f}_{L^2 (\T^d)}^2 \Big)
\end{align}
By \eqref{eq:mgfasintfracres} and Lemma~\ref{lemma:fNtNlmt}(i),  it is sufficient to show that the fraction of integrals on the right-hand side tends to $1$.

To see this,  we take \eqref{eq:theRGthm1}.
Since $e^{u_{N}}$ cancels in the denominator and the numerator, we only need to consider $I_{N} + K_{N}$.  
Each component of $V_N$ satisfies
\begin{align}
	& |\pi_2 V_{N} (\Lambda,  v_N) | 
		\lesssim |\nu_N^{(\emptyset)}| \Big( t_N |z|^2 + \norm{\tilde{\f}_N}_{\ell^{2}}^2  \Big) 
		\le O_L (1) \tilde{g}_N  L^{- (d-2+\eta) N} \Big( t_N |z|^2  + \frac{1}{(\ba^{(\emptyset)} )^2} \Big)  ,    \label{eq:WNlimitpf1} \\
	& 0\le \pi_4 V_{N} (\Lambda,  v_N) 
		\lesssim g_N^{(\emptyset)} \left( L^{-dN} t_N^2  |z|^4 + \norm{\tilde{\f}_N}_{\ell^{4}}^4  \right) \nnb
		&\qquad\qquad\qquad
		 \lesssim \tilde{g}_N L^{-dN}   \Big( t_N^2 |z|^4 + \Big( L^{(2-\eta)N} + 1 / \ba^{(\emptyset)} \Big)^2 (\ba^{(\emptyset)})^{-2} \Big)   \label{eq:WNlimitpf2}
\end{align}
where in both \eqref{eq:WNlimitpf1} and \eqref{eq:WNlimitpf2},  the final inequalities use Lemma~\ref{lemma:fNtNlmt}(ii) to bound $\tilde{\f}_N$ and the restriction $V_{N} \in \tilde\cD_{N}$ to bound $\nu_N^{(\emptyset)}$ and $g_N^{(\emptyset)}$. 
Other terms can be bounded using the same principle, but since $\nabla \one \equiv 0$, we can ignore the constant field for $\pi_{2,\nabla} V_N$ and $\pi_{4,\nabla} V_N$.
With $\km_1 \in \ko_{2,\nabla}$ and $\km_2 \in \ko_{4,\nabla}$,  by Lemma~\ref{lemma:fNtNlmt}(ii),
\begin{align}
	|\pi_{\km_1} V_{N} (\Lambda,  v_N) |  &\lesssim |\nu^{(\km_1)}_N| L^{-q(\km_1) N} \big( \ba^{(\emptyset)} \big)^{-2} \\ 
	| \pi_{\km_2} V_{N} (\Lambda,  v_N) | &\lesssim |g^{(\km_2)}_N| L^{-( q (\km_2) + d) N} (L^{(2-\eta)N} + 1 /\ba^{(\emptyset)}  )^2 (\ba^{(\emptyset)})^{-2}  .
\end{align}
From the restriction $V_{N,\bulk} \in \tilde\cD_{N}$,  we see that they all tend to 0 as $N\rightarrow \infty$. 
Also,  since $\norm{W_{N,V} (\Lambda)}_{h_N, T_N (0)} \le O_L (1) (\tilde{g}_N \scale_N)^{1/2}$ due to \eqref{eq:FWPbounds3obsapplied},
along with Lemma~\ref{lemma:polynorm},
\begin{align}
	|\pi_\bulk W_{N,V} (\Lambda,v_N)| \le O_L (1) (\tilde{g}_N \scale_N)^{1/2} \big( 1 + \norm{v_N}_{h_N, \Phi_N} \big)^6 .   \label{eq:WNbnd}
\end{align}
But by the Sobolev inequality \eqref{eq:logtildeGNbnd} and Lemma~\ref{lemma:fNtNlmt}
\begin{align}
	\norm{v_N}_{h_N, \Phi_N}^2 
		\lesssim
		L^{-jN} \max_{n \le d+ p_\Phi} L^{2nN} \norm{\nabla^n v_N}^2_{\ell^{2}}	 / h_{N,\bulk}^2
		\lesssim
		\tilde{g}_N^{1/2} L^{-\frac{d}{2} N} \big( t_N  |z|^2  + 1/ ( \ba^{(\emptyset)} )^2  \big) ,
\end{align}
but since $t_N < 1/ \ba^{(\emptyset)}$ by Proposition~\ref{prop:Gammajbounds2},
so $|\pi_{\bulk} W_N (\Lambda,v_N) | = o (1 + |z|^6)$ as $N\rightarrow \infty$.

Next,  $K_{N,\bulk} \in M \cK_{N}$ and Lemma~\ref{lemma:KNbnd} imply
\begin{align}
	| K_N (\Lambda, v_N) | 
		\le O_L (1) \chi_N^{3/2} \tilde{g}_N^{\frac{3}{4}} \scale_N^{\kpe} \tilde{G}_N (\Lambda,  v_N) e^{-\kappa L^{-dN} \norm{v_N / h_{N,\bulk} }_{\ell^2} }  .
\end{align}
Since $v_N - \tilde{\f}_N$ is constant-valued,  Lemma~\ref{lemma:logtildeGNbnd} implies $\tilde{G}_N (\Lambda, v_N) = \tilde{G}_N (\Lambda ,  \tilde{\f}_N)$ and by \eqref{eq:logtildeGNbnd} with Lemma~\ref{lemma:fNtNlmt}(ii),
\begin{align}
	\log \tilde{G}_N (\Lambda, \tilde{\f}_N) \lesssim L^{-(2-\eta)N} \sum_{n \le d + p_{\Phi}} L^{2nN} L^{-2nN} \lesssim L^{-(2-\eta)N}
\end{align}
so $|K_{N,\bulk} (\Lambda, v_N) | = o(1)$ as $N\rightarrow \infty$,  uniformly in $z$.
If we plug in $f \equiv 0$,  we see that the same bounds hold true for $Z_{N,\bulk} (t_N^{1/2} L^{-\frac{dN}{2}} z \one)$.
Therefore,  along with the the obtained bounds,  we can use the Dominated convergence theorem to see that
\begin{align}
	\int_{\R^n} e^{- \frac{1}{2} |z|^2  } \rd z 
		&= \lim_{N\rightarrow \infty}  \int_{\R^n} Z_{N,\bulk} ( v_N ) e^{- \frac{1}{2} |z|^2  } \rd z \nnb
		&=   \lim_{N\rightarrow \infty} \int_{\R^n} Z_{N,\bulk} ( t_N^{1/2} L^{-Nd/2} z \one ) e^{- \frac{1}{2} |z|^2  } \rd z ,
\end{align}
as desired. 
\end{proof}

\subsection{Proof of Theorem~\ref{thm:NGlimit}(i): non-Gaussian limit}
\label{sec:NGlimitproof}

We now take $\g_N = f_N / \sb_N$,  with $\sb_N$ from \eqref{eq:bNcN},
so that $(f_N ,\varphi) / \sb_N = (g_N ,\varphi)$.
We take Lemma~\ref{lemma:Vtilde} and the second equality of Proposition~\ref{prop:mgfref} as the starting point.  

The parameters are set as in Proposition~\ref{prop:theRGthm} with $k=\infty$,  so in particular $\nu_0^{(\emptyset)} = \nu = \nu_c$ and $\ba^{(\emptyset)}=0$.
Again,  with only bulk part of the RG flow taken into consideration,  we have
\begin{align}
	\big\langle e^{(\varphi,f_N) / \sb_N} \big\rangle_{g, \nu_c , N}
		&= e^{\frac{1}{2} (\g_N ,  \tilde\g_N)} 
		\frac{\int_{\R^n} Z_{N,\bulk} ( y \one + \tilde{\g}_N) e^{( \g_N  , y \one)}  \rd y }{\int_{\R^n} Z_{N,\bulk}  ( y \one )  \rd y } \nnb
		&= e^{\frac{1}{2} (\g_N ,  \tilde\g_N)} 
		\frac{\int_{\R^n} Z_{N,\bulk} ( v_N ) e^{ (g_N^{(\emptyset)} )^{-1/4} L^{-dN/4} ( \g_N  ,  z \one)}  \rd z }{\int_{\R^n} Z_{N,\bulk}  ( (g_N^{(\emptyset)} )^{-1/4} L^{-d N / 4}  z \one)  \rd z }
		\label{eq:mgfNGalt}
\end{align}
with change of variable $z = L^{\frac{d}{4} N} (g_N^{(\emptyset)} )^{1/4} y$
where $\tilde{\g}_N = w_N \g_N$ and $v_N = y \one + \tilde{\g}_N$.
We again need a purely computational lemma proven in Section~\ref{sec:covcomp}.

\begin{lemma} \label{lemma:gNtNlmt}
Following hold for $\g_N$ when $\ba^{(\emptyset)} =0$.
\begin{enumerate}
\item $\lim_{N\rightarrow \infty} (\g_N,   \tilde\g_N) = 0$.
\item For any $r \in [0,\infty]$,  there is a constant $C_n >0$ such that
\begin{align}
	\norm{\nabla^n \tilde\g_N}_{\ell^{2+r}} \le C_{n} \tilde{g}_N^{1/4} L^{- (n + \frac{d}{4} ) N} L^{-\frac{d r}{2(2 +r)} N} L^{(2-\eta)N} \norm{\nabla^n f}_{L^{\infty}}
\end{align}
uniformly in $N$,  with $\frac{r}{2+r} = 1$ when $r = \infty$.
\end{enumerate}
\end{lemma}

Again,  we consider $\max_{n \le d + p_\Phi} \norm{\nabla^n f}_{\ell^{\infty}}$ as a constant and do not write it explicitly.   Next lemma shows that the scaling of \eqref{eq:mgfNGalt} is natural.

\begin{lemma} \label{lemma:Vtozfour}
For $f \in C^{\infty} (\T^d)$ and $V_{N} \in \tilde\cD_{N,\bulk}$,
\begin{align}
	\lim_{N\rightarrow \infty} I_{N}^{\Lambda} (v_N ) = \lim_{N\rightarrow \infty} e^{-V_{N}^{\stable} (\Lambda,  v_N ) } = e^{-\frac{1}{4} |z|^4}
	\label{eq:Vtozfourm1}
\end{align}
and they are uniformly integrable in $z$.
\end{lemma}
\begin{proof}
For the second limit,  we have
\begin{align}
	\pi_4 V_{N} \big( \Lambda ,  y \one \big) & = \frac{1}{4} |z|^4  ,  \label{eq:Vtozfour0}
\end{align}
so we only have to check that the other terms tend to 0.  Indeed,  we have 
\begin{align}
	| \nu_N^{(\emptyset)} | \sum_{x \in \Lambda} | y |^2
		\lesssim |\nu_N^{(\emptyset)} | \tilde{g}_N^{-1/2} L^{\frac{d}{2} N} |z|^2
		\le O_L (1) \tilde{g}_N^{1/2} L^{-\frac{d-4 + 2\eta}{2} N} |z|^2 
			\label{eq:Vtozfour1}
\end{align}
and $\pi_{\km_1} V_{N,\bulk} (y \one) = \pi_{\km_2} V_{N,\bulk} (y \one) = 0$. 

Also,  by Lemma~\ref{lemma:gNtNlmt}(ii),  for $\km_1 \in \ko_{2,\nabla}$ and $\km_2 \in \ko_{4,\nabla}$,
\begin{align}
	| \nu_N^{(\emptyset)} | \sum_{x \in \Lambda} | \tilde{\g}_N (x) |^2
		& \lesssim |\nu_N^{(\emptyset)} | \,  \tilde{g}_N^{1/2} L^{-\frac{d-8+4\eta}{2} N} \lesssim \tilde{g}_N^{3/2} L^{-\frac{3(d-4+2\eta)}{2} N}  ,   \label{eq:Vtozfour2} \\
	| \nu_N^{(\km_1)} | \sum_{x \in \Lambda} | S^{(\km_1)}_{x} (\tilde{\g}_N) | 
		& \lesssim |\nu_N^{(\km_1)} | \, \tilde{g}_N^{1/2} L^{-\frac{d-8+4\eta}{2} N}	L^{-q (\km_1) N} ,  \label{eq:Vtozfour3} \\
	| g_N^{(\emptyset)} | \sum_{x \in \Lambda} | \tilde{\g}_N (x) |^4  
		& \lesssim  \tilde{g}^2_N L^{-2 (d-4+2\eta) N} ,  \label{eq:Vtozfour4} \\
	| g_N^{(\km_2)} | 	 \sum_{x \in \Lambda} | S_x^{(\km_2)} ( \tilde{\g}_N ) |
		& \lesssim |g_N^{(\km_2)} | \, \tilde{g}_N L^{- (2 (d-4+2\eta) + q (\km_2) ) N} , \label{eq:Vtozfour5}
\end{align}
which all tend to 0 by the assumption $V_{N,\bulk} \in \tilde\cD_{N,\bulk}$.

For the first limit of \eqref{eq:Vtozfourm1},  by \eqref{eq:WNbnd},
\begin{align}
	| \pi_{\bulk} W_{N,V} (\Lambda,  v_N) | \le O_L (1) (\tilde{g}_N \scale_N)^{1/2} \big( 1+ \norm{v_N}_{h_N, \Phi_N} \big)^6 ,
\end{align}
but now Lemma~\ref{lemma:gNtNlmt} and \ref{lemma:logtildeGNbnd} give
\begin{align}
	\norm{v_N}_{h_N,\Phi_N}^2 
		\lesssim L^{-dN} \sum_{n \le d + p_{\Phi}} L^{2n N} \norm{\nabla^n v_N}^2_{\ell^2} / h_{N,\bulk}^2
		\lesssim |z|^2 + L^{-(d-4+2\eta)N} \tilde{g}_N ,
\end{align}
so $W_{N,V} (\Lambda,  v_N) = o (1 + |z|^6)$ as $N\rightarrow \infty$.

Uniform integrability follows because the quartic term $\pi_4 V_N$ dominates the integral.  
\end{proof}

\begin{proof}[Proof of Theorem~\ref{thm:NGlimit}(i)]
By \eqref{eq:mgfNGalt} and Lemma~\ref{lemma:gNtNlmt}(i),  it is sufficient to prove for some $c_3 >0$
\begin{align}
	\lim_{N\rightarrow \infty} \frac{\int_{\R^n} Z_{N} ( v_N ) e^{ (g_N^{(\emptyset)} )^{-1/4} L^{-dN/4} ( \g_N  ,  z \one)}  \rd z }{\int_{\R^n} Z_{N}  ( (g_N^{(\emptyset)} )^{-1/4} L^{-\frac{d}{4} N}  z \one)  \rd z } 
		= \frac{\int_{\R^n} e^{c_3^{-1/4} (f, \one) \cdot z} e^{-\frac{1}{4}  |z|^4 } \rd z }{\int_{\R^n}  e^{-\frac{1}{4}  |z|^4 } \rd z } . 
	\label{eq:thmNGlimitpf1}
\end{align}
We already have an estimate of $e^{-V_{N}^\stable}$ in Lemma~\ref{lemma:Vtozfour},  and
\begin{align}
	(g_N^{(\emptyset)})^{-\frac{1}{4}} L^{-\frac{d}{4}N} (\g_N, y\one) = z \cdot (f_N,\one) \times \begin{cases}
		(N g_N^{(\emptyset)})^{-1/4} & (d=4) \\
		(g^{-1} g_N^{(\emptyset)})^{-1/4} & (d \ge 5) 
	\end{cases}
\end{align}
by definition,
and by Lemma~\ref{lemma:gNasymp},  we have $N g_N \sim \bfb^{-1}$ for $d=d_{c,u}$ and $g_N /g \sim g_{\infty} /g$ for $d > d_{c,u}$.
Thus if we can prove $\int_{\R^n} K_{N,\bulk}(\Lambda,  (g_N^{(\emptyset)})^{-\frac{1}{4}} L^{-\frac{d}{4}N} z \one ) \rd z$ and $\int_{\R^n} K_{N,\bulk}(\Lambda, v_N) \rd z$ tend to 0 as $N\rightarrow \infty$,  then by the Dominated convergence theorem,  we have \eqref{eq:thmNGlimitpf1} with
\begin{align}
	c_3 = \begin{cases}
			\bfb^{-1} & (d = d_{c,u}) \\
			g^{(\emptyset)}_{\infty} / g & (d > d_{c,u}) .
		\end{cases} 
		\label{eq:c3defi}
\end{align} 

To bound $K_{N,\bulk}$,  observe that $K_{N,\bulk} \in M \cK_{N}$ and Lemma~\ref{lemma:KNbnd} imply
\begin{align} \label{eq:KNLambdavN}
	| K_{N,\bulk} (\Lambda, v_N) | \le O_L (1) \tilde{g}_N^{\frac{3}{4}} \scale_N^{\kpe} \tilde{G}_N (\Lambda,  v_N) e^{-\kappa L^{-dN} \norm{v_N / h_{N,\bulk} }_{\ell^2}^2 }  .  
\end{align}

By Lemma~\ref{lemma:logtildeGNbnd},  we have $\tilde{G}_N (\Lambda, v_N) = \tilde{G}_N (\Lambda ,  \tilde\g_N)$ and by \eqref{eq:logtildeGNbnd} with Lemma~\ref{lemma:gNtNlmt}(ii),
\begin{align}
	\log \tilde{G}_N (\Lambda,  \tilde\g_N ) 
		\lesssim L^{-(2-\eta)N} \sum_{n \le d + p_{\Phi}} \tilde{g}_N^{1/2} L^{(2-\eta)N} L^{-\frac{d}{2} N}  \lesssim \tilde{g}_N^{1/2} L^{-\frac{d-4+2\eta}{2}N} ,
\end{align}
while again by Lemma~\ref{lemma:gNtNlmt}(ii),  since $|x+y|^2 \ge \frac{1}{2} |x|^2 - |y|^2$,
\begin{align}
	L^{-dN} \Big\| \frac{v_N }{ h_{N,\bulk} } \Big\|_{\ell^2}^2
		\ge  \frac{1}{2} \Big( \frac{\tilde{g}_N}{g_N^{(\emptyset)}} \Big)^{1/2} |z|^2 - C L^{- (d-4 +2\eta)N} \tilde{g}_N  \label{eq:KNLambdavN3}
\end{align}
so $\int_{\R_n} K_N (\Lambda,  v_N) \rd z \rightarrow 0$ as $N\rightarrow \infty$.
The same bounds hold true for $K_{N,\bulk}(\Lambda,  g_N^{-\frac{1}{4}} L^{-\frac{d}{4}N} z \one )$,  i.e.,  for $f \equiv 0$,  completing the proof. 

\end{proof}

\subsection{Proof of Theorem~\ref{thm:NGlimit}(ii): free field limit}

In what follows, we always assume that $\int_{\T^d} f (x) \rd x = 0$.  Proof of part (ii) requires scaling $\sh_N = ( f_N - \Phi_N (f_N) ) / \sc_N = L^{\frac{d-2+\eta}{2} N} ( f_N - \Phi_N (f))$,  so that $(\varphi, f_N - \Phi_N (f_N)) / \sc_N = (\varphi,  \sh_N)$.  Again,  we take Lemma~\ref{lemma:Vtilde} and the second equality of Proposition~\ref{prop:mgfref} as the starting point.  
Parameters are set exactly as in Section~\ref{sec:NGlimitproof}.

With only bulk part of the RG flow taken into consideration,  we have
\begin{align}
	\big\langle e^{(\varphi,  f_N - \Phi_N (f_N)) / \sc_N} \big\rangle_{g, \nu_c , N}
		&= e^{\frac{1}{2} (\sh_N ,  w_N \sh_N)} 
		\frac{\int_{\R^n} Z_{N,\bulk} ( y \one +  w_N \sh_N )  \rd y }{\int_{\R^n} Z_{N,\bulk}  ( y \one )  \rd y } \nnb
		&= e^{\frac{1}{2} (\sh_N ,  w_N \sh_N)} 
		\frac{\int_{\R^n} Z_{N,\bulk} ( v_N )  \rd z }{\int_{\R^n} Z_{N,\bulk}  ( \tilde{g}_N^{-1/4} L^{-dN/4} z \one )  \rd  z }
		\label{eq:mgfNGalt2}		
\end{align}
where $w_N \sh_N = \tilde{\sh}_N$,  $z = \tilde{g}_N^{1/4} L^{\frac{d}{4} N} y$ and $v_N = y \one + \tilde{\sh}_N = \tilde{g}_N^{-1/4} L^{-\frac{d}{4}} z \one + \tilde{\sh}_N$.
Note that the second equality of Proposition~\ref{prop:mgfref} should contain $e^{( \sh_N  , y \one)}$ in the numerator,  but by our definition,  $(\sh_N, \one) = 0$,  so it is not included above.
Proof of next lemma is deferred to Appendix~\ref{sec:covcomp}.

\begin{lemma} \label{lemma:hNtNlmt}
Following hold for $\sh_N$ when $\ba^{(\emptyset)} = 0$.

\begin{enumerate}
\item $\lim_{N\rightarrow \infty} (\sh_N,   w_N \sh_N) = (f - \Phi(f) , (-\Delta)^{-1+\eta/2} (f- \Phi(f))) \times \begin{cases}
(1 + \bar{\ba}_{\Delta})^{-1} &(\eta=0) \\
1 & (\eta >0) 
\end{cases}$.
\item For any $r \in [0,\infty]$,  there is a constant $C_n >0$ such that
\begin{align}
	\norm{\nabla^n \tilde\sh_N}_{\ell^{2+r}} \le C_{n} L^{(\frac{2-\eta}{2} - n)N} L^{-\frac{d r}{2(2 +r)} N} \norm{\nabla^n f}_{\ell^{\infty}}  
\end{align}
uniformly in $N$,  with $\frac{r}{2+r} = 1$ when $r = \infty$.
\end{enumerate}
\end{lemma}

\begin{proof}[Proof of Theorem~\ref{thm:NGlimit}(ii)]
By the previous lemma,  if we set $c_4 (d,\eta) = (1+ \bar{\ba}_{\Delta})^2$ when $\eta=0$ and $c_4 (d,\eta) = 1$ when $\eta >0$,  it is sufficient to show 
\begin{align}
	\lim_{N\rightarrow \infty} \frac{\int_{\R^n} Z_{N,\bulk} ( v_N ) \rd z }{\int_{\R^n} Z_{N,\bulk}  ( \tilde{g}_N^{-1/4} L^{-\frac{d}{4}} z \one )  \rd z } \label{eq:mgfNGalt2fracpt}
\end{align}
is equal to $1$.
By Lemma~\ref{lemma:KNbnd},  we can bound $| K_{N,\bulk} (\Lambda, v_N) | \le O_L (1)  \tilde{g}_N^{\frac{3}{4}} \scale_N^{\kb_1} \tilde{G}_N (\Lambda,  v_N) e^{-\kappa L^{-dN} \norm{v_N / h_{N,\bulk} }_{\ell^2}^2 }$ (see \eqref{eq:KNLambdavN}),
and using Lemma~\ref{lemma:logtildeGNbnd} with Lemma~\ref{lemma:hNtNlmt}(ii),
\begin{align}
	\log \tilde{G}_N (\Lambda,  v_N ) 
		= \log \tilde{G}_N (\Lambda,  \sh_N) 
		\lesssim L^{-dN} \sum_{n \le d + p_{\Phi}} L^{(d-2+\eta) N} L^{(2-\eta)N}  \lesssim 1 .
\end{align}
Also,  for some $c,C >0$,   just as for \eqref{eq:KNLambdavN3},
\begin{align}
	L^{-dN} \Big\| \frac{v_N }{ h_{N,\bulk} } \Big\|_{\ell^2 (\Lambda)}^2 \ge c |z|^2 - C \tilde{g}_N^{1/2} L^{-\frac{d-4+2\eta}{2} N} .
\end{align}
Putting together, 
\begin{align}
	| K_{N} (\Lambda, v_N) | \le O_L (1) \tilde{g}_N^{\frac{3}{4}} \scale_N^{\kpe} e^{- c \kappa |z|^2 } ,
\end{align}
so $K_{N}$ has vanishing effect in the integrals of \eqref{eq:mgfNGalt2fracpt}.

Bounds \eqref{eq:Vtozfour0}--\eqref{eq:Vtozfour5} appear almost the same: 
\begin{align}
	\big| \pi_4 V_{N} (\Lambda ,  \tilde{g}_N^{-1/4} L^{-\frac{d}{4} N}  z \one ) \big| &\asymp \tilde{g}_N^{-1} g_N^{(\emptyset)}  |z|^4 \asymp |z|^4 \\
	\big| \pi_2 V_{N} (\Lambda ,  \tilde{g}_N^{-1/2} L^{-\frac{d}{4} N}  z \one ) \big| &\lesssim \tilde{g}_N^{-1/2} | \nu_N^{(\emptyset)} | L^{\frac{d}{2} N} |z|^2 \le O_L (1)  \tilde{g}_N^{1/2} L^{-\frac{d-4+2\eta}{2} N} |z|^2
\end{align}
and by Lemma~\ref{lemma:hNtNlmt},
\begin{align}
	| \pi_2 V_{N} (\Lambda ,  \tilde\sh_N ) | & \lesssim |\nu_N^{(\emptyset)} |  L^{(2-\eta)N}  ,   \label{eq:Vtozfour22} \\
	|\nu_N^{(\km_1)} |  \sum_{x \in \Lambda} | S^{(\km_1)}_x ( \tilde\sh_N ) | & \lesssim |\nu_N^{(\km_1)} |  L^{(2-\eta)N}	L^{-q (\km_1) N} ,  \label{eq:Vtozfour32} \\
	| \pi_4 V_{N} (\Lambda ,   \tilde\sh_N ) | & \lesssim | g_N^{(\emptyset)} |  L^{-(d-4+2\eta) N}  ,  \label{eq:Vtozfour42} \\
	|g_N^{(\km_2)} |  \sum_{x \in \Lambda} | S^{(\km_2)}_x ( \tilde\sh_N ) | & \lesssim |g_N^{(\km_2)} | L^{- (d - 4 +2\eta + q (\km_2) ) N} .  \label{eq:Vtozfour52}
\end{align}
Also,  by \eqref{eq:WNbnd},  Lemma~\ref{lemma:hNtNlmt} and \ref{lemma:logtildeGNbnd},
\begin{align}
	|\pi_\bulk W_{N,V} (\Lambda, \tilde{\sh}_N) | &\le O_L (1) (\tilde{g}_N \scale_N)^{1/2} (1 + \norm{\tilde{\sh}_N}_{h_N,\Phi_N})^6 \\
	\norm{\tilde{\sh}_N}_{h_N,\Phi_N}^2 &\lesssim \tilde{g}_N^{1/2} L^{-\frac{d-4+2\eta}{2} N} . 
\end{align}
With the restriction $V_{N} \in \tilde\cD_{N,\bulk}$, 
we see that only $\pi_4 V_{N,\bulk} (\Lambda ,  \tilde{g}_N^{-1/4} L^{-dN/4}  z \one )$ survives in the integrals of \eqref{eq:mgfNGalt2fracpt}, which are the same in the denominator and the numerator.  This shows that \eqref{eq:mgfNGalt2fracpt} is simply 1. 
\end{proof}

\section{Stability of the observable flow}
\label{sec:stabobsflow}

To prove Theorem~\ref{thm:infvol2pt}--\ref{thm:plateauGen},  we repeat the strategy of Section~\ref{sec:cotcpfdf}--\ref{sec:mainresults}.
In this section,  we control the full RG flow with the observable fileds included,  and a number of bounds on the observable coefficients will be proved along. 
The initial condition was already constructed in Section~\ref{sec:cotcpfdf},  repeated here.
\begin{equation} \stepcounter{equation}
	\tag{\theequation ${\rm{A}_{IC}}$} \label{asmp:obsRG}
	\begin{split}
		\parbox{\dimexpr\linewidth-4em}{
		Let $L$ be sufficiently large,  $g$ be sufficiently small,  $\vec{\ba} = \vec{\ba}_{c,\infty}$ (so that $\ba^{(\emptyset)}_{c,\infty}=0$),  $|\lambda_\hash| \le 1$ for $\hash\in \{ \o, \x \}$ and $\alpha =1$.
Let $K_0 = 0$ and $V_0$ be given by
$$
\begin{cases}
	\nu_0^{(\emptyset)} = \nu_{c},   \;\; g_0^{(\emptyset)} = g & {} \\
	\nu_0^{(\km_1)} = - p^{(\km_1)} \vec\ba_{c,\infty} & (\km_1 \in \kA_1) \\
	\nu_0^{(\km_1)} = g_0^{(\km_2)} = 0 & (\km_1 \in \kA_2 \cup \kA_3, \; \km_2 \in \ko_{4,\nabla}) \\
	\lambda^{(\emptyset)}_{\hash, 0} = \lambda_\hash  ,  \;\;  \lambda^{(\km_3)}_{\hash, 0} = 0 & (\km_3 \in \ko_{1,\nabla}) .
\end{cases}	
\vspace{-40pt}
$$}
	\end{split}
\end{equation}
\vspace{20pt}

\begin{proposition} \label{prop:obsRG}
Assume \eqref{asmp:obsRG} and $\kt$ be sufficiently small.  Then the RG flow of infinite length exists.  
\end{proposition}

\begin{lemma}  \label{lemma:deltalambda}
Suppose $d > d_{c,u}$ and $(V_j, K_j) \in \D_j$.  Let $\hash \in \{\o,\x\}$ and $\km_3 \in \ko_1 \cup \ko_{1,\nabla}$.
Then
\begin{align}
	|\lambda_{\hash, j+1}^{(\km_3)} - \lambda_{\hash, j}^{(\km_3)}| \le O_L (\tilde{\chi}_j \tilde{g_j}) L^{q(\km_3) j} \times
		\begin{cases}
			L^{-\eta j} & (d = 4, \,  \eta > 0) \\
			L^{-\frac{1}{2} (d-4 + 5\eta)} & (d \ge 5) .
		\end{cases}
	\label{eq:deltalambda}
\end{align}
\end{lemma}
\begin{proof}
Since the flow of $\lambda_{\hash,j}^{(\km_3)}$ is independent of $j_{\ox}$,  we may take $j_{\ox} = \infty$. 
By taking $p(\km) =1$ in \eqref{eq:ko1defi}--\eqref{eq:ko1nabladefi},  $V_\o$ and $V_\x$ can only be linear polynomials of $\varphi$.  
Thus $\E_{j+1} \theta \pi_{\km_3} V = \pi_{\km_3} V$ for any $\km_3 \in \ko_1 \cup \ko_{1,\nabla}$ and if we expand out \eqref{eq:RjUdefi},
\begin{align}
	\sigma_{\x} \big( \lambda_{\hash,j+1}^{(\km_3)} - \lambda_{\hash, j}^{(\km_3)} \big) S^{(\km_3)} 
		= \pi_{\km_3} \big(  R_{j+1} (V_j,K_j) + \Phi^\pt_{j+1} (V_j) - \E_{j+1} \theta V_j \big)
\end{align}
Hence by Lemma~\ref{lemma:VptmthV} and \eqref{eq:controlledRG22},
\begin{align}
	\norm{\sigma_{\x} (\lambda_{\hash,j+1}^{(\km_3)} - \lambda_{\hash, j}^{(\km_3)}) S^{(\km_3)} }_{\cL_{j+1} (\ell)} 
		\le O_L (\tilde{\chi}_{j+1} \tilde{g}_{j+1}^2 \scale_{j+1}^2 ) .  \label{eq:deltalambda1}
\end{align}
On the other hand,  Definition~\ref{defi:cLnorm} indicates
\begin{align}
	\norm{\sigma_{\x} S^{(\km_3)} }_{\cL_{j} (\ell)} & \asymp 
		L^{-q (\km_3) j} \ell_{\sigma,j} \ell_{j} = \tilde{g}_j \ell_0 L^{-q (\km_3) j} \times
		\begin{cases}	
			L^{- \frac{d-2+\eta}{2} j} L^{(1-\frac{3}{2}\eta) j } & (d=4) \\
			L^{-( \frac{d-2+\eta}{2} + d-5+\eta) j } & (d \ge 5)
		\end{cases}
	\label{eq:deltalambda4}
\end{align}
so we have the desired bounds.
\end{proof}

\begin{proof}[Proof of Proposition~\ref{prop:obsRG}]
The case $d=d_{c,u}$ is discussed in Proposition~\ref{prop:obsstablemanifold-dcu}.

We are only left with $d > d_{c,u}$.
By Theorem~\ref{thm:theStableManifold},  the bulk RG flow of infinite length exists,  so $(V_{j,\bulk} ,K_{j,\bulk}) \in \D_{j,\bulk} (\alpha)$ for each $j\ge 0$.
We proceed by induction from here,  so assume that $K_k \in \cK_k$ for each $k \le j$,  and satisfies in addition
\begin{align}
	| \lambda_{\hash, k}^{(\km_3)} - \lambda_{\hash, 0}^{(\km_3)} |
		\le O_L (1) \sum_{m \le k} \tilde{\chi}_{m} \tilde{g}_m L^{q (\km_3) m} \times \begin{cases}
			L^{-\eta m} & (d= 4,  \; \eta >0) \\
			L^{-\frac{1}{2} (d-4+5\eta) m } & (d > d_{c,u})  
		\end{cases} \label{eq:lambdakmlambdan}
\end{align}
for any $\km_3 \in \ko_1 \cup \ko_{1,\nabla}$.
When $d=4,  \eta > 0$,  any $\km_3 \in \ko_1 \cup \ko_{1,\nabla}$ has  $q(\km_3 ) \in \{ 0,1 \}$.  When $q(\km_3) = 0$,  then the sum is bounded by an absolute constant bounded by $O_L (g)$,  so $\lambda_{\hash, k}^{(\km_3)}$ stays inside the RG domain. 
When $q(\km_3 ) = 1$,  then the sum diverges and we see that
\begin{align}
	| \lambda_{\hash, k}^{(\km_3)}  | \le | \lambda_{\hash, 0}^{(\km_3)} | + O_L ( g ) L^{(1-\eta) k} <  C_{\cD} L^{q(\km_3) k}
\end{align}
for sufficiently small $g$,  so $\lambda_{\hash, k}^{(\km_3)}$ stays inside the RG domain. 
When $d >4$ and $q (\km_3) \in \{ 0,1\}$,  then the same principles apply and $\lambda_{\hash, k}^{(\km_3)}$ stays inside the RG domain. 
When $q(\km_3) \ge [2,  \frac{1}{2} (d-4 +5\eta)] $,  then 
\begin{align}
	\sum_{m \le k} \tilde{\chi}_m \tilde{g}_m L^{q(\km_3)m- \frac{1}{2} (d-4+5\eta) m} 
		\le
		\sum_{m \le k} \tilde{\chi}_m \tilde{g}_m
		\le O_L (k g) ,
\end{align}
so $\lambda_{\hash, k}^{(\km_3)}$ stays inside the RG domain. 
When $q (\km_3) \in (  \frac{1}{2} (d-4 +5\eta) ,  \frac{d-2+\eta}{2} )$,  then
$z := \frac{d-2+\eta}{2} - \lceil \frac{d-2+\eta}{2} \rceil +1 >0$,  and
\begin{align}
	| \lambda_{\hash, k}^{(\km_3)}  | 
		\le | \lambda_{\hash, 0}^{(\km_3)} | + O_L ( g ) L^{(q (\km) - \frac{1}{2} (d-4+5\eta)) k} \lesssim 1 + O_L (g) L^{(1-2\eta - z) k} 
		< C_{\cD} \scale_k^\kt L^{(2-\eta) k}
\end{align}
for sufficiently small $\kt$ (compare with \eqref{eq:Dstsigmadefi}).
so $\lambda_{\hash, k}^{(\km_3)}$ stays inside the RG domain. 
In summary,  we have $V_{\hash, k}^{(\km_3)} \in \cD^\st_k (\alpha)$ whenever \eqref{eq:lambdakmlambdan} holds.
We can then apply Lemma~\ref{lemma:deltalambda} to see that
the bound \eqref{eq:lambdakmlambdan} also holds for $k = j+1$. 

There is nothing to say about $K_{j+1}$,  since \eqref{eq:controlledRG24} with $p=q=0$ implies
\begin{align}
	\norm{K_{j+1}}_{\cW_{j+1}} \le C_{\rg} \tilde{\chi}_{j+1}^{3/2} \tilde{g}_{j+1}^3 \scale_{j+1}^{\kae} ,
\end{align}
and the induction proceeds.  It is already proved in the induction process that $(V_j, K_j ) \in \D_j$.
\end{proof}

In the process of the proof,  we deduce convergence of $\lambda^{(\emptyset)}_{\hash, j}$.
We also use $( \hat{\lambda}_{j,\hash}^{(\km)})_{\km \in \ko_1 \cup \ko_{1,\nabla}}$ to denote the coefficients of $\pi_\hash \hat{V}_j$ (recall Definition~\ref{def:WP}).

\begin{corollary} \label{cor:lambdainfty}
Assume \eqref{asmp:obsRG}.  
Let $q(\km) = 0$ when $d=4$ and $q (\km) < \frac{1}{2} (d-4 + 5\eta)$ when $d \ge 5$.  Then there exits $\lambda^{(\km)}_{\hash, \infty} = \lambda^{(\km)}_{\hash,0} (1 + O(g))$ such that
\begin{align}
	| \lambda_{\hash, j}^{(\km)} - \lambda^{(\km)}_{\hash,\infty} | \le O_L (1) \tilde{\chi}_{j} \tilde{g}_j L^{q (\km_3) j} \times \begin{cases}
			L^{-\eta j} & (d=4) \\
			L^{-\frac{1}{2} (d-4+5\eta) j} & (d \ge 5)  .
		\end{cases}
		\label{eq:lambdainfty1}
\end{align}
If $q(\km) = 1$ when $d=4$ and $q (\km) \ge \frac{1}{2} (d-4 + 5\eta)$ when $d \ge 5$,  we have
\begin{align}
	|\lambda_{\hash, j}^{(\km)}| < O(1) + O_L (g) L^{(1-2\eta) k} .
		\label{eq:lambdainfty2}
\end{align}
Both bounds \eqref{eq:lambdainfty1} and \eqref{eq:lambdainfty2} also hold for $\hat{\lambda}_{\hash, j}^{(\km_3)}$,  with the same $\lambda_{\hash, \infty}^{(\km_3)}$.
\end{corollary}
\begin{proof}
For $d > d_{c,u}$,  the first statement follows from the proof of Proposition~\ref{prop:obsRG},  while for $d=d_{c,u}$,  it follows from Proposition~\ref{prop:obsstablemanifold-dcu}.

For the final statement on $\hat{\lambda}_{\hash, j}^{(\km_3)}$,  we may observe from Lemma~\ref{lemma:Qjbound} that
\begin{align}
	\norm{\sigma_\hash ( \lambda_{\hash, j}^{(\km)} -   \hat{\lambda}_{\hash, j}^{(\km)}  ) S^{(\km)} }_{\cL_j (\ell)} \le \norm{Q_j}_{\cL_j (\ell)} \le O_L (1) \tilde{\chi}_j^{3/2} \tilde{g}_j^3 \scale_j^{\kae} ,
\end{align}
which is smaller than \eqref{eq:deltalambda1} multiplied by $\tilde{g}_j$,  so by \eqref{eq:deltalambda4},  we have
\begin{align}
	\big|  \lambda_{\hash, j}^{(\km)} -   \hat{\lambda}_{\hash, j}^{(\km)}  \big| \le O_L (\tilde{\chi}_j  \tilde{g}_j^2 ) L^{q(\km_3) j } \times 		
		\begin{cases}
			L^{-\eta j} & (d = 4, \,  \eta > 0) \\
			L^{-\frac{1}{2} (d-4 + 5\eta)} & (d \ge 5) .
		\end{cases}
\end{align}
We have the desired conclusion once we add this bound on \eqref{eq:lambdainfty1} and \eqref{eq:lambdainfty2}.
\end{proof}

\subsection{Final integral on the torus}

We can now prove a version of Proposition~\ref{prop:theRGthm} for the observable RG flow.

\begin{proposition} \label{prop:finalintegral}
Under the assumptions of Proposition~\ref{prop:theRGthm},  
also for $\hash\in \{\o,\x\}$,  let $|\lambda_{\hash,\o}^{(\emptyset)}| ,  |\lambda_{\hash,\x}^{(\emptyset)}| \le 1$ and $\lambda_{\hash,\o}^{(\km)} = \lambda_{\hash,\x}^{(\km)} =0$ for $\km \neq \emptyset$.  Then $Z_N = \E_{w_N} \theta \exp(-V_{0} (\Lambda))$ satisfies
\begin{align}
	Z_N = e^{- u_N  (\Lambda_N)} (I_N + K_N)
\end{align}
where $I_N = \cI_N (V_N)$,  $K_N \in M \cK_{N}$ for some (possibly $L$-dependent constant) $M >0$ and $u_N = \sum_{j=1}^{N-1} \delta u_j + \delta u_{N}^{\Lambda_N}$ for some $(\delta u_j)_{j \le N-1}$ that does not depend on $N$.   
Moreover,  if $u_{N,\bulk}$ and $V_{N,\bulk}$ are those given by Proposition~\ref{prop:theRGthm},  we have $\pi_\bulk (u_N,  V_N ) = (u_{N,\bulk} ,V_{N,\bulk})$ and if $\lambda_{N,\hash}^{(\km)}$ are the coefficients of $\pi_\hash V_N$ ($\hash \in \{\o,\x\}$),  then they satisfy the estimates of Corollary~\ref{cor:lambdainfty}.
\end{proposition}

\begin{proof}
We proceed just as in the proof of Proposition~\ref{prop:theRGthm}.  When we set initial conditions \eqref{asmp:obsRG},  by Theorem~\ref{thm:infvolRGmap},  the finite volume RG flow up to scale $j<N$ can obtained as a projection of the infinite volume observable RG flow constructed by Proposition~\ref{prop:obsRG}.  Thus we obtain $(V_j, K_j)_{j\le N-1}$,  and $(\delta u_j )_{j\le N-1}$ identical to those obtained in the infinite volume. 

We now consider $(\delta u_N, V_N, K_N) := \Phi_N (V_{N-1}, K_{N-1})$.  
By \eqref{eq:controlledRG23},  we obtain $K_N \in M \cK_N$ for $M = M_{0,0} / (2C_{\rg})$.
Since the RG map respects the graded structure (Definition~\ref{defi:contrlldRG2}),  it satisfies $\pi_\bulk (u_N,  V_N ) = (u_{N,\bulk} ,V_{N,\bulk})$.
Also,  Lemma~\ref{lemma:deltalambda} still holds in this setting,  we obtain the bounds of Corollary~\ref{cor:lambdainfty} by simply adding \eqref{eq:deltalambda} on the estimates on $\lambda_{N-1, \hash}^{(\km)}$ in Corollary~\ref{cor:lambdainfty}.
\end{proof}

\section{Plateau}
\label{sec:plateauproof}

Now,  we complete the proof of Theorem~\ref{thm:infvol2pt} and \ref{thm:plateauGen},  following the strategy of Section~\ref{sec:mainresults},  but using Proposition~\ref{prop:2ptfncrestt} in place of Proposition~\ref{prop:mgfref}.
We introduce the natural scale
\begin{align}
	h'_N = (g_N^{(\emptyset)})^{-1/4} L^{-dN/4} 
\end{align}
to state the following intermediate result for the two-point function,  proved at the end this section.

\begin{proposition} \label{prop:plateauInt}
Assume \eqref{asmp:obsRG} and $Y$ be an $\R^n$-valued random variable as in Lemma~\ref{lemma:NGcnst}.  Then
\begin{align}
	\langle \varphi_\o^{(1)} \varphi_\x^{(1)} \rangle_{g,\nu, \Lambda} &=  w_N (\x) +  \frac{(h'_N)^2 \E [ |Y|^2] }{n} + \psi_1 (\x) + \psi_2 (\x,N)
\end{align}
as $N\rightarrow \infty$,  
where $\psi_1 (\x)$ is a function independent of $N$ such that $\lim_{|\x|\rightarrow \infty} |\x|^{d-2+\eta} \psi_1 (\x) =0$ and $|\psi_2 (\x , N)| \le c_N ( (h'_N)^2 + (1 \vee |\x|)^{-(d-2+\eta)} )$ for some sequence $\lim_{N\rightarrow \infty} c_N =0$.
\end{proposition}

This implies the main theorems on the two-point function  due to observations made in Appendix~\ref{sec:asympGreenfnc}.

\begin{proof}[Proof of Theorem~\ref{thm:infvol2pt} and \ref{thm:plateauGen}]

Theorem~\ref{thm:infvol2pt} is obtained almost directly. 
By Proposition~\ref{prop:plateauInt} and Lemma~\ref{lemma:GreenminuswN},  we have
\begin{align}
	\langle \varphi_\o^{(1)} \varphi_\x^{(1)} \rangle_{g,\nu, \Lambda_N}	= C^{(\vec{\ba})}_{\Z^d} (\x) + \psi_1 (\x) + \psi_2 (\x,N) + \frac{(h'_N)^2 \E[|Y|^2]}{n} + O(L^{-(d-2+\eta)N})
	\label{eq:infolv2ptpf1}
\end{align}
so we obtain the desired limit with
\begin{align}
	\mathbb{\C}_\x =  C^{(\vec{\ba})}_{\Z^d} (\x) + \psi (\x) ,
\end{align}
and it follows the desired asymptotic due to Lemma~\ref{lemma:Greenasymp},  with choices
$c_1 = (1 + \bar{\ba}_{c,\Delta})^{-1} \gamma$ when $\eta =0$ and $c_1 = \gamma$ when $\eta \neq 0$.
Also,  by the final sentence of Theorem~\ref{thm:theStableManifold},
we see $\bar{\ba}_{c,\Delta} = O(g)$,  and we have the desired estimates on $c_1$.

For the proof of Theorem~\ref{thm:plateauGen}, 
we expand on $(h'_N)^2 = (g_N^{(\emptyset)})^{-1/2} L^{-dN/2}$ inside \eqref{eq:infolv2ptpf1}.
When $d > d_{c,u}$,  we can use Lemma~\ref{lemma:gNasymp} to replace $g_N^{(\emptyset)} = g_\infty (1 + o(1))$ for some $g_{\infty} = g + O(g^2)$,  and when $d = d_{c,u}$,  we can use Lemma~\ref{lemma:gNasymp} to replace $g_N^{(\emptyset)} = (\bfb N)^{-1} ( 1 + O( 1 / \log N ) )$,  thus the desired identity follows with
\begin{align}
\begin{cases}
	c_1 = \bfb^{1/2} =\sqrt{n+8} /4\pi & (d=d_{c,u} ) , \\
	c_2 = (g/g_\infty)^{1/2} = 1+ O(g) & (d > d_{c,u}) .
\end{cases}
\end{align}
\end{proof}

\subsection{Two-point function}

Recall $h'_N = (g_N^{(\emptyset)})^{-1/4} L^{-dN/4}$ and
consider the Lebesgue measure $\sm (F) = \int_{\R^n} F(z) \rd z$.
Under \eqref{asmp:obsRG},  due to Proposition~\ref{prop:2ptfncrestt},  
\begin{align}
	\lambda^{(\emptyset)}_{0, \o} \lambda^{(\emptyset)}_{0, \x}	\frac{\E_{C^{(\vec{\ba})}} [ \varphi_\o^{(1)} \varphi_\x^{(1)} e^{-V_{0,\bulk} (\Lambda, \varphi)  }  ] } {\E_{C^{(\vec{\ba})}} [ e^{-V_{0,\bulk} (\Lambda, \varphi) }  ]}
		& = 
		\frac{\int_{\R^n} Z_{N,\ox} ( y \one ) \rd y }{\int_{\R^n} Z_{N,\bulk}  ( y \one )  \rd y} 
		=
		\frac{\sm \big[  Z_{N,\ox} ( h'_N z \one ) \big] }{\sm \big[ Z_{N,\bulk}  ( h'_N z \one ) \big]}   
\end{align}
after a change of variable $y = h'_N z$ to obtain the second equality.
By expanding in the observable field,  
\begin{align}
	\pi_\ox e^{u_{N,\bulk} |\Lambda|} Z_N 
		& =  - \sigma_\ox u_{N,\ox} (\Lambda) Z_{N,\bulk} + \pi_\ox I_N^\Lambda +  \sigma_\ox K_{N,\ox} (\Lambda)  
\end{align}
so we have the following.

\begin{corollary} \label{cor:crlfncreform}
Under \eqref{asmp:obsRG} and $\alpha \le 1$,
\begin{align}
	\lambda^{(\emptyset)}_{0, \o} \lambda^{(\emptyset)}_{0, \x}	\frac{\E_{C^{(\vec{\ba})}} [ \varphi_\o^{(1)} \varphi_\x^{(1)} e^{-V_{0,\bulk} (\Lambda, \varphi)  }  ] } {\E_{C^{(\vec{\ba})}} [ e^{-V_{0,\bulk} (\Lambda, \varphi) }  ]}
		&=  - u_{N,\ox} (\Lambda) + A_{N,\ox} + B_{N,\ox} ,  \label{eq:corcrlfncreform1} 
\end{align}
where $\delta u_{N} = u_{N} - u_{N-1}$ and
\begin{align}
	A_{N}  &= \frac{\sm [  K_{N} (\Lambda,  z h'_N \one ) ] }{e^{u_{N,\bulk} |\Lambda|}  \,  \sm [ Z_{N,\bulk} ( z h'_N \one) ]},\qquad
	B_{N} =  \frac{\sm [ I_N^{\Lambda} ( z h'_N \one ) ]  }{e^{u_{N,\bulk} |\Lambda|} \,  \sm[ Z_{N,\bulk} (z h'_N \one ) ]} .
\end{align}
\end{corollary}

\subsection{Observable projection of the perturbative map}

In Definition~\ref{def:WP},  we did not specify $P_{j,V}$,  but we will need $\pi_\ox P_{j,V}$ to precisely control the flow of $u_{j,\ox}$.  
For functions $F, G$ valued in $\sum_{* \in \{\bulk,\o,\x,\ox\}} \sigma_* \R$,  let $\pi_* F = \sigma_* F_*$ and $\pi_* G = \sigma_* G_*$.  We use the following notations.
\begin{itemize}
\item Define $\Loc^{\ox} F (\varphi) = \sigma_\ox F_{\ox} (0) $.
\item For a covariance matrix $C$,  let
\begin{align}
	\F_{C} [F ; G] &= \Cov_C [ e^{-\frac{1}{2} \Delta_C } F ; e^{-\frac{1}{2} \Delta_C } G ] ,\\
	\F_{\pi,C} [F; G] &= \F_{C} [F ; \pi_\bulk G] + \F_C [ (1-\pi_\bulk)F ; G  ]  ,  \\ 
	\Cov_{\pi,C} [F;G]&= \Cov_C [ F ; \pi_\bulk G  ] + \Cov_{C} [ (1-\pi_\bulk ) F ; G  ] .
\end{align}
As for the case of the expectation,  we denote $\Cov_{\pi,+} = \Cov_{\pi,  \Gamma_+}$.
Also,  as a function of $V \in \cV$ and covariance $C$,  define the $\pi_{\ox}$-projection of 
\begin{align}
	\pi_{\ox} \mathbb{W}_{C, V} (\{x\}) = \frac{1}{2} \big( 1- \Loc^\ox \big) \mathbb{F}_{\pi, C} [V (\{ x \}) ; V(\Lambda)] .
\end{align}
\end{itemize}

The localisation seems slightly different from that of \cite{FSmap},  but it is nevertheless the same for $\sigma_\ox$-functions.
We also specialise \cite[Definition~4.10,  4.13]{FSmap} to the $\ox$-components.  (The definitions are essentially the same for the other components,  but just the localisation is more difficult to define.)

\begin{definition} \label{defi:pioxWP}
When $j < N$,  the $\ox$-projection of $W_{j,V}$ is defined as
\begin{align}
	\pi_{\ox} W_{j, V} =  \pi_{\ox}  \mathbb{W}_{w_j ,  V_j }
		\label{eq:pioxWjdefi}
\end{align}
The $\pi_\ox$-projection of $P_{j,V}$ is defined as
\begin{align}
	\pi_{\ox} P_{j,V} = \begin{cases} 
			\frac{1}{2} \Loc^{\ox} \Big( \Cov_{\pi, j+1} [\theta V ; \theta V] + \E_{j+1} \theta W_{j,V}  \Big) & (j +1 < N) \\
			0 & (j+1=N) .
		\end{cases}
		\label{eq:pioxPjV}
\end{align}
\end{definition}

\subsection{Proof of the plateau}

\subsubsection{$u_{N,\ox}$} \label{sec:uNox}

If we recall $\Phi^{\pt}_j (V) =  \mathbb{E}_{j+1} \theta V - P_{j,V}$ from Definition~\ref{def:WP},  we need to expand out \eqref{eq:pioxPjV} in order to trace the flow of $u_{j,\ox}$.  As always,  we are using the convention
\begin{align}
	\pi_{\hash} V_x (\varphi) = \sum_{\km \in \{\ko,\ko_{1,\nabla} \}} \lambda_{\hash}^{(\km)} \one_{x = \hash} S^{(\km)} (\varphi) 
		\qquad \text{when} \quad \hash \in \{\o,\x\}  .
\end{align}

\begin{lemma}
For $V_j \in \cV$ and $j < N$,
\begin{align}
	\sigma_\ox^{-1} \Cov_{\pi,j+1} [ \theta V_y  ; \theta V_z ] &= \sum_{\km_3, \km_4 \in \ko_1 \cup \ko_{1,\nabla}} \Big( \lambda^{(\km_3)}_{\o} \lambda^{(\km_4)}_{\x} \nabla^{\km_3}_{y}  \nabla^{\km_4}_{z} \Gamma_{j+1} (y-z) \one_{y=\o, z=\x} 
	\label{eq:oxCovVV} \\
		& \qquad\qquad\qquad + \lambda^{(\km_3)}_{\x} \lambda^{(\km_4)}_{\o} \nabla^{\km_4}_{x'}  \nabla^{\km_3}_{x''} \Gamma_{j+1} (y-z) \one_{y=\x, z=\o} \Big)  \nonumber
\end{align}
and
\begin{align}
	\sigma_{\ox}^{-1} W_{j} &= 0 .
		\label{eq:oxW}
\end{align}
\end{lemma}
\begin{proof}
By definition,  $\pi_{\ox} \Cov_\pi [A;B] = \sigma_\ox ( \Cov [ A_\o ; B_\x ] + \Cov [ A_\x ; B_\o ] )$. 
Also,  since $\theta V_{\hash, y} (\varphi) = V_{\hash, y} (\varphi) + V_{\hash, y} (\zeta)$ and $V_{\hash, y} (\zeta)$ is linear in $\zeta$,  we have $\E_{j+1} V_{\hash, y} (\zeta) = 0$,
\begin{align}
	\pi_\ox \Cov_{j+1} [ \theta V_{\hash, y} ; \theta V_{\hash',   z}] = \pi_\ox \E_{j+1} [ V_{\hash, y} (\zeta) V_{\hash',  z} (\zeta)  ] 
\end{align}
for $\hash, \hash' \in \{\o,\x\}$.
This is a Gaussian integral of a quadratic function,  so it can be expressed in terms of $\Gamma_{j+1}$,  giving \eqref{eq:oxCovVV}.

For \eqref{eq:oxW},  observe that for $j<N$
\begin{align}
	\pi_\ox W_{j,V,x} &= \frac{1}{2} (1-\Loc^{\ox}) \sigma_{\ox} \F_{\pi,w_{j}} [  V_x ;  V (\Lambda)  ] \nnb
		&= \frac{1}{2} (1-\Loc^{\ox}) \sigma_{\ox} \Big( \F_{w_{j}} [  V_{\o,x} ;  V_{\x} (\Lambda)  ] + \F_{w_{j}} [  V_{\x,x} ;  V_{\o} (\Lambda)  ] \Big) .
\end{align}
Since $V_{\hash,y}$ is a linear function,  we have $e^{-\Delta_{C}} V_{\hash, y} = V_{\hash, y}$ for any covariance $C$,  and thus
\begin{align}
	\F_{C} [ V_{\hash,x} ;  V_{\hash', y}  ] =  \Cov_{C}  [ \theta V_{\hash,x} ;  \theta V_{\hash', y}  ]
	\qquad \hash,  \hash' \in \{\o,\x \}	
\end{align}
is constant-valued,
and it vanishes when we apply $1 - \Loc^\ox$.
Thus $\pi_\ox W_{j,x}$ also vanishes
\end{proof}

Due to the lemma and Definition~\ref{def:WP},  if we let $(\hat{\lambda}_{j,\hash}^{(\km)})_{\km \in \ko_1 \cup \ko_{1,\nabla}}$ to be the coefficients of $\hat{V}_j$,
\begin{align}
	\delta u_{j+1,\ox} (\Lambda) &= - \frac{1}{2} \sum_{\km_3, \km_4 \in \ko_1 \cup \ko_{1,\nabla}} \Big( \hat\lambda^{(\km_3)}_{\o,j} \hat\lambda^{(\km_4)}_{\x,j} \nabla^{\km_3} ( \nabla^{\km_4} )^T  \Gamma_{j+1} (\o - \x)   \;\;\;\; \nnb
		& \qquad\qquad\qquad\qquad +  \hat\lambda^{(\km_3)}_{\x,j} \hat\lambda^{(\km_4)}_{\o,j} \nabla^{\km_4} ( \nabla^{\km_3} )^T  \Gamma_{j+1} (\x - \o)  \Big) ,
		\label{eq:ujoxform}
\end{align}
where we denoted
\begin{align}
	\nabla^{\km_3} ( \nabla^{\km_4} )^T  \Gamma_{j+1} (a - b) 
		= \nabla^{\km_3}_{y}  \nabla^{\km_4}_{z} \Gamma_{j+1} (y-z) \big|_{y = a,  z = b}  .
\end{align}

\begin{lemma}

Assuming \eqref{asmp:obsRG} and $\lambda^{(\emptyset)}_{\hash, \infty}$ as in Corollary~\ref{lemma:deltalambda},  there exists $\tilde{\psi}_N (\x)$ and $\tilde{\psi}_\infty = \lim_{N\rightarrow \infty} \tilde{\psi}_N$ such that
\begin{align}
	u_{N,\ox} (\Lambda) = - \lambda_{\o,\infty}^{(\emptyset)} \lambda_{\x,\infty}^{(\emptyset)} w_{N} (\x) + \tilde{\psi}_N (\x)
	\label{eq:uNmoxest}
\end{align}
and satisfy $\lim_{|\x| \rightarrow \infty} |\x|^{d-2+\eta} \tilde\psi_N (\x) = 0$ uniformly in $N$ and $|\tilde{\psi}_\infty (\x) - \tilde{\psi}_N (\x) | \le O(L^{-(d-2+\eta)N})$ uniformly in $\x$.
\end{lemma}

\begin{proof} 
Since $\Gamma_j (\x - \o) = 0$ whenever $j < j_\ox$, 
\begin{align} \label{eq:uNmoxestexapnd}
	u_{N,\ox} (\Lambda) = - \sum_{j\in [ j_\ox ,  N-1]} \delta u_{j+1,\ox}
\end{align}
with $\delta u_{j,\ox}$ given by \eqref{eq:ujoxform}.

When $d > d_{c,u}$ and $(\km_3) = (\km_4) = (\emptyset)$,  then we can apply Proposition~\ref{prop:theFRD} to bound $\Gamma_j$ and apply Corollary~\ref{cor:lambdainfty} to approximate $\lambda^{(\emptyset)}_{\hash, j}$ by $\lambda^{(\emptyset)}_{\hash, \infty}$ and obtain
\begin{align}
	\Big| \sum_{j\ge j_\ox} \big( \hat{\lambda}^{(\emptyset)}_{\hash, j} - {\lambda}^{(\emptyset)}_{\hash, \infty}  \big) \Gamma_j (\x) \Big| 
		& \le O_L (1) \sum_{j \ge j_\ox} \tilde{\chi}_j \tilde{g}_j L^{-\frac{1}{2} (d-4+2\eta) j} L^{-(d-2+\eta) j } \nnb
		& \le O_L (1)  \tilde{\chi}_{j_\ox} \tilde{g}_{j_\ox} |x|^{-\frac{1}{2} (3d - 8 + 4 \eta)}	
\end{align}
where the final inequality follows from that $L^{-j_\ox} \le O_L (|\x |^{-1})$.
When $d = d_{c,u}$,  then Corollary~\ref{cor:lambdainfty} similarly gives
\begin{align}
	\Big| \sum_{j \in [j_\ox ,  N-1]} \big( \hat{\lambda}_{\hash, j}^{(\emptyset)} - \lambda_{\hash, \infty}^{(\emptyset)}  \big) \Gamma_j (\x) \Big| 
		& \le O_L \big( \sum_{j \ge j_\ox} \tilde{g}_j  L^{-2 j } \big)
		\le O_L \big( ( |x|^{2} \log |x|  )^{-1} \big) ,
\end{align}
where the final inequality follows from Lemma~\ref{lemma:gjasympcritd} that $\tilde{g}_{j_\ox} \lesssim j_\ox^{-1} \le O_L ( 1/ \log |\x| )$.

When $d > d_{c,u}$,  we also need to consider terms $\km_3 \in \ko_{1,\nabla}$ or $\km_4 \in \ko_{1,\nabla}$.  Observe that Corollary~\ref{cor:lambdainfty} also implies
\begin{align}
	|\lambda_{\hash,j}^{(\km')} | \le O_L (1) \times \begin{cases}
			1 & (q(\km_3) <  \frac{1}{2} (d-4+5\eta) ) \\
			L^{(1-2\eta) j} & (\textnormal{otherwise}) 
		\end{cases}
\end{align}
for any $\km' \in \ko_{1,\nabla}$,  thus if we again use Proposition~\ref{prop:theFRD}
\begin{align}
	& \Big| \sum_{j\in [j_\ox ,  N-1]}  \Big( \nabla^{\km_3} (\nabla^{\km_4})^T \Gamma_j (\o-\x) +  \nabla^{\km_3} (\nabla^{\km_4})^T \Gamma_j (\x-\o) \Big)  \hat{\lambda}^{(\km_3)}_{\hash, j} \hat{\lambda}^{(\km_4)}_{\hash, j} \Big|  \nnb
		& \qquad
		\le O_L (1)  \sum_{j \ge j_\ox} L^{- \big( d-2+\eta + q(\km_3) + q(\km_4) \big) j } \big( L^{q(\km_3) +q(\km_4)} \big)^{(1-c) j} \label{eq:nnablaGammajllambdabnd}
\end{align}
for some $c>0$,  and it is bounded by $|x|^{-(d-2+\eta + c) j}$.
Thus we have completed bounding
\begin{align}
	\tilde\psi_{N} (\x) &:= \sum_{j \in [ j_\ox, N-1]} ( \hat{\lambda}^{(\emptyset)}_{\hash, j} -  \lambda^{(\emptyset)}_{\hash, \infty} ) \Gamma_j (\x) \nnb
	&\qquad \qquad \qquad  
		+  \Big( \nabla^{\km_3} (\nabla^{\km_4})^T \Gamma_j (\o-\x) 
		  +  \nabla^{\km_3} (\nabla^{\km_4})^T \Gamma_j (\x-\o) \Big)  \hat{\lambda}^{(\km_3)}_{\hash, j} \hat{\lambda}^{(\km_4)}_{\hash, j}
\end{align}
and the bound on $\tilde{\psi}_\infty := \lim_{N\rightarrow \infty} \tilde{\psi}_N$ is obtained similarly.
\end{proof}

\subsubsection{$B_N$}  \label{sec:INox}

We estimate $B_N$ of Corollary~\ref{cor:crlfncreform}.
By direct expansion,
\begin{align}
	I_{N,\ox} &= I_{N,\bulk} V_{N,\o} V_{N,\x} +  e^{-V_{N,\bulk}^\stable} \big(W_{N, V,\ox}  - V_{N,\o} W_{N,V,\x} - V_{N,\x} W_{N,V,\o} \big) \label{eq:BNox} .
\end{align}

\begin{lemma} \label{lemma:IVNox}
Assuming \eqref{asmp:obsRG},  as $N\rightarrow \infty$,
\begin{align}
	\frac{ \sm \big[ ( I_{N,\bulk} V_{N,\o} V_{N,\x} ) (\Lambda ,  z h'_N \one) \big]}{e^{u_{N,\bulk}|\Lambda| } \sm [ Z_{N,\bulk} (z h'_N \one)  ]} 
		&\sim \lambda^{(\emptyset)}_{\infty,\o} \lambda^{(\emptyset)}_{\infty,\x} \frac{(h'_N)^2}{n} \frac{ \sm [ |z|^2 e^{-\frac{1}{4} |z|^4  } ]}{\sm [ e^{-\frac{1}{4} |z|^4} ] }   .
\end{align}
\end{lemma}
\begin{proof}
We have $I_{N,\bulk} (\Lambda,  h'_N z\one) \rightarrow e^{-\frac{1}{4} |z|^4}$ by Lemma~\ref{lemma:Vtozfour} and we can absorb the integral of $K_{N,\bulk}$ into the error term using Lemma~\ref{lemma:KIintegralRatio} stated below,
so by the Dominated convergence theorem,
\begin{align}
	\sm [ (I_{N,\bulk} + K_{N,\bulk} ) (\Lambda,  h'_N z\one) ] \sim \sm [ e^{-\frac{1}{4} |z|^4 } ] .
		\label{eq:IKhNzoneint}
\end{align}
For the numerator,  since the gradient of $h'_N z \one$ vanishes, we have $\pi_{\km} V_{N,\hash} (\Lambda,  h'_N z\one) = 0$ for $\km  \in \ko_{1,\nabla}$,  so together with the symmetry explained after \eqref{eq:ko1nabladefi},
\begin{align}
	V_{N,\hash} (\Lambda,  h'_N z\one) = \lambda^{(\emptyset)}_{\hash ,  N} h'_N z^{(1)} ,
	\label{eq:VNhashzone}
\end{align}
for both $\hash \in \{\o,\x\}$.
By the Dominated convergence theorem and Corollary~\ref{cor:lambdainfty},
\begin{align}
	\sm [ ( I_{N,\bulk} V_{N,\o} V_{N,\x} \big) \big(\Lambda,  h'_N z \one \big) ]
		&\sim  (h'_N)^2 \lambda^{(\emptyset)}_{N,\o} \lambda^{(\emptyset)}_{N,\x}  \sm \big[ |z^{(1)}|^2 e^{-\frac{1}{4} |z|^4 } \big]  \nnb
		&\sim \frac{1}{n} (h'_N)^2 \lambda^{(\emptyset)}_{\infty,\o} \lambda^{(\emptyset)}_{\infty,\x}  \sm \big[ |z|^2 e^{-\frac{1}{4} |z|^4 } \big]  
\end{align}
as $N\rightarrow \infty$. 
\end{proof}

We need the following estimate to complete the lemma.

\begin{lemma} \label{lemma:KIintegralRatio}
Assuming \eqref{asmp:obsRG},  for sufficiently large $N$,
\begin{align}
	\left| \frac{ \sm [ K_{N,\bulk} (\Lambda,   h'_N z \one)] }{\sm [ I_{N,\bulk} (\Lambda,  h'_N z \one) ]} \right|
		\le O_L (1) \tilde{g}_N^{3/4} \scale_N^{\kpe}  .
\end{align}
\end{lemma}
\begin{proof}
By the argument of \eqref{eq:KNLambdavN}--\eqref{eq:KNLambdavN3},  since $\tilde{G} (\Lambda,  h'_N z\one) = 1$, 
\begin{align}
	\big|K_{N,\bulk} \big(\Lambda ,  h'_N z \one \big)\big| 
		\le C_{\rg} \tilde{g}_N^{\frac{3}{4}} \scale_N^{\kpe} e^{-c |z|^2} 
\end{align}
and by Lemma~\ref{lemma:Vtozfour},  we have $I_{N,\bulk} (\Lambda,  h'_N z \one) \rightarrow e^{-\frac{1}{4} |z|^4}$.
\end{proof}

Other terms of \eqref{eq:BNox} can be shown to be sub-dominant.

\begin{lemma} \label{lemma:EVWox}
Assuming \eqref{asmp:obsRG} and with sufficiently small $\epsilon' >0$,
\begin{align}
	& \lim_{N \rightarrow \infty} (h'_N)^{-2} \frac{ \sm \big[ \big(  e^{-V_{N,\bulk}^\stable} W_{N, V,\ox} \big) (\Lambda, z h'_N \one ) \big] }{e^{u_{N,\bulk}|\Lambda| } \sm [ Z_{N,\bulk}  (z h'_N \one)  ]} = 0 
	\label{eq:EVWox0}	 \\
	& \lim_{N\rightarrow \infty} (h'_N)^{-2} 
		\frac{\sm \big[ \big( e^{-V^\stable_{N,\bulk}} V_{N,\o} W_{N,V,\x} \big) (\Lambda,  zh'_N \one ) \big] }{e^{u_{N,\bulk}|\Lambda| } \sm [ Z_{N,\bulk}  (z h'_N \one)  ]} = 0 ,
	\label{eq:EVWox1}
\end{align}
and the same holds for $V_{N,\x} W_{N,V,\o}$.
\end{lemma}

\begin{proof}
Using Lemma~\ref{lemma:FWPbounds3obs} to bound $W_{N,V}$,
\begin{align}
	\norm{W_{N,V} (\Lambda) }_{h_N, T_N (0)} \le O_L (1) (\tilde{g}_N \scale_N)^{1/2} ,
\end{align}
and since $\norm{h'_N z \one}_{h_N,\Phi_{N}} \lesssim |z|$ and $W_N$ is a polynomial of degree $\le 6$,
\begin{align}
	h_{N,\ssigma} \big| W_{N,V,\ox}  (\Lambda,  h'_N z \one) \big| 
		& \le O_L (1)  (\tilde{g}_N \scale_N)^{1/2} (1+|z|)^6 ,\\
	h_{N,\sigma} \big| W_{N,V,\x}  (\Lambda,  h'_N z \one) \big| 
		& \le O_L (1)  (\tilde{g}_N \scale_N)^{1/2} (1+|z|)^6 . 
\end{align}
but since $h_{N,\sigma}^{-1} = h_{N,\bulk} \lesssim h'_N$ and $h_{N,\ssigma} = \tilde{g}_N^{1/2} L^{\frac{d}{2} j_\ox} (L^{ \frac{d}{2} (1 - \epsilon') - (d-4+2\eta) \kpe } )^{N - j_\ox}$,
\begin{align}
	\big| W_{N,V,\ox}  (\Lambda,  h'_N z \one) \big| 
		& \le O_L (\tilde{g}_N \scale_N)^{1/2} (h'_N)^2 (1+|z|)^6  , \\
	\big| W_{N,V,\x}  (\Lambda,  h'_N z \one) \big| 
		& \le O_L (\tilde{g}_N \scale_{j_\ox})^{1/2} ( L^{\frac{d}{2}\epsilon' + (d-4+2\eta) (\kpe-\frac{1}{2}) }  )^{(N-j_{\ox})_+} (h'_N)^2 (1+|z|)^6  ,
\end{align}
and together with \eqref{eq:VNhashzone},  and the condition $\kpe <1/2$ (see after \eqref{eq:kaekbekpedefi}) with a small choise of $\epsilon'$,
\begin{align}
	& \max \Big\{ \big|  W_{N,V,\ox} (\Lambda,  h'_N z \one) \big|  ,  \;
		\big| \big( V_{N,\o} W_{N,V,\x} \big) (\Lambda,  h'_N z \one) \big|  \Big\} 
		 \le   o_L (1) (h'_{N})^2 (1+|z| )^7 ,
\end{align}
We can plug this bound into \eqref{eq:EVWox1} and using Lemma~\ref{lemma:Vtozfour} to see that $e^{-V^\stable_{N,\bulk}} (\Lambda,  h'_N z \one )  \rightarrow e^{-\frac{1}{4} |z|^4}$
\begin{align}
	\max\Big\{	\Big| \sm \big( e^{-V^\stable_{N,\bulk}} W_{N,V,\ox} \big) (\Lambda,  h'_N z \one) \Big| ,  \;\; 
		\Big| \sm \big( e^{-V^\stable_{N,\bulk}} V_{N,\o} W_{N,V,\x} \big) (\Lambda,  h'_N z \one)  \Big| \Big\} \nnb
			\qquad 
			\le o_L (1) (h'_N)^2 \sm [ (1+ |z|)^7  e^{-\frac{1}{4} |z|^4} ]z .
\end{align}
On the other hand,  the denominators of \eqref{eq:EVWox0} and \eqref{eq:EVWox1} converges to $\sm [  e^{-\frac{1}{4} |z|^4} ]$ by \eqref{eq:IKhNzoneint},  giving \eqref{eq:EVWox0} and \eqref{eq:EVWox1}.
\end{proof}

\subsubsection{$A_{N}$} \label{sec:ANox}

Since $A_{N}$ comes from the remainder term $K_N$,  so we only need crude bounds on it. 

\begin{lemma} \label{lemma:ANox}
Assuming \eqref{asmp:obsRG},
\begin{align}
	& | A_{N,\ox} | \le c_N \big(  (h'_N)^2 + (1\vee  |\x| )^{-(d-2+\eta)} \big)
\end{align}
for some sequence $\lim_{N\rightarrow \infty} c_N =0$.
\end{lemma}

\begin{proof}
Recall the definition $A_N =  \sm [ K_N (\Lambda,  z h'_N \one)] / (e^{u_{N,\bulk} |\Lambda|} \sm [Z_{N,\bulk} (\Lambda,  z h'_N \one) ] )$.
As in the proof of Lemma~\ref{lemma:IVNox},  the denominator converges to $\sm [ e^{-\frac{1}{4} |z|^4} ]$.

For the numerator,  by the argument of \eqref{eq:KNLambdavN}--\eqref{eq:KNLambdavN3},  since $\tilde{G}_N (\Lambda,  h'_N z\one) \le 1$ by Corollary~\ref{cor:logtildeGNbndcor}, 
\begin{align}
	h_{N,\ssigma} | K_{N,\ox} (\Lambda ,  h'_N z\one )| 
		\le C_{\rg} \tilde{g}_N^{\frac{3}{4}} \scale_N^{\kpe} e^{-c |z|^2}  ,
\end{align}
but since $h_{N,\ssigma} = \tilde{g}_N^{1/2} L^{\frac{d}{2} j_\ox} (L^{ \frac{d}{2} (1 - \epsilon') - (d-4+2\eta) \kpe } )^{N - j_\ox}$,  we obtain
\begin{align}
	\big| \sm [ K_{N} (\Lambda,  h'_N z \one ) ]  \big|
		& \le O_L (1)
		h_{N,\ssigma}^{-1} \tilde{g}_N^{\frac{3}{4}} L^{- (d-4+2\eta) \kpe N}  L^{ - \frac{d}{2} j_\ox } \int_{\R^n} e^{-c |z|^2} \rd z \nnb
		& = O_L (1)  
		\tilde{g}_N^{1/4}
			\big( L^{-\frac{d}{2} \epsilon' - (d-4+2\eta) \kpe} \big)^{j_\ox} L^{-\frac{d}{2} (1- 2\epsilon') N} L^{-\frac{d}{2} \epsilon' N} .
\end{align}
Let us denote $E (j_\ox, N) = \big( L^{-\frac{d}{2} \epsilon' - (d-4+2\eta) \kpe} \big)^{j_\ox} L^{-\frac{d}{2} (1-2\epsilon') N}$.
When $j_\ox \le \frac{d}{2(d-2+\eta)} N$,  then for some $\epsilon''$ proportional to $\epsilon'$,
\begin{align}
	E (j_\ox, N) \le \big( L^{-(d-2+\eta) - (d-4 + 2\eta) (\kpe - \epsilon'')} \big)^{j_\ox}
		\le  (1 \vee |\x | )^{-(d-2+\eta) - (d-4+2\eta) (\kpe - \epsilon'')} 
\end{align}
where we used that $L^{-j_\ox} \le (1 \vee |\x | )^{-1}$.
When $j_\ox \ge \frac{d}{2(d-2+\eta)} N$,  then for large $N$ and sufficiently small $\epsilon'$ and some $\epsilon'''$ proportional to $\epsilon'$,
\begin{align}
	E (j_\ox, N) \le \big( L^{-\frac{d}{2} ( 1 + \frac{d-4+2\eta}{d-2+\eta} \kpe ) - \epsilon''' } \big)^N
		\le O \big( \tilde{g}_N^{-1/2} L^{-dN/2} \big)
	 .
\end{align}
Thus for any $j_\ox$ and $N$,
\begin{align}
	E (j_\ox, N) \le O( L^{-\frac{d}{2}\epsilon' N} ) \big( \tilde{g}_N^{-1/2} L^{-dN/2} + |\x|^{-(d-2+\eta)} \big) .
\end{align}
\end{proof}

We have all the building blocks to conclude this section.

\begin{proof}[Proof of Proposition~\ref{prop:plateauInt}]
We choose $\lambda_{0,\o}^{(\emptyset)} = \lambda_{0,\x}^{(\emptyset)} =1$.
By Lemma~\ref{lemma:Vtilde} and plugging in the estimates of Section~\ref{sec:uNox}--\ref{sec:ANox} into \eqref{eq:corcrlfncreform1},
\begin{align}
	\langle \varphi_\o^{(1)} \varphi_\x^{(1)} \rangle_{g,\nu, \Lambda} 
		&= 
		\lambda_{\o,\infty}^{(\emptyset)} \lambda_{\x,\infty}^{(\emptyset)} \big( w_N (\x) 
		+(h'_N)^2 \E [ |Y|^2] /n  \big)
		+ \psi_1 (\x) + \psi_2 (\x , N) 
		\label{eq:propplateauIntint}
\end{align}
for $\psi_1$ and $\psi_2$ with the desired properties.
If we sum up \eqref{eq:propplateauIntint} over $\x \in \Lambda$,  then by translation invariance and \eqref{eq:propplateauIntint},
\begin{align}
	\big\langle \big( \varphi_\x^{(1)} \big)^2 \big\rangle_{g,\nu, \Lambda}  = 
		= \lambda^{(\emptyset)}_{0,\infty} \sum_\x \lambda^{(\emptyset)}_{\x,\infty} (h'_N)^2 \E[|Y|^2] = (g_N^{(\emptyset)})^{-1/2} \lambda^{(\emptyset)}_{\o,\x}	\frac{\sum_\x \lambda^{(\emptyset)}_{\x,\infty}}{L^{dN}} L^{\frac{d}{2} N} \E[|Y|^2] .
\end{align}
as $N\rightarrow \infty$---contribution of $(h'_N)^2$ dominates over that of $w_N$ due to Corollary~\ref{cor:Coneinftynrm}.
Since $\chi_{g,\nu, \Lambda}$ should not have dependence on the specific choice $\o$,  we see that the right-hand side is also translation invariant,  thus $\lambda = \lambda^{(\emptyset)}_{\o,\infty} = \lambda^{(\emptyset)}_{\x,\infty}$ for some $\bar{\lambda}$.  
Also,  comparing this with \eqref{eq:corNGlimit2} (notice that Corollary~\ref{cor:NGlimit} follows from Theorem~\ref{thm:NGlimit},  which does \emph{not} depend on this section), 
using \eqref{eq:gNasymp} to approximate $g_N^{(\emptyset)}$ and recalling $c_3$ from \eqref{eq:c3defi},  we see that $\bar{\lambda}$ should actually be 1.
\end{proof}

\begin{remark}
Another proof of the shocking fact that $\lambda_{\x,\infty} =1$ in a similar context is given in \cite[Lemma~4.6]{MR3345374},  which makes reference to a global symmetry and requires a bound on $D_\varphi Z_N$.  This may be repeated here,   but in the proof above,  we get it almost for free by comparing it with the FSS of the susceptibility.  
\end{remark}

\section{Covariance decomposition}
\label{sec:covdcmp}

In this final section,  we prove Proposition~\ref{prop:theFRD}.  
The idea is based on the infinitesimal decomposition of the covariance matrix,  inspired by \cite{MR3129804}.  This allows to express the scale-decomposed covariances $\Gamma_j$'s as integrals.

We first discuss the covariance decomposition on $\Z^d$,  i.e.,  $(\dot\Gamma_t)_{t \ge 0}$ that satisfy
\begin{align}
	( \L_\eta^{(\vec{\ba})} )^{-1} = \int_0^{\infty} \dot{\Gamma}_t (\vec{\ba}) \rd t  
	, \qquad \vec{\ba} \in \HB_{\epsilon_p}	
	 \label{eq:covdecmpZd}
\end{align}
Proposition~\ref{prop:Gammajbounds} and Lemma~\ref{lemma:FRDcontinuity} are simply restatements of Proposition~\ref{prop:theFRD},  when we define $\Gamma_j$ using \eqref{eq:Gammajasintegrals}.

\begin{proposition} \label{prop:Gammajbounds}
Let $d \ge 3$,  $\eta \in [0,2)$ and $\vec{\ba} \in \HB_{\epsilon_p}$ for sufficiently small $\epsilon_p$. 
Then there exist covariance matrices $(\dot\Gamma_t : \Z^d \times \Z^d \rightarrow \R)_{t\ge 0}$ such that \eqref{eq:covdecmpZd} and the following hold.

\begin{enumerate}
	\item (Symmetries) $\dot\Gamma_t : \Z^d \times \Z^d \rightarrow \R$ is a covariance matrix invariant under isometries,
	i.e., $\dot\Gamma_t \ge 0$ and $\dot\Gamma_t (E (x) , E(y)) = \dot\Gamma_t (x,y)$ for any isometry $E : \Z^d \rightarrow \Z^d$.
	\item (Finite range property) $\dot\Gamma_t$ has range $< t$ in the $\ell^1$-metric,  i.e.,  $\dot\Gamma_t (x,y) =0$ whenever $\norm{x-y}_{\ell^{1}} \ge t$. 
	\item (Upper bound) For each $k, k_x,k_y \ge 0$ with $k_x + k_y = k$,  
	\begin{align}
		\big\| \nabla_x^{k_x} \nabla_y^{k_y} \dot{\Gamma}_t (x,y) \big\|_{\ell^\infty (\Z^d \times \Z^d)} \le C_k \frac{t^{-d+1-\eta-k}\wedge 1}{1 + \ba^{(\emptyset)} t^{2\beta}}	 .  \label{eq:Gammajbounds1infi}
	\end{align}
for some $j,L$-independent constant $C_k$.	
\end{enumerate}
\end{proposition}

By (i),  $\dot\Gamma_t$ is also translation invariant,  so we just denote $\dot\Gamma_t (x-y) \equiv \dot\Gamma_t (x,y)$.

On finite volume torus $\Lambda_N$,  we alternatively require
\begin{align}
	(\L_\eta^{(\vec{\ba})})^{-1} = \int_0^{L^{N-1}} \dot\Gamma_t \rd t +  \int_{L^{N-1}}^\infty \dot\Gamma_t^{\Lambda_N} \rd t + t_N  Q_N
	\label{eq:covdecmpTd}
\end{align}
where $\dot\Gamma_N^{\Lambda_N}$ is now $\Lambda_N$-dependent and $Q_N (x,y) = L^{-dN}$.

\begin{proposition}  \label{prop:Gammajbounds2}
Under the assumptions of Proposition~\ref{prop:Gammajbounds},  
if $\ba^{(\emptyset)} > 0$ in addition,
then there exist $t_N > 0$ and covariance matrices $\dot\Gamma_t$'s and $\dot\Gamma_t^{\Lambda_N}$ on $\Lambda_N$ satisfying \eqref{eq:covdecmpTd} and the following.
\begin{enumerate}
	\item $\dot\Gamma_t$'s are the projections of those of Proposition~\ref{prop:Gammajbounds} on $\Lambda_N$.  (Note that this only makes sense due to the finite range property.)
	\item $\dot\Gamma_t^{\Lambda}$ satisfies the same symmetries and the upper bounds on $\dot\Gamma_t$.  

	\item $t_N \in (0, (\ba^{(\emptyset)})^{-1})$ and there exists $C >0$ such that $t_N >  (\ba^{(\emptyset)})^{-1} - CL^{(2-\eta) N}$.
\end{enumerate}
\end{proposition}

Given these covariance matrices,  we define
\begin{align}
	\begin{cases}
	\Gamma_1 = \int_{0}^L \dot{\Gamma}_t \rd t , \;\; 
	\Gamma_{j} = \int_{L^{j-1}}^{L^j} \dot{\Gamma}_t \rd t \quad (2 \le j \le N-1) \\
	\Gamma_{N}^{\Lambda_N} = \int_{L^{N-1}}^{\infty} \dot{\Gamma}_t^{\Lambda_N}  .
	\end{cases}
	\label{eq:Gammajasintegrals}
\end{align}
To finish the restatement of Proposition~\ref{prop:theFRD},  we also need the continuity.

\begin{lemma} \label{lemma:FRDcontinuity}
Under the assumptions of Proposition~\ref{prop:Gammajbounds} $\Gamma_j$,  $\Gamma_{N}^{\Lambda_N}$ and $t_N$ are continuous in $\vec{\ba}$.
\end{lemma}

\subsection{Scale decomposition lemma}

The next lemma forms the basis of the theory,  which allows to express inverse of an operator as an integral. 

\begin{lemma} \cite[Lemma~3.8]{dgauss1}
\label{lemma:Ptdef}
For $t >0$, there exists polynomial $P_t$ of degree at most $t$ such that for $x \in (0,3]$,
\begin{align}
	\frac{1}{x} = \int_0^\infty t^2 P_t (\lambda) \frac{\rd t}{t}
	.
\end{align}
For $t < 1$, $P_t (x) = c / t$ for some constant $c>0$ and for $t \ge 1$,  the polynomials satisfy
\begin{align}
	0\le P_t (x) & \le C e^{-c(x t^2)^{1/4}} .  \label{eq:Ptbound} 
\end{align}
\end{lemma}

\subsection{Infinite lattice,  $\eta = 0$}

For any translation invariant operator,  we can substitute the Fourier symbol of the operator in Lemma~\ref{lemma:Ptdef} to invert it. 
To be specific,  we define
\begin{align}
	\dot{\Gamma}_t = \frac{t}{16d^3} P_{t / 2d} ( \L_0^{(\vec{\ba})} / 4d )
	\qquad \Rightarrow C^{(\vec{\ba})} = \int_0^\infty \dot{\Gamma}_t (t) dt
	\label{eq:Gammajdefn}
\end{align}
Some properties are immediate from the definition.

\begin{proof}[Proof of Proposition~\ref{prop:Gammajbounds}(i),(ii) for $\eta =0$]
Symmetries hold trivially by definition.  
For the finite range property,  since $P_t$ has degree at most $t$ and $\L_0^{(\vec{\ba})}$ involves at most $2d-6$ derivatives,  
for any lattice function $h : \Z^d  \rightarrow \R$ with compact support,  
$P_t (\L_0^{(\vec{\ba})} /4d ) h$ vanishes outside $\{ y : d_1 (y ,  \operatorname{supp} (f)) \le (2d-6) t \}$,  thus $t P_{t/2d} (\L_0^{(\vec{\ba})} /4d )$ has range at  most $t(2d-6) / 2d < t$. 
\end{proof}

\subsection{Torus,  $\eta = 0$}

In a finite volume torus $\Lambda = \Lambda_N$,  we define $\dot\Gamma_t$ for $t < L^{N-1}$ using the same definition \eqref{eq:Gammajdefn}.
To define $\dot\Gamma_t^{\Lambda_N}$ and $t_N$,  we work in the Fourier space---see Appendix~\ref{sec:fa} for the conventions.
In particular,  $\lambda$ and $\lambda^{(\vec{\ba})}$ are the Fourier symbols of $-\Delta$ and $\L^{(\vec{\ba})}_\eta$,  respectively.

Now,  we can take
\begin{align}
	\dot\Gamma_t^{\Lambda} (x) 
		&= \frac{1}{|\Lambda|} \sum_{p \in \Lambda^* \backslash \{ 0 \}} e^{i p \cdot x} \frac{t}{16d^3} P_{t/2d} ( \lambda^{(\vec{\ba})}   (p) / 4d ) ,
	\label{eq:GammaNLambdadefi} .  \\
	t_N 
		&=  \int_{L^{N-1}}^{\infty} \frac{t}{16d^3} P_{t/2d} ( \lambda^{(\vec{\ba})}   (p = 0)  / 4d ) \rd t  \label{eq:tNdefi} . 
\end{align}

\begin{proof}[Proof of Proposition~\ref{prop:Gammajbounds2}(i) for $\eta =0$]
By the Fourier inversion formula,
\begin{align}
	\int_{L^{N-1}}^\infty \dot\Gamma_t^{\Lambda} (x) \rd t + t_N  Q_N (x) 
		&= \frac{1}{|\Lambda|} \sum_{p \in \Lambda^*} e^{i p \cdot (y-x)} \int_{L^{N-1}}^{\infty} \frac{t}{16d^3} P_{t/2d} \Big( \frac{ \lambda^{(\vec{\ba})}   (p)  }{4d} \Big) \rd t  \nnb
	& = \int_{L^{N-1}}^{\infty} \frac{t}{16d^3} P_t \Big( \frac{ \L^{(\vec{\ba})}_0 }{4d} \Big) (x) \rd t ,
\end{align}
so we have
\begin{align}
	\int_0^{L^{N-1}} \dot{\Gamma}_t \rd t + \int_{L^{N-1}}^\infty \dot{\Gamma}_t^{\Lambda} \rd t  + t_N  Q_N 
		=  \int_{0}^{\infty}\frac{t}{16d^3} P_t ( \L^{(\vec{\ba})}_0 ) (x) \rd t 
		= ( \L^{(\vec{\ba})}_0)^{-1} . 
\end{align}
Thus \eqref{eq:covdecmpTd} is satisfied.
\end{proof}

\subsection{Decay estimates,   $\eta = 0$}

\begin{proof}[Proof of Proposition~\ref{prop:Gammajbounds}(iii) and Proposition~\ref{prop:Gammajbounds2}(ii),(iii) for $\eta =0$]

By \eqref{eq:Ptbound} and Lemma~\ref{lemma:Lapptv-res},
\begin{align}
	|\nabla^k \dot{\Gamma}_t (x)|
		\lesssim \int_{(\Z^d)^*} |p|^k t P_{t/2d} \big(\frac{ \lambda^{(\vec{\ba})} }{4d}) \rd p  
		&\lesssim \int_{\R^d}  t |p|^k e^{- c ( |p|^2 + \ba^{(\emptyset)} )^{1/4} t^{1/2} } \rd p \nnb 
		&\lesssim t^{-d+1-k} e^{-c' ( \ba^{(\emptyset)} t^2 )^{1/4} } 
\end{align}
when $t \ge 2d$,  while for $t < 2d$,  $|\nabla^k \dot{\Gamma}_t (x) |\lesssim 1$.
The same holds for $\dot\Gamma_t^{\Lambda}$,  only with a discrete sum replacing the integral.

To bound $t_N$,  if we only consider the 0-mode ($p=0$) in the Fourier space,
\begin{align}
	( \ba^{(\emptyset)} )^{-1} = C^{(\vec{\ba})} (0) = \int_{0}^{L^{N-1}} \hat{\dot{\Gamma}}_t (0) \rd t + t_N ,
	\label{eq:tNest1}
\end{align}
but for $t \ge 0$,
\begin{align}
	0\le \hat{\dot{\Gamma}}_t (0) = \frac{t}{16d^3} P_{t / 2d} ( \ba^{(\emptyset)} / 4d ) \lesssim t ,
	\label{eq:tNest2}	
\end{align}
thus
\begin{align}
	0 \le \int_{0}^{L^{N-1}} \hat{\dot\Gamma}_t (0) \rd t \lesssim L^{2 (N-1)} ,
	\label{eq:tNest3}	
\end{align}
giving $t_N - ( \ba^{(\emptyset)} )^{-1} \in ( - C L^{2N} ,  0)$ for some $L$-independent constant $C$.
Also,  $t_N > 0$ is obvious since $P_t \ge 0$.
\end{proof}

\subsection{Infinite lattice,   $\eta > 0$}

The proof for the long-range interaction uses a spectral decomposition of the fractional power originally introduced in \cite{mitter2016finite} and a series expansion.
We will use $\beta = 1- \eta/2$ for notational simplicity.

\begin{proposition} \cite[Lemma~2.2]{MR3772040} \label{prop:spectralfracpw}
Let $\beta \in (0,1)$,  $w,z\ge 0$ and $w^\beta + z >0$.  Then
\begin{align}
	\frac{1}{w^{\beta} + z} = \frac{\sin (\pi \beta )}{\pi}  \int_{0}^\infty \frac{1}{ s^\beta ( w+s ) \sigma_\beta (s,z)} \rd s 
	\label{eq:spectralfracpw}
\end{align}
where $\sigma_{\beta} (s, z) =  1 + s^{-2\beta} z^2 + 2 z s^{-\beta} \cos (\pi \beta) > 0$ whenever $s > 0$.
\end{proposition}

Let $\delta \L =  \L^{(\vec{\lambda})}_\eta - (- \Delta)^{1-\eta/2} - \ba^{(\emptyset)}$ and $\delta \lambda$ be its Fourier transform,  so 
$\lambda^{(\vec{\ba})} = \lambda^{\beta} + \ba^{(\emptyset)} + \delta \lambda$.
We want to decompose  $1/ \lambda^{(\vec{\ba})}$ using Proposition~\ref{prop:spectralfracpw},  but unfortunately,  $\ba^{(\emptyset)} + \delta \lambda$ may be negative.  Alternatively,  we take
\begin{align}
	\tilde{\lambda} = \lambda^{\beta} + \ba^{(\emptyset)} + C_{\rm F} \epsilon_p \lambda  \qquad \Rightarrow \qquad \lambda^{(\vec{\ba})} = \tilde{\lambda} - (C_{\rm F} \epsilon_p \lambda - \delta \lambda )
\end{align}
with sufficiently large $C_{\rm F}$ so that $C_{\rm F} \epsilon_p \lambda - \delta \lambda \ge 0$---such choice is possible due to Lemma~\ref{lemma:Lapptv-res}.  Then we obtain a Neumann series expansion
\begin{align}
	\frac{1}{\lambda^{(\vec{\ba})}} = \sum_{n=0}^{\infty} \big( C_{\rm F} \epsilon_p \lambda - \delta \lambda \big)^n \tilde{\lambda}^{-(n+1)} ,
	\label{eq:lambdaNeumann}
\end{align}
and we will apply Proposition~\ref{prop:spectralfracpw} to decompose $\tilde{\lambda}$ instead.

\begin{lemma} \label{lemma:tildelambdadecomp}
Let $\beta \in (0, 1)$.  
For sufficiently small $\epsilon_p$ so that $\tilde{\lambda} \ge 0$, 
there exist covariance matrices $(H_t)_{t\ge 0}$ invariant under lattice symmetries with range $< t$ and Fourier transform $\hat{H}_t$ that satisfy
\begin{align}
	\frac{1}{\tilde{\lambda}} = \int_0^\infty \hat{H}_t \rd t , \qquad
	\hat{H}_t (p) \le C \frac{t^{2\beta-1}}{1 + \ba^{(\emptyset)}t^{2\beta} }  e^{-c( \lambda t^2 )^{1/4}}
\end{align}
for some $C,c>0$.
\end{lemma}

We prove this lemma in Section~\ref{sec:tildelambdadecompproof}.  Subsequently,  we can use expansion \eqref{eq:lambdaNeumann} to obtain the decomposition of $1 / \lambda^{(\vec{\ba})}$.

\begin{proof}[Proof of Proposition~\ref{prop:Gammajbounds}(i),(ii) for $\eta >0$]
We let
\begin{align}
	\frac{1}{\lambda^{(\vec{\ba})}} = \int_0^\infty \hat{\dot{\Gamma}}_t \rd t , \qquad
	\hat{\dot{\Gamma}}_t = \sum_{n =0}^{\infty} \big( C_{\rm F} \epsilon_p \lambda - \delta \lambda \big)^n \hat{H}^{\star (n+1)}_{t-2dn } 
\end{align}
with the convention $\hat{H}^{\star n}_{s} =0$ when $s \le 0$ and 
\begin{align}
	\hat{H}^{\star n}_s = \int_{t_1 + \cdots + t_n = s,  \; t_i \ge 0} \prod_{i=1}^n \hat{H}_{t_i} \rd t_i ,\qquad s > 0 .
	\label{eq:hatHstarndefi}
\end{align}
The finite range property holds because $(C_{\rm F} \epsilon_p \lambda - \delta \lambda )^n$ has range $< 2dn$ and $H^{\star n}_t$ has range $<t$.  The symmetries follow from those of $H_t$.
\end{proof}

We can also obtain a bound on $\hat{\dot{\Gamma}}_t$,  proved in Section~\ref{sec:hatdotGammatbndproof}.

\begin{lemma} \label{lemma:hatdotGammatbnd}
Let $\beta \in (0,  1)$.  For sufficiently small $\epsilon_p$ and some $c>0$,
\begin{align}
	0 \le \hat{\dot{\Gamma}}_t (p) \le O\Big( \frac{t^{2\beta  - 1} }{ 1 + \ba^{(\emptyset)} t^{2\beta} } \Big) e^{-c ( \lambda (p) t^2 )^{1/4}} .
\end{align}
\end{lemma}

\subsubsection{Application of Proposition~\ref{prop:spectralfracpw}}
\label{sec:tildelambdadecompproof}

We apply \eqref{eq:spectralfracpw} with $w = \lambda$  and $z = \ba^{(\emptyset)} + C_{\rm F} \lambda$ and obtain
\begin{align}
	\frac{1}{\tilde{\lambda}} = \frac{\sin (\pi \beta)}{\pi} \int_0^\infty \frac{s^{-\beta} \rd s}{(s+ \lambda ) \sigma_\beta( s,  z )}  .
\end{align}
We want to decompose the integrand using Lemma~\ref{lemma:Ptdef},  but the denominator is not bounded from above,
so we decompose
\begin{align}
	\frac{1}{\tilde\lambda}
		= \frac{\sin (\pi \beta)}{\pi} \Big( I_0 + I_1 +  \sum_{k=1}^{\infty}  J_k  \Big) ,  \qquad 
	\begin{cases}
		J_k = 
			\int_{2^{-k} z^{1/\beta} }^{2^{-k+1} z^{1/\beta}}  \frac{s^{-\beta}}{(\lambda+s) \sigma_\beta (s,z)} \rd s \\
		I_0 = \int_{z^{1/\beta}}^1 \frac{s^{-\beta}}{(\lambda+s) \sigma_\beta (s,z)} \rd s \\
		I_1 = \int_{1}^\infty \frac{s^{-\beta}}{(\lambda+s) \sigma_\beta (s,z)} \rd s .
	\end{cases}
	\label{eq:IIJkdecomp}
\end{align}
(With the assumption that $\vec{\ba}$ is sufficiently small,  we have $z \le 1$.)
Noting that each $J_k$, $I_0$ and $I_1$ are functions in the momentum space with variable $p \in \Lambda^*$,  we have the following decomposition. 

\begin{lemma} \label{lemma:JIkDecomp}
For any $\beta \in (0, 1)$ and each $k \ge 0$ and $l \in \{0,1\}$,
\begin{align}
	J_k = \int_0^\infty \hat{D}^{[k]}_t \rd t ,  \qquad I_l = \int_0^\infty \hat{E}_t^{[l]} \rd t
	\label{eq:JIkDecomp}
\end{align}
where $\hat{D}^{[k]}_t$ and $\hat{E}^{[k]}_t$ are the Fourier transforms of covariance matrices $D^{[k]}_t$ and $E^{[k]}_t$,  respectively,  with range $\le t$ and invariant under lattice symmetries.
\end{lemma}

In the following proof,  we make a change of variables $t \mapsto t/R$ in Lemma~\ref{lemma:Ptdef} so that
\begin{align}
	\frac{1}{x} = R^{-2} \int_0^{\infty} t P_{t/R} (x) dt ,\qquad x \in (0,3] . \label{eq:Ptdefrescaled}
\end{align}
If $x$ is an operator with range $\le R$,  then $P_{t/R} (x)$ has range $\le t$.  We will take $x$ to be a multiple of $(\lambda +s) \sigma_\beta (s,z)$ in the proof,  so $R = 2d$ is sufficient.

\begin{proof}[Proof of Lemma~\ref{lemma:JIkDecomp}]
For $k \ge 1$,  we restate
\begin{align}
	J_k 
		&= \frac{M}{2^{2 \beta k} } \int_{2^{-k} z^{1/\beta}}^{2^{-k+1} z^{1/\beta}} \frac{s^{-\beta} \rd s}{M (\lambda +s) 2^{-2\beta k} \sigma_\beta (s, z) } ,\\
	I_0
		&= \frac{M}{2^{k+1}} \int_{z^{1/\beta}}^{1} \frac{s^{-\beta} \rd s}{M (\lambda +s) \sigma_\beta (s, z) }  ,\qquad 
	I_1
		= \frac{M}{2^{k+1}} \int_{1}^{\infty} \frac{s^{-\beta-1} \rd s}{M s^{-1} (\lambda +s) \sigma_\beta (s, z) }  
\end{align}
for $M = 1 / 100d$.
Then by the choice of $M$ and the domain of $s$,  the denominators of the integrand of $J_k$ and $I_k$ are bounded from above by $1$,  so we can apply \eqref{eq:Ptdefrescaled},  giving
\begin{align}
	\frac{1}{M (\lambda +s) 2^{-2\beta k} \sigma_\beta (s, z) }
		&= R^{-2} \int_{0}^\infty t P_{t/R} ( M 2^{-2\beta k} (\lambda +s)  \sigma_\beta (s, z) ) \rd t
\end{align}		
for $s \in [2^{-k}z^{1/\beta},  2^{-k+1} z^{1/\beta}]$ and		
\begin{align}
	\frac{1}{M (1 \vee s )^{-1} (\lambda +s) \sigma_\beta (s, z) }
		&= R^{-2} \int_{0}^\infty t P_{t/R} ( M (1 \vee s)^{-1} (\lambda +s) \sigma_\beta (s, z) ) \rd t 
\end{align}
for $s \ge z^{1/\beta}$,  so we obtain \eqref{eq:JIkDecomp} with
\begin{align}
	\hat{D}_t^{[k]} &= \frac{M R^{-2}}{2^{2\beta k}} \int_{2^{-k} z^{1/\beta}}^{2^{-k+1} z^{1/\beta}} \frac{t}{s^\beta}  P_{t/R}  ( M 2^{-2\beta k} (\lambda +s)  \sigma_\beta (s, z) ) \rd s \label{eq:hatDdefi} \\
	\hat{E}_t^{[0]} &= M R^{-2}  \int_{z^{1/\beta}}^{1}  \frac{t}{s^{\beta}}  P_{t/R} ( M (\lambda +s) \sigma_\beta (s, z) ) \rd s \label{eq:hatE0defi}\\
	\hat{E}_t^{[1]} &= M R^{-2}  \int_{1}^{\infty}  \frac{t}{s^{\beta+1}}  P_{t/R^2} ( M s^{-1} (\lambda +s) \sigma_\beta (s, z) ) \rd s \label{eq:hatE1defi}
	.
\end{align}
\end{proof}

We arrive at a decomposition of $(\tilde{\lambda})^{-1}$ by letting
\begin{align}
	\hat{H}_t (p) = \frac{\sin (\pi \beta)}{\pi} \Big( \hat{E}_t^{[0]} + \hat{E}_t^{[1]}  + \sum_{k=1}^\infty \hat{D}_t^{[k]} \Big) (p) . 
	 \label{eq:dotHtdefieta}
\end{align}

To bound $\hat{H}_t$,  we need the following computational lemma.

\begin{lemma} \label{lemma:expintest}
For $\beta \in (0,1)$,  $a, c, t,\mu \ge 0$,  there exists $c' > 0$ such that
\begin{align}
	\int_a^{\infty} s^{-\beta} e^{-c ( (\mu + s) t^2 )^{1/4} } \rd s \le O( t^{2\beta -2} ) e^{- c' ( ( \mu + a) t^2 )^{1/4}} .
\end{align}
\end{lemma}
\begin{proof}
Since $(\mu + s)^{1/4} \asymp ( \mu^{1/4} + s^{1/4} )$,   we actually only have to bound
\begin{align}
	e^{-c' ( (\mu+a) t^2 )^{1/4} } \int_a^{\infty} s^{-\beta} e^{- c' ( st^2 )^{1/4}  } \rd s  
\end{align}
for some $c' >0$.
After a change of variable $s \mapsto s t^{-2}$,  we see that the integral is bounded by a constant multiple of $t^{2\beta -2}$,  so we have the desired bound.
\end{proof}

\begin{proof}[Proof of Lemma~\ref{lemma:tildelambdadecomp}]
Since we defined $\hat{H}_t$ as a sum of $\hat{D}_t^{[k]}$ and $\hat{E}_t^{[l]}$ in \eqref{eq:dotHtdefieta},  
it will be sufficient to prove
\begin{align}
	2^{\beta k} \hat{D}^{[k]}_t (p)  ,  \;\;  \hat{E}^{[l]}_t (p) \le O\Big( \frac{t^{2\beta  - 1} }{ 1 + z t^{2\beta} } \Big) e^{-c ( \lambda  t^2 )^{1/4}} 
\end{align}
where $O(1)$ and $c$ are constants uniform in $k$ and $l$.

To obtain their bounds,  we use \eqref{eq:Ptbound} and \eqref{eq:hatDdefi} to see that,
for some $c', c'' >0$ uniform in $k$,
\begin{align}
	\hat{D}_t^{[k]} 
		&\lesssim \Big( \frac{t}{2^{2\beta k}} \Big) \int_{2^{-k} z^{1/\beta}}^{2^{-k+1} z^{1/\beta}} s^{-\beta} e^{- c \big( ( \lambda + s ) t^2  \big)^{1/4}  } \rd s \nnb
		& \lesssim \Big( \frac{t^{2\beta-1}}{2^{2\beta k}}  \Big) 
			e^{-c'' (  \lambda t^2 )^{1/4}} 
			e^{-c'' (  2^{-k} z^{1/\beta} t^2 )^{1/4}} 
\end{align}
where we applied Lemma~\ref{lemma:expintest} for the final inequality.
To remove $2^{-k}$ from the exponents,
we use that fact that $e^{- |x|^{1/4}} \le O ( |x|^{-\beta})$,  so
\begin{align}
	\hat{D}_t^{[k]} \lesssim \Big( \frac{t^{2\beta-1}}{2^{2\beta k}}  \Big) \frac{1}{(1 + 2^{-k} z^{1/\beta} t^2)^{\beta}} e^{-c'' (\lambda t^2)^{1/2} }
		\lesssim \Big( \frac{t^{2\beta-1}}{2^{\beta k}}  \Big) \frac{1}{1 + z t^{2\beta}} e^{-c'' (\lambda t^2)^{1/2} } 
\end{align}

Bounds on $\hat{E}_+^{[k]}$ are almost the same.  By \eqref{eq:Ptbound},  \eqref{eq:hatE0defi},  \eqref{eq:hatE1defi} and Lemma~\ref{lemma:expintest},
\begin{align}
	\hat{E}_t^{[0]} &\lesssim t \int_{z^{1/\beta}}^1 s^{-\beta} e^{-c ( (\lambda +s) t^2 )^{1/4}} \rd s \lesssim t^{2\beta -1} e^{-c' ((\lambda + z^{1/\beta} )t^2)^{1/4}} \\
	\hat{E}_t^{[1]} &\lesssim t \int_{1}^{\infty} s^{-\beta-1} e^{-c ( ( \lambda /s + 1) t^2 )^{1/4}} \rd s .
\end{align}
Using substitution $u = \lambda t^2 / s$,  we obtain
\begin{align}
	\hat{E}_t^{[1]} \lesssim t^{-2\beta + 1} \lambda^{-\beta} e^{-c' t^{1/2}} \int_0^{\lambda t^2} u^{\beta-1} e^{-c u^{1/4}} \rd u
		\lesssim t e^{-c' t^{1/2}} ,
\end{align}
where in the final inequality,  we used $\int_0^{\lambda t^2} u^{\beta-1} du \lesssim  (\lambda t^2)^{\beta}$.
Since $\lambda$ is bounded from above,  this can be made sufficiently smaller than the desired bound by choosing the right constant in the exponent.
\end{proof}

\subsubsection{Series expansion}
\label{sec:hatdotGammatbndproof}

To bound the convolution \eqref{eq:hatHstarndefi},  we use the following simple fact. 

\begin{lemma} \label{lemma:starconvltnbnd}
If $\beta > 0$,  $c\ge 0$ and $a \ge 0$,  then for $t_1 + t_2 = t$,  
\begin{align}
	\Big(\frac{t_1^{2\beta-1}}{1 + a t_1^{2\beta}} e^{-c(\lambda t_1^2)^{1/4}} \Big)
	\Big(\frac{t_2^{2\beta-1}}{1 + a t_2^{2\beta}} e^{-c(\lambda t_2^2)^{1/4}} \Big) \le 
		\frac{1}{2^{2\beta -2}} \Big(\frac{t^{2\beta-1}}{1 + a t^{2\beta}} e^{-c(\lambda t^2)^{1/4}} \Big) .
\end{align}
\end{lemma}
\begin{proof}
Since $\sqrt{x} + \sqrt{y} \ge \sqrt{x+y}$,  we first have
\begin{align}
	e^{-c(\lambda t_1^2)^{1/4}} e^{-c(\lambda t_2^2)^{1/4}} \le e^{-c(\lambda t^2)^{1/4}} ,
\end{align}
and by the AM-GM inequality,
\begin{align}
	t_1^{2\beta-1} t_2^{2\beta-1} \le \frac{1}{4^{2\beta -1}} t^{2\beta-1} ,
\end{align}
Also,  $t_1^{2\beta} + t_2^{2\beta} \ge 2^{-2\beta} (t_1 +t_2)^{2\beta}$,  so
\begin{align}
	(1 + a t_1^{2\beta}) (1 + a t_2^{2\beta})
		\ge 1 + a ( t_1^{2\beta} + t_2^{2\beta} ) \ge 1 + 2^{-2\beta} a t^{2\beta} ,
\end{align}
so we have bounds on all the components.
\end{proof}

\begin{proof}[Proof of Lemma~\ref{lemma:hatdotGammatbnd}]
By Lemma~\ref{lemma:tildelambdadecomp} and \ref{lemma:starconvltnbnd},
\begin{align}
	\dot{H}_t^{\star n} (p) \le C^n \frac{t^{2\beta  - 1} }{ 1 + \ba^{(\emptyset) } t^{2\beta} }  e^{-c ( \lambda (p) t^2 )^{1/4}} 
\end{align}
for some $C>0$,
and since $0 \le C_{\rm F} \epsilon_p \lambda - \delta \lambda \lesssim \epsilon_p$, 
\begin{align}
	0 \le (C_{\rm F} \epsilon_p \lambda - \delta \lambda )^n \hat{H}_{t-2dn}^{\star (n+1)}
		\le  (C' \epsilon_p)^{n} \frac{(t-2dn)^{2\beta-1} }{1 + \ba^{(\emptyset)} (t-2dn)^{\beta}} e^{-c ( \lambda (t-2dn)^2 )^{1/4}}
\end{align}
for some $C' >0$ and $t \ge 2dn$.  By setting $\epsilon_p$ sufficiently small,  we have
\begin{align}
	\epsilon_p^{n/2} \lesssim \frac{1}{1 + \ba^{(\emptyset)} (2dn)^{\beta} } e^{-c (\lambda (2dn)^2)^{1/4}} ,
\end{align}
so by another application of Lemma~\ref{lemma:starconvltnbnd},
\begin{align}
	0 \le (C_{\rm F} \epsilon_p \lambda - \delta \lambda )^n \hat{H}_{t-2dn}^{\star (n+1)}
		\le (C'' \epsilon_p)^{n/2}  \frac{t^{2\beta  - 1} }{ 1 + \ba^{(\emptyset) } t^{2\beta} }  e^{-c ( \lambda (p) t^2 )^{1/4}}
\end{align}
for some $C''>0$.
The desired bound on $\hat{\dot{\Gamma}}_t$ follows after summing over $n$.
\end{proof}

\subsection{Torus,  $\eta >0$}

As for $\eta =0$,  \eqref{eq:GammaNLambdadefi} and \eqref{eq:tNdefi},  we use Fourier transform on a torus $\Lambda = \Lambda_N$.
We define
\begin{align}
	\Gamma_N^{\Lambda} (x) &= \frac{1}{|\Lambda|} \sum_{p \in \Lambda^* \backslash \{ 0 \}} e^{i p \cdot x} \int_{L^{N-1}}^{\infty} \hat{\dot{\Gamma}}_t (p)  \rd t 
	,\qquad
	t_N = \int_{L^{N-1}}^{\infty} \hat{\dot{\Gamma}}_t (0)    \rd t
	\label{eq:GammaNLambdadefieta}  
\end{align}
where $\hat{\dot{\Gamma}}$ is the function defined by \eqref{eq:dotHtdefieta} restricted on $p \in \Lambda^*$.

\begin{proof}[Proof of Proposition~\ref{prop:Gammajbounds2}~(i) for $\eta >0$]
This follows from the Fourier inversion formula,  just as for the $\eta =0$ case.
\end{proof}

\subsection{Decay estimates,  $\eta >0$}

We are now ready to prove the bounds on $\dot{\Gamma}_t$.

\begin{proof}[Proof of Proposition~\ref{prop:Gammajbounds}(iii) and Proposition~\ref{prop:Gammajbounds2}(ii),(iii) for $\eta > 0$]
For the upper bounds,  Lemma~\ref{lemma:hatdotGammatbnd} implies
\begin{align}
	\int_{[-\pi,\pi]^d} |p|^k \hat{\dot{\Gamma}_t} (p) \rd p
		&\lesssim \int_{[-\pi,\pi]^d} \frac{t^{2\beta -1}}{1 + \ba^{(\emptyset)} t^{2\beta}} e^{- c ( |p| t )^{1/2}} |p|^k \rd p \nnb
		&\lesssim \frac{t^{-d + 2\beta - 1 - k}}{1 + \ba^{(\emptyset)}  t^{2\beta}} \int_{\R^d} e^{-c' |p|^{1/2}} \rd p \nnb
		&\lesssim t^{-d + 2\beta - 1 - k}   \frac{1}{1 + \ba^{(\emptyset)} t^{2\beta}} 
\end{align}
with change of variable $pt \mapsto p$.
Exactly the same computation,  with just the discrete sum replaced by integral gives the bound on $\nabla^k \dot{\Gamma}^{\Lambda}_t$.

Finally,  for $t_N$,  we proceed exactly as in \eqref{eq:tNest1}--\eqref{eq:tNest3}: Lemma~\ref{lemma:hatdotGammatbnd} implies
\begin{align}
	0 \le \int_0^{L^{N-1}} \hat{\dot{\Gamma}}_t (0) \rd t \lesssim \int_0^{L^{N-1}} t^{2\beta -1} \rd t \lesssim L^{2\beta (N-1)} ,
\end{align}
thus by \eqref{eq:GammaNLambdadefieta},
\begin{align}
	(\ba^{\emptyset})^{-1} - t_N = \int_0^{L^{N-1}} \hat{\dot{\Gamma}}_t (0) \rd t \lesssim L^{2\beta (N-1)} .
\end{align}
That $t_N > 0$ follows from its definition.
\end{proof}

\begin{proof}[Proof of continuity of Lemma~\ref{lemma:FRDcontinuity}]

Covariance matrices $\dot\Gamma_t,  \dot\Gamma_t^{\Lambda_N}$ and $t_N$ are continuous in $\vec{\ba}$ because the integrands defining each object are continuous in $\vec{\ba}$,  due to the Dominated convergence theorem.
Also,  the bounds on covariances are uniform on each interval of $t$ defining $\Gamma_j$ and $\Gamma_N^{\Lambda_N}$,  so these covariance matrices are also continuous in $\vec{\ba}$.
\end{proof}

\appendix
\section{RG flow in $d = d_{c,u}$}
\label{sec:critdRGflow}

In this appendix,  we summarise the results on the critical point in $d= d_{c,u}$.
In this case,  $\ko_{4,\nabla}$ and $ \ko_{1,\nabla}$ are empty.

For $\ko_{2,\nabla}$,  due to the symmetries of the system,  whenever $V \in \cV_{2,\nabla}$,  there exist $\nu_{\nnabla}$ and $\nu_{\Delta}$ such that
\begin{align}
	V_{x} (\varphi) = \nu_{\nnabla} \nabla \varphi_x \cdot \nabla \varphi_x + \nu_{\Delta} \varphi_x \cdot \Delta \varphi_x ,
\end{align}
as in Remark~\ref{remark:Vsymmetries}.
Thus we denote $\vec{\nu}_j = \{ \nu_j^{(\emptyset)},  \nu_j^{(\nnabla)},  \nu_{j}^{(\Delta)} \}$ and $\vec{\ba} = \{ \ba^{(\emptyset)},  \ba^{(\nnabla)},  \ba^{(\Delta)} \}$ to denote the coefficients of the effective potential and the covariance,  respectively.

\begin{proposition} \label{prop:stablemanifold-dcu}
Let $(d,\eta) = (4,0)$,  $\delta >0$ be sufficiently small,  $g \in (0,\delta)$ and $\ba \in [0, \delta )$.  Then there exists $\vec{\nu}_{c,\ba} = ( \nu_{c,\ba}^{(\emptyset)},  \nu_{c,\ba}^{(\nnabla)} ,  \nu_{c,\ba}^{(\Delta)} )$ such that there exists an infinite bulk RG flow with 
\begin{align}
	& \ba^{(\emptyset)} = \ba, \quad \ba^{(\nnabla)} = - \nu_{c,\ba}^{(\nnabla)}, \quad \ba^{(\Delta)} = - \nu_{c,\ba}^{(\Delta)} , \label{eq:badefi_dcu} \\
	& \vec{\nu}_0 = \vec{\nu}_{c,\ba} , \quad g_0^{(\emptyset)} = g , \quad K_0 = 0 . 
	\label{eq:dcuIC}
\end{align}
Moreover,  $\vec{\nu}_{c,\ba}$ is continuous in $(g,\ba)$ and differentiable in $g \in (0,\delta)$ with uniformly bounded derivative.
\end{proposition}
\begin{proof}
The statement is \cite[Theorem~3.6]{MR3269689},  but just formulated slightly differently. 
(The RG map used in the reference is not the same as that of \cite{FSmap},  but it nevertheless satisfies Theorem~\ref{thm:contrlldRG}--\ref{thm:infvolRGmap} the same,  which is enough for the proof. )
For the notations,  we use $(\nu_{c, \ba}^{(\emptyset)},  \nu_{c,\ba}^{(\Delta)})$ instead of $(\mu_0^c ,  z_0^c)$ in the reference.  
We can just take $\nu_{c,\ba}^{(\nnabla)} = 0$.
In the reference,  $\ba^{(\nnabla)}$ and $\ba^{(\Delta)}$ are not present.
However,  the coefficient of the covariance,  $\vec{\ba}$,  in \eqref{eq:badefi_dcu} is equivalent to the reparametrisation of the RG coefficients as in \cite[(2.1)--(2.5)]{MR3269689}.  

The essential idea of the proof is contained in \cite{MR3317791}.
\end{proof}

\begin{remark}
Even though we are making reference to \cite{MR3269689} for the construction of the critical point,  estimate on each $K_j$ is improved,  so the results of the paper is not implied by \cite{MR3269689} even for $d = d_{c,u}$.
\end{remark}

Recalling Definition~\ref{def:RGflow},  this in particular implies the RG coordinates defined by \eqref{eq:dcuIC} satisfy
\begin{align}
	g^{(\emptyset)}_j \in [\tilde{g}_j / 2 ,  2 \tilde{g}_j] ,  \qquad | \nu_j^{(\emptyset)} | \le C_{\cD} L^{-2j} \tilde{g}_j, \quad \norm{K_{j,\bulk}}_{\cW_j} \le C_{\rg} \tilde{\chi}_j \tilde{g}_j^3 .
	\label{eq:couplingconstantsbounds}
\end{align}
Also,  the uniform differentiability in $g$ implies
\begin{align}
	|\vec{\nu}_{c,\ba} | \le O(g) .
\end{align}
We also have a similar result for the observable flow.

\begin{proposition}  \label{prop:obsstablemanifold-dcu}
Under the assumptions of Proposition~\ref{prop:stablemanifold-dcu},  also take $\vec{\ba}$, $\vec{\nu}_0$,  $g_0^{(\emptyset)}$ and $K_0$ as in \eqref{eq:badefi_dcu} and \eqref{eq:dcuIC}.
Also,  let $|\lambda_{\o,0}^{(\emptyset)}|$,  $|\lambda_{\x,0}^{(\emptyset)}| \le 1$.

Then there exists an infinite RG flow with these initial conditions and for both $\hash \in \{ \o, \x \}$,  
there exists $\lambda_{\hash, \infty}$ such that
\begin{align}
	\big| \lambda_{\hash, j}^{(\emptyset)} - \lambda_{\hash, \infty} \big|
		\le O_L  ( \tilde{\chi}_j \tilde{g}_j)  \qquad \text{for all} \quad j \ge 0 .
\end{align}
\end{proposition}

\begin{proof}[Proof sketch]
The proof is essentially given in \cite[Section~5]{MR3459163},  but we summarise the proof here again to removed unclarity related to the difference of the observable projection of the RG map.

Since $(\pi_\o + \pi_\x)V_{j,x} (\varphi) = \varphi_x \sum_{\hash \in\{ \o,\x\}} \sigma_{\hash} \lambda_{\hash,j}^{(\emptyset)} \one_{x = \hash}$,
we only need to consider the flow of $\lambda_{\hash, j}^{(\emptyset)}$,  which we simply abbreviate by $\lambda_{\hash, j}$ when $d = d_{c,u}$. 
Likewise,  we just write $(\nu_j, g_j)$ for $(\nu_j^{(\emptyset)},g_j^{(\emptyset)})$.

Due to \cite[(5.21)]{MR3459163} (also see \cite[(3.34)]{BBS3} for a more explicit expression,  when the flow is written in the case $n=0$ for the weakly self-avoiding walk), 
\begin{align}
	\lambda_{\hash, j+1} = \lambda_{\hash, j}  ( 1 - \delta (\nu_j w_j^{(1)}) ) +  O_L ( \tilde{\chi}_j^{3/2} \tilde{g}_j^2 )  \label{eq:lambdajflow_dcu}
\end{align}
where
\begin{align}
	& \delta (\nu_j w_j^{(1)}) := (\nu_j + \eta'_j g_j) \Gamma_{j+1}^{(1)} + \eta'_j g_j w_j^{(1)} , \\
	& \Gamma_{j+1}^{(1)} = \sum_{x \in \Z^d} \Gamma_{j+1} (x) , \quad
		w_j^{(1)} = \sum_{x \in \Z^d} w_{j} (x) ,   \quad \eta'_j = (n+2) \Gamma_{j+1} (0)  .	
\end{align}
A crucial observation is that the sum of $\delta (\nu_j w_j^{(1)})$ is bounded by an absolute constant. 
Indeed,  if we denote $\nu_{j+1, \pt} = \nu_j + \eta'_j g_j$,  then we can write
\begin{align}
	\delta (\nu_j w_j^{(1)}) = \nu_{j+1, \pt} w_{j+1}^{(1)} - \nu_j w_j^{(1)} , 
\end{align}
so for any $M > m \ge 0$,
\begin{align}
	\sum_{j=0}^{M-1} \delta (\nu_j w_j^{(1)}) = \nu_{M} w_M^{(1)} - \nu_m w_m^{(1)} + \sum_{j=m}^{M-1} ( \nu_{j+1, \pt} - \nu_{j+1} ) w_{j+1}^{(1)} ,
\end{align}
and by \eqref{eq:controlledRG22},
\begin{align}
	| \nu_{j+1, \pt} - \nu_{j+1} | \le O_L (1) \tilde{\chi}_j^{3/2} \tilde{g}_j^3 L^{-2j} .
\end{align}
Due to Corollary~\ref{cor:Coneinftynrm},  we have $w_j^{(1)} \lesssim \tilde{\chi}_j L^{2j}$.
Due to Lemma~\ref{lemma:gNasymp},  $\sum_{j=m}^\infty \tilde{\chi}_j^{3/2} \tilde{g}_j^3 \lesssim \tilde{\chi}_m^{3/2} \tilde{g}_m^2$.
Also,  by the choice of the initial condition and \eqref{eq:couplingconstantsbounds},  we have $|\nu_{j,\pt} | \lesssim L^{-2j} \tilde{g}_j$,  so 
\begin{align}
	\Big| \sum_{j=m}^{M-1} \delta (\nu_j w_j^{(1)})  \Big| \lesssim \tilde{\chi}_m \tilde{g}_m
\end{align}

Now,  by the summability,  $\tilde{\lambda}_{\hash, j} = \prod_{k < j} (1 - \delta (\nu_k w_k^{(1)})) \lambda_{\hash, 0}$ has some limit $\tilde{\lambda}_{\hash, \infty}$ as $j\rightarrow \infty$ and
\begin{align}
	\tilde{\lambda}_{\hash, j} = \tilde{\lambda}_{\hash,  \infty} + O_L (\tilde{\chi}_j \tilde{g}_j) ,
\end{align}
so we see that $\tilde{\lambda}_{\hash, j}$ has some limit $\tilde{\lambda}_{\hash, \infty}$ as $j\rightarrow \infty$. 
This also gives a limit of $\lambda_{\hash, j}$,  and the system $(\lambda_{\hash, j})_{j\ge 0}$ stays uniformly bounded.  
In fact,  as one can see from the proof,  we can make $\sup_j |\lambda_{\hash, j} - \lambda_{\hash, 0}|$ arbitrarily small by taking sufficiently small $g_0$,  so the flow of $\lambda_{\hash, j}$ stays inside the RG domain.
\end{proof}

\section{Fourier transformation}
\label{sec:fa}

Since Fourier analysis is used excessively in the appendices and the main text,  we spare a section to fix notations. 
For $\Lambda = \Lambda_N$,  $g \in (\R)^{\Lambda}$,  $h \in (\R)^{\Z^d}$ and $f \in \cS (\T^d)$,  we let $\Lambda^* = 2 \pi L^{-N} \Lambda_N$,  $(\Z^d)^* = [-\pi, \pi)^d$ and $(\T^d)^* = 2 \pi \Z^d$.  Then for $p \in \Lambda^*$,  $r \in (\Z^d)^*$ and $q \in (\T^d)^*$,
\begin{align}
	\hat{g} (p) = \sum_{x \in \Lambda} e^{-i x \cdot p} g (x)	,
	\quad \hat{h} (r) = \sum_{y \in \Z^d} e^{-i y\cdot r} h(y) ,
	\quad \hat{f} (q) = \int_{\T^d} e^{-i y \cdot q} f (y) \rd y .
\end{align}
The inverse transformations are
\begin{align}
	g (x) = \frac{1}{|\Lambda|} \sum_{p \in \Lambda^*} e^{i x \cdot p} \hat{g}(p) ,
	\quad h(y) = \frac{1}{(2\pi)^d} \int_{[-\pi,\pi)^d} e^{i y \cdot r} \hat{h} (r) \rd r ,
	\quad f (y) =  \sum_{q \in (\T^d)^*} e^{i y\cdot q} \hat{f} (q) . 
\end{align}
If we make the $N$-dependence of $\Lambda_N$ explicit,  we observe that $L^N \Lambda^*_N$ converges to the lattice $(\T^d)^*$ and the same holds for the Fourier symbols.  To make this meaning precise,  suppose that $f_N \in (\R)^{\Lambda_N}$ is given by $f_N (x) = L^{-dN} f (\hat{i} (x))$ for some $f \in \cS (\T^d  ; \R)$.  Also,  suppose that there are translation invariant covariance matrices $Q_N$ on each $\Lambda_N$ with Fourier symbol $\hat{Q}_N$ and let $Q$ be a translation invariant covariance operator on $\T^d$ with Fourier symbol $\hat{Q}$.  
Then standard convergence results follow.

\begin{lemma} \label{lemma:facnv2}
For each $q \in (\T^d)^*$,  we have $\lim_{N\rightarrow \infty} \hat{f}_N (L^{-N} q) = \hat{f} (q)$.
\end{lemma}
\begin{proof}
By definition, for any $q \in (\T^d)^*$,
\begin{align}
	\hat{f}_N ( L^{-N} q) = L^{-dN} \sum_{x \in \Lambda_N} e^{-i x \cdot q / L^N} f (\hat{i} (x)) .
\end{align}
This converges to the Riemann integral $\int_{\T^d} e^{-ix \cdot q} f(x) \rd x = \hat{f} (q)$.  
\end{proof}

\begin{lemma} \label{lemma:facnv1}
Assume that, for each $q \in (\T^d)^*$,  $\lim_{N\rightarrow \infty} \hat{Q}_N (L^{-N} q ) = \hat{Q} (q)$ and $\sup_{N}\norm{\hat{Q}_N}_{\ell^{\infty}} < \infty$.
Then
\begin{align}
	\lim_{N\rightarrow \infty} L^{dN} (f_N,  Q_N f_N) \rightarrow (f,Q f)	.
\end{align}
\end{lemma}
\begin{proof}
By the Plancherel theorem,  
\begin{align}
	L^{dN} (f_N, Q_N f_N) &=  \sum_{p \in \Lambda_N^*} \hat{Q}_N (p) |\hat{f}_N (p)|^2 
		= \sum_{q \in 2\pi \Lambda_N} \hat{Q}_N (L^{-N} q) \left| \hat{f}_N (L^{-N} q) \right|^2 .
\end{align}
But since $\lim_{N\rightarrow \infty} \hat{f}_N (L^{-N} q) = \hat{f} (q)$ by Lemma~\ref{lemma:facnv2},  the Dominated convergence theorem implies
\begin{align}
	\lim_{N\rightarrow \infty}  L^{dN} (f_N, Q_N f_N) = \sum_{q \in (\T^d)^*} \hat{Q} (q) |\hat{f} (q)|^2 \rd q = (f, Qf) .
\end{align}
\end{proof}

\subsection{Positivity of modified Laplacian}

Let $\lambda$ be the Fourier symbol of $-\Delta$ and
for $\km_1 \in \kA_1$, let $\lambda_{\km_1}$ be the Fourier symbol of the quadratic form
$\varphi \mapsto \sum_{x \in \Lambda} S^{(\km_1)}_x (\varphi)$---since $S_x^{(\km_1)} (\varphi)$ is the symmetrisation of 
\begin{align}
	M_x^{(\km)} (\varphi) = \nabla^{\mu_1^1  \cdots \mu^1_{i_1} } \varphi^{(\alpha_1)}_x  \nabla^{\mu_1^2  \cdots \mu^2_{i_2} } \varphi^{(\alpha_2)}_x ,
\end{align}
(when $\mu_i^k$ and $\alpha_i$ are determined by $\km$,  recall Section~\ref{sec:modotcov}) it is indeed a quadratic form. 

Recall that $\L_\eta^{(\vec{\ba})}$ is defined by \eqref{eq:Lapkadefi}.
The Fourier symbol of $\L_\eta^{(\vec{\ba})}$ is given by
\begin{align}
	\lambda^{(\vec\ba)} := \lambda^{1-\eta} + \ba^{(\emptyset)} + \sum_{\km_1 \in \kA_1}  \ba^{(\km_1)} \lambda_{\km_1} .
	\label{eq:LapbaFrr}
\end{align}
To claim that its inverse $C^{(\vec{\ba})}$ is a covariance matrix,  we need the following lemma.

\begin{lemma}
\label{lemma:Lapptv-res}
When $\vec{\ba} \in \HB_{\epsilon_p}$ for sufficiently small $\epsilon_p >0$,  
(i) $\L^{(\vec{\ba})}_\eta \ge 0$ and
(ii) $\lambda^{(\vec\ba)} (p) - \ba^{(\emptyset)} \asymp \lambda^{1-\eta/2} (p)  \asymp |p|^{2-\eta} \wedge 1$.
\end{lemma}
\begin{proof}
Since $\lambda (p) = 2 \sum_{i=1}^d (1- \cos (p_i))$,  we have
$c \big( |p|^2 \wedge 1 \big) \le \lambda(p) \le |p|^2$
for some $c > 0$,  and since the number of derivatives $q(\km_1) = i_1 + i_2 \ge 2$ for $\km_1 \in \kA_1$,  there exists $C$ such that $| \lambda_{\km_1} (p) | \le C (|p|^2 \wedge 1) \le 4 \pi^2 C (|p|^2 \wedge 1)$.  These give (i) and (ii).
\end{proof}

\section{Green's function asymptotic}
\label{sec:covcomp}

On $\Lambda = \Lambda_N$ or $\Z^d$,  with the notation as above and $\bar{\ba}_{\Delta} = -\ba^{(\Delta)} + \ba^{(\nnabla)}$ as in Remark~\ref{remark:kAsymmetries},
\begin{align}
	\hat{C}^{(\vec{\ba})}_\Lambda := \frac{1}{\lambda^{(\vec{\ba}) }} = \frac{1}{\ba^{(\emptyset)} + \lambda^{1-\eta/2} + \bar{\ba}_\Delta \lambda + \sum_{\km_1 \in \kA_1}^{q(\km_1) \ge 4} \ba^{(\km_1)} \lambda_{\km_1}}
\end{align}
which is finite for $\vec{\ba} \in \HB_{\epsilon_p}$ such that $\ba^{(\emptyset)} > 0$,  due to Lemma~\ref{lemma:Lapptv-res}.

\subsection{Asymptotic of the Green's function}
\label{sec:asympGreenfnc}

We prove an estimate on $C^{(\vec{\ba})}$ based on a fairly standard estimate on the fractional Laplacian.

\begin{lemma} \cite[Proposition~A.1]{MS22asymptotic} \label{lemma:fracLapGreen}
Let $\beta \in (0, d/2)$,  $c_{d,\beta} = \frac{\Gamma (d/2-\beta)}{2^{2\beta} \pi^{d/2} \Gamma (\beta)}$ for $\Gamma (\cdot)$ the Gamma function.
Then for $x \neq 0$,
\begin{align}
	\frac{1}{(2\pi)^d} \int_{\R^d} \frac{e^{i p \cdot x}}{|p|^{2\beta}} \rd p = c_{d,\beta} |x|^{-d+ 2\beta} . \label{eq:fracLapGreen2}
\end{align}
\end{lemma}

In the proof of the next lemma,  we write $\delta \lambda$ for the Fourier transform of $\delta \L_\eta^{(\vec{\ba})}$,  where
\begin{align}
	\delta \L_\eta^{(\vec{\ba})} = \begin{cases}
		\L_0^{(\vec{\ba})} + (1+ \bar{\ba}_{\Delta}) \Delta - \ba^{(\emptyset)} & (\eta=0) ,  \\
		\L_{\eta}^{(\vec{\ba})} - (-\Delta)^{1-\eta/2} - \ba^{(\emptyset)} &  (\eta >0) .
	\end{cases}
\end{align}

\begin{lemma}
\label{lemma:Greenasymp}

Let $\eta \in [0,1)$ and $\vec{\ba} \in \HB_{\epsilon_p}$ be such that $\ba^{(\emptyset)} = 0$. 
then
\begin{align}
	C^{(\vec{\ba})}_{\Z^d} (x) = \frac{\gamma (d,\eta)}{|x|^{d-2+\eta}} \times \begin{cases}
		(1+ \bar{\ba}_{\Delta})^{-1} (1 + O(|x|^{-1})) & (\eta =0) \\
		1 + O(|x|^{-\eta}) & (\eta > 0) 
	\end{cases}
\end{align}
where
\begin{align}
	\gamma(d,\eta) = 2^{-2+\eta} \pi^{-d/2}  \frac{\Gamma( (d -2 + \eta)/2 )}{\Gamma ( (2-\eta)/2 )}	.
\end{align}
\end{lemma}
\begin{proof}
We take $\gamma (d,\eta) = c_{d, 1-\eta/2}$ for $c_{d,\beta}$ as in Lemma~\ref{lemma:fracLapGreen}.

When $\eta =0$,  we can rescale all the coefficients of $\L_\eta^{(\vec{\ba})}$,  so we can assume $\bar{\ba}_{\Delta} = 0$ without losing generality.
Since $\lambda$ and $\delta \lambda$ are both analytic functions,  there exists an analytic function $f(p)$ such that
\begin{align}
	\frac{1}{\lambda^{(\vec{\ba})}} = \frac{1}{|p|^2} + f(p) , 
\end{align}
for $p \in \R^d$.  Also,  let $\chi : \R^d \rightarrow [0,1]$ be some smooth function such that $\chi (p) = 1$ for $|p| \le 1$ and $\chi(p) = 0$ for $|p| \ge 2$. 
Then we have decomposition
\begin{align}
	& C^{(\vec{\ba})}_{\Z^d} (x) 
		= \frac{1}{(2\pi)^d} \int_{[-\pi, \pi]^d}	\frac{e^{ip\cdot x}}{\lambda^{(\vec{\ba})}} \rd p \nnb
		&\qquad
		= \frac{1}{(2\pi)^d} \Big(
			\int_{\R^d} e^{ip \cdot x} \Big(\frac{1}{|p|^2} + f \chi - \frac{(1-\chi)}{|p|^2} \Big) \rd p + \int_{[-\pi, \pi]^d} (1- \chi) \frac{e^{i p \cdot x}}{\lambda^{(\vec{\ba})}} \rd p 
			\Big) .
\end{align}
$f\chi$,  $(1-\chi)/ |p|^2$ and $(1-\chi) / \lambda^{(\vec{\ba})}$ are all smooth functions of $p$,  so it decays faster than $|x|^{d-2}$,  and Lemma~\ref{lemma:fracLapGreen} gives the main contribution $\gamma |x|^{-d+2}$.

When $\eta >0$,  let $\beta = 1-\eta/2$ and we use the fact that $\lambda^{(\vec{\ba})} = |p|^{2\beta} + c|p|^2 + g(p)$ for some $c \in \R$ where $g(p)$ is an even analytic function such that $g(p) = O(|p|^4)$.  If we take sufficiently small $r$,  we have $|c| |p|^2 + |g(p)| \le |p|^{2\beta} /2$ for $|p| < r$,  so we restrict the domain by taking $\chi_r (p) = \chi (p /r)$,  and we can expand
\begin{align}
	\frac{1}{\lambda^{(\vec{\ba})} (p) } 
		&= 
		\chi_r (p) \Big( \frac{1}{|p|^{2\beta}} - \frac{c|p|^2}{|p|^{4\beta}} + \frac{h(p)}{|p|^{4\beta}} \Big)+ \frac{1-\chi_r}{\lambda^{(\vec{\ba})}} 
\end{align}
where $h(p)$ is some function smooth away from $0$ and satisfies $|D^n h(p)| \le O_n (|p|^{4-n})$ for each $n\ge 0$.
The singular part $1/|p|^{2\beta}$ gives contribution $\gamma |x|^{-d+\eta}$ due to Lemma~\ref{lemma:fracLapGreen} and the contribution of $|p|^2 / |p|^{4\beta}$ is bounded by $O(|x|^{-(d-2+2\eta)})$.  Also,  since
\begin{align}
	\Big| D^d \Big( \frac{\chi_r (p)h(p)}{|p|^{4\beta}} \Big) \Big|
		\le O \Big( \frac{|p|^{4-d}}{|p|^{4\beta}} \Big)
\end{align}
is integrable in $p$,  we may bound
\begin{align}
	\Big| \int_{\R^d} \chi_r (p) \frac{e^{i p \cdot x}}{|p|^{4\beta}} h(p) d p  \Big| \lesssim |x|^{-d} \le O \Big( \frac{1}{|x|^{2\beta}}  \Big) \frac{1}{|x|^{d-2 \beta}}  .
\end{align}
All the remaining terms are smooth in $p$,  so only create contributions smaller than $|x|^{-d + 2 \beta - 1}$ as $|x| \rightarrow \infty$.
\end{proof}

\begin{lemma}
\label{lemma:GreenminuswN}
If we consider $x$ as an element of both $\Z^d$ and $\Lambda$,
\begin{align}
	\big| C^{(\vec{\ba})}_{\Z^d} (x) - w_N (x) \big|
		\le O_L ( L^{-(d-2+\eta)N} ) .
\end{align}
\end{lemma}
\begin{proof}
Using the decomposition
\begin{align}
	C^{(\vec{\ba})}_{\Z^d} = \sum_{j=1}^\infty \Gamma_j , \qquad w_N = \sum_{j=1}^{N-1} \Gamma_j + \Gamma_{N}^{\Lambda_N} ,
\end{align}
for the `same' $\Gamma_j$ on the left and the right,  we simply have
\begin{align}
	\big| C^{(\vec{\ba})}_{\Z^d}(x) - w_N (x) \big| \le |\Gamma_{N}^{\Lambda_N} (\x)| +\sum_{j=N}^{\infty} |\Gamma_j (\x)|  .
\end{align}
By Proposition~\ref{prop:theFRD},  this is bounded by a constant multiple of $L^{-(d-2+\eta)N}$.
\end{proof}

Infinitesimal decomposition can also be used to prove crude bounds on the Green's function. 

\begin{corollary} \label{cor:Coneinftynrm}
Under the assumptions of Proposition~\ref{prop:Gammajbounds2},  
if $\Lambda = \Lambda_N$ and $j \le N$,
\begin{align}
	&\sum_{y \in \Lambda} \big| C^{(\vec{\ba})} (y) \big|  \lesssim  L^{(2-\eta) N} + ( \ba^{(\emptyset)} )^{-1} ,   \qquad
	\sum_{y \in \Lambda} |w_j (y)| \lesssim \tilde{\chi}_j L^{(2-\eta) j}
\end{align}
\end{corollary}
\begin{proof}
By \eqref{eq:Gammajbounds1infi},  and since $\dot{\Gamma}_t$ and $\dot{\Gamma}_t^{\Lambda}$ have range $\le t \wedge L^N$, 
\begin{align}
	\sum_{y \in \Lambda} |\dot{\Gamma}_t (y)|,  \; \sum_{y \in \Lambda} |\dot{\Gamma}^{\Lambda}_t (y)| \lesssim 
		\frac{t^{-d+1 - \eta} \wedge 1}{1 + \ba^{(\emptyset)} t^{2\beta}} (t \wedge L^{N})^d ,
\end{align} 
we have 
\begin{align}
	& \sum_{y \in \Lambda} |w_j (y)| \lesssim \int_0^{L^j} \frac{t^{1-\eta}}{1 + \ba^{(\emptyset)} t^{2\beta}} dt \lesssim \tilde{\chi}_j L^{(2-\eta) j} \qquad (j < N) ,  \\
	& \sum_{y \in \Lambda} |w_N (y)| \lesssim \int_{0}^{\infty} \frac{t^{-d+1-\eta}}{1 + \ba^{(\emptyset)} t^{2\beta}} (t \wedge L^{N})^d \rd t \lesssim \tilde{\chi}_N L^{(2-\eta)N} .
\end{align}
Then
\begin{align}
	\sum_{y \in \Lambda} | C^{(\vec{\ba})}  (y) |  \le \sum_{y \in \Lambda} |w_N (y)| + t_N \lesssim L^{(2-\eta) N} + ( \ba^{(\emptyset)} )^{-1}
\end{align}
\end{proof}

\subsection{Proof of Lemma~\ref{lemma:fNtNlmt}}

\begin{proof}[Proof of Lemma~\ref{lemma:fNtNlmt}]
For (i),  by Lemma~\ref{lemma:facnv1}, it is sufficient to show for each $\km_1 \in \kA_1$
\begin{align}
	\lim_{N\rightarrow \infty} \lambda_{\km_1} (L^{-N} q) = 0 .
\end{align}
Indeed, they hold because $|\lambda_{\km_1} (p)| \le C |p|^2$ for some $C >0$. 

For (ii),  by applying the Plancherel identity twice,  
\begin{align}
	\norm{\nabla^n C^{(\vec{\ba})} \f_N}_{\ell^{2} }^2 
		\lesssim  \frac{1}{|\Lambda|} \big\| |p|^{\frac{n}{2}} \widehat{C^{(\vec{\ba})} \f_N } \big\|_{\ell^2}^2 
		\lesssim  \frac{1}{ (\ba^{(\emptyset)} )^{2} |\Lambda|} \norm{|p|^{\frac{n}{2}} \hat{\f}_N}_{\ell^2}^2  
		\lesssim  \frac{1}{(\ba^{(\emptyset)} )^{2}} \norm{\nabla^n \f_N}_{\ell^2}^2  
\end{align}
with $\lesssim$ depending on $n$, so by definition of $\f_N$, 
\begin{align}
	\norm{\nabla^n C^{(\vec{\ba})} \f_N}_{\ell^{2} } \lesssim L^{ -n N} (\ba^{(\emptyset)} )^{-1} \norm{\nabla^n f}_{\ell^{\infty}} .	\label{eq:CfNell2}
\end{align}
On the other hand, 
\begin{align}
	| \nabla^n C^{(\vec{\ba})} \f_N (x) | \le \sum_{y \in \Lambda} \big| C^{(\vec{\ba})} (x,y) \big| \,   |\nabla^n \f_N (y) |  \le \sum_{y \in \Lambda} \big| C^{(\vec{\ba})} (x,y) \big| \norm{\nabla^n \f_N}_{\ell^\infty}  
\end{align}
and by Corollary~\ref{cor:Coneinftynrm},  this is bounded by 
\begin{align}
	\lesssim ( L^{(2-\eta) N} + (\ba^{(\emptyset)} )^{-1} ) \norm{\nabla^n \f_N}_{\ell^\infty}   \le L^{-nN} L^{-\frac{d}{2} N} ( L^{(2-\eta) N} + (\ba^{(\emptyset)} )^{-1} ) \norm{\nabla^n f}_{L^\infty} ,
\end{align}
so we have shown
\begin{align}
	\norm{\nabla^n C^{(\vec{\ba})} \f_N}_{\ell^{\infty} } \lesssim L^{-n N} L^{-\frac{d}{2} N} ( L^{(2-\eta) N} +(\ba^{(\emptyset)} )^{-1} ) \norm{\nabla^n f}_{L^{\infty}} . 	\label{eq:CfNellinfty}
\end{align}
Now we can use the H\"older inequality to interpolate \eqref{eq:CfNell2} and \eqref{eq:CfNellinfty} and obtain the desired bound.
\end{proof}

\subsection{Proof of Lemma~\ref{lemma:gNtNlmt}}

The proof is almost identical as above,  but we rather use that $w_N = \sum_{j \le N} \Gamma_j$ satisfies
\begin{align}
	\sum_{y \in \Lambda} | w_N (x,y) | \lesssim L^{(2-\eta)N} , \qquad
	\norm{\hat{w}_N (p)}_{\ell^{\infty}} \lesssim L^{(2-\eta)N} .  \label{eq:wNnorms}
\end{align}
The first inequality is in Corollary~\ref{cor:Coneinftynrm} and the second inequality follows from the first.

\begin{proof}[Proof of Lemma~\ref{lemma:gNtNlmt}]

For part (i),  we use \eqref{eq:wNnorms} and the Plancherel identity to obtain
\begin{align}
	|(\g_N, w_N \g_N) | \lesssim \norm{\hat{w}_N (p)}_{\ell^{\infty}} \norm{\g_N}^2_{\ell^2} \le b_N^{-2} L^{-(d-2+\eta)N} \norm{\f_N }_{\ell^2}^2 .
\end{align}
But since $\norm{\f_N (x)}_{\ell^2}^2 \rightarrow \norm{f}_{\ell^2}^2$ as $N\rightarrow \infty$,  we see $\lim_{N\rightarrow \infty} |(\g_N, w_N \g_N) | = 0$.

For part (ii),  we proceed identically as in the proof of Lemma~\ref{lemma:fNtNlmt},  but just use \eqref{eq:wNnorms} instead. 
\end{proof}

\subsection{Proof of Lemma~\ref{lemma:hNtNlmt}}

\begin{proof}[Proof of Lemma~\ref{lemma:hNtNlmt}(i)]
For part (i),  since $\sh_N = L^{\frac{d-2+\eta}{2} N} (f_N - \Phi_N (f_N))$, 
\begin{align}
	( \sh_N,  w_N \sh_N ) &= L^{(d-2+\eta)N} \left( f_N - \Phi_N (f_N), w_N \big(f_N - \Phi_N (f_N) \big) \right)  \nnb
		&= \frac{L^{(d-2+\eta)N}}{|\Lambda| } \sum_{p \in \Lambda^* \backslash \{ 0 \}} \frac{| \hat{f}_N (p) |^2 }{ \lambda^{1-\eta/2} (p) + \bar{\ba}_\Delta \lambda(p) + \sum_{\km_1} \ba_{\km_1} \lambda_{\km_1}(p)}  .
\end{align}
By reparametrising $ q= L^N p$, 
\begin{align}
	= \sum_{q \in 2\pi \Lambda \backslash \{ 0 \}} \frac{L^{-2(1-\eta/2)N} | \hat{f}_N ( L^{-N} q ) |^2}{ \lambda^{1-\eta/2} ( L^{-N} q ) + \bar{\ba}_\Delta \lambda ( L^{-N} q ) + \sum_{q ( \km_1)\ge 4}  \lambda_{\km_1} (L^{-N} q) \ba^{(\km_1)} }  .
\end{align}
Lemma~\ref{lemma:facnv2} says $\lim_{N\rightarrow \infty} \hat{f}_N ( L^{-N} q ) = \hat{f} (q)$ and we see from the explicit formula on $\lambda$ that
$\lim_{N\rightarrow \infty} L^{2N} \lambda( L^{-N} q ) \rightarrow |q|^2$.
Also,  since $|\lambda_{\km_1} (p)| \le C( |p|^4 \wedge 1 )$ whenever $q(\km_1) \ge 4$,  we have $L^{2N} \lambda_{\km_1} (L^{-N} q) \rightarrow 0$,  so 
\begin{align}
	\lim_{N\rightarrow\infty} \sum_{q \in 2\pi \Lambda \backslash \{ 0 \}} (\cdots)
		&= \sum_{q \in (\T^d)^* \backslash \{ 0 \}} \frac{|\hat{f} (q)|^2}{|q|^{2-\eta}} \times 	
		\begin{cases}
			(1+ \bar{\ba}_\Delta )^{-1} & (\eta = 0) \\
			1 & (\eta  > 0)		
		\end{cases} \nnb
		&= \big( f - \Phi(f) ,  (-\Delta)^{-1} (f  - \Phi(f) ) \big) \times 	
		\begin{cases}
			(1+ \bar{\ba}_\Delta )^{-1} & (\eta = 0) \\
			1 & (\eta  > 0)		
		\end{cases} .
\end{align}
Part (ii) follows from Lemma~\ref{lemma:gNtNlmt}(ii) because $\sh_N = b_N L^{\frac{d-2+\eta}{2} N} ( \g_N - \Phi_N (\g_N) )$. 
\end{proof}

\section*{Acknowledgements}

The author gratefully acknowledges the support and hospitality of the Department of Mathematics at Seoul National University during the completion of a part of this work. 
The author also thanks Gordon Slade,  Yucheng Liu and Taegyun Kim for their useful comments on this paper.
The author is supported by Basic Science Research Program through the National Research Foundation of Korea funded by the Ministry of Science and ICT (RS2025-00518980).


\end{document}